\documentclass[reqno]{amsart} 
\usepackage{amsmath,amsthm,amssymb,amsfonts}
\usepackage{mathrsfs}
\usepackage{longtable}
\numberwithin{equation}{section}
\usepackage{color}
\usepackage{enumerate}
\usepackage{multirow} 
\newtheoremstyle{mystyle}
{}
{}
{\normalfont}
{ }
{\bfseries}
{}
{10pt}
{ }
\theoremstyle{mystyle}
\newtheorem{theorem}{Theorem}
\newtheorem{proposition}{Proposition}
\newtheorem{lemma}{Lemma}
\newtheorem{remark}{Remark}

\usepackage{mathtools}
\usepackage[pdftex]{graphicx}

\def\Diag{\mathop{\rm Diag}\nolimits}
\def\tr{\mathop{\rm tr}\nolimits}
\def\vec{\mathop{\rm vec}\nolimits}
\def\vech{\mathop{\rm vech}\nolimits}
\def\det{\mathop{\rm det}\nolimits}
\def\rank{\mathop{\rm rank}\nolimits}

\allowdisplaybreaks[4]
\usepackage[foot]{amsaddr}
\usepackage[top=3cm, bottom=3cm, left=2.5cm, right=2.5cm]{geometry}

\usepackage[hang,small,bf]{caption}
\usepackage[subrefformat=parens]{subcaption}
\usepackage{comment}
\large
\title[SEM for diffusion processes with jumps]{Statistical inference in SEM for diffusion processes with jumps based on high-frequency data}
\author[S. Kusano]{Shogo Kusano $^{1}$}
\author[M. Uchida]{Masayuki Uchida $^{2,3,4}$}
\address{$^{1}$ Faculty of Advanced Science and Technology, Kumamoto University, Kumamoto, Japan}
\address{$^{2}$ Graduate School of Engineering Science, The University of Osaka, Toyonaka, Japan}
\address{$^{3}$ Center for Mathematical Modeling and Data Science (MMDS), The University of Osaka, Toyonaka, Japan}
\address{$^{4}$ CREST, Japan Science and Technology Agency, Tokyo, Japan}
\begin{document}
\begin{abstract}
\fontsize{8pt}{10pt}\selectfont
We study structural equation modeling (SEM) for diffusion processes with jumps. Based on high-frequency data, we consider the parameter estimation 
and the goodness-of-fit test in the SEM. Using a threshold method, 
we propose the quasi-likelihood of the SEM and prove that the quasi-maximum likelihood estimator has consistency and asymptotic normality. To examine whether a specified parametric model is correct or not, we also construct the quasi-likelihood ratio test statistics and investigate the asymptotic properties. Furthermore, numerical simulations are conducted.
\end{abstract}
\maketitle

\section{Introduction}
\fontsize{10pt}{16pt}\selectfont
\setlength{\abovedisplayskip}{8pt}
\setlength{\belowdisplayskip}{8pt}
We consider structural equation modeling (SEM) for diffusion processes with jumps. Let $(\Omega, \mathcal{F},(\mathcal{F}_t)_{t\geq 0}, {\bf{P}})$ be a stochastic basis. The $p_1$-dimensional observable process $\{X_{1,t}\}_{t\geq 0}$ is defined by the following factor model:
\begin{align}
    X_{1,t}&={\bf{\Lambda}}_{1}\xi_t+\delta_t, \label{X1}
\end{align}
where 
$\{\xi_t\}_{t\geq 0}$ and $\{\delta_t\}_{t\geq 0}$ are $k_1$ and $p_1$-dimensional c{\`a}dl{\`a}g $(\mathcal{F}_t)$-adapted latent processes on the stochastic basis, $k_1\leq p_1$, and ${\bf{\Lambda}}_{1}\in\mathbb{R}^{p_1\times k_1}$ is a constant loading matrix. $\{\xi_t\}_{t\geq 0}$ and $\{\delta_t\}_{t\geq 0}$ satisfy the following stochastic differential equations:
\begin{align}
    d\xi_{t}&=a_{1}(\xi_{t-})dt+{\bf{S}}_{1}d W_{1,t}+\int_{E_1}c_1(\xi_{t-},z_1)p_1(dt,dz_1),\quad \xi_0=x_{1,0}, \label{xi}\\
    d\delta_{t}&=a_{2}(\delta_{t-})dt+{\bf{S}}_{2}d W_{2,t}+\int_{E_2}c_2(\delta_{t-},z_2)p_2(dt,dz_2),\quad \delta_0=x_{2,0}, \label{delta}
\end{align}
where $E_1=\mathbb{R}^{k_1} \backslash \{0\}$, $E_2=\mathbb{R}^{p_1} \backslash \{0\}$, and $x_{1,0}\in\mathbb{R}^{k_1}$ and $x_{2,0}\in\mathbb{R}^{p_1}$ are 
non-random vectors. $\{W_{1,t}\}_{t\geq 0}$ and $\{W_{2,t}\}_{t\geq 0}$ are $r_1$ and $r_2$-dimensional standard $(\mathcal{F}_t)$-Wiener processes, respectively. $p_1(dt,dz_1)$ and $p_2(dt,dz_2)$ are Poisson random measures on $\mathbb{R}_{+}\times E_1$ and $\mathbb{R}_{+}\times E_2$ with compensator $q_1(dt,dz_1)={\bf{E}}\bigl[p_1(dt,dz_1)\bigr]$ and $q_2(dt,dz_2)={\bf{E}}\bigl[p_2(dt,dz_2)\bigr]$, respectively. The functions $a_1:\mathbb{R}^{k_1}\longrightarrow\mathbb{R}^{k_1}$, $a_2:\mathbb{R}^{p_1}\longrightarrow\mathbb{R}^{p_1}$, 
$c_1:\mathbb{R}^{k_1}\times E_1\longrightarrow\mathbb{R}^{k_1}$,
$c_2:\mathbb{R}^{p_1}\times E_2\longrightarrow\mathbb{R}^{p_1}$
are unknown Borel functions, and ${\bf{S}}_{1}\in\mathbb{R}^{k_1\times r_1}$ and ${\bf{S}}_{2}\in\mathbb{R}^{p_1\times r_2}$ are constant matrices. Set the volatility matrices of $\{\xi_t\}_{t\geq 0}$ and $\{\delta_t\}_{t\geq 0}$ as ${\bf{\Sigma}}_{\xi\xi}={\bf{S}}_{1}{\bf{S}}_{1}^{\top}$ and 
${\bf{\Sigma}}_{\delta\delta}={\bf{S}}_{2}{\bf{S}}_{2}^{\top}$, where $\top$ stands for the transpose. The $p_2$-dimensional observable process $\{X_{2,t}\}_{t\geq 0}$ satisfies the factor model as follows:
\begin{align}
    X_{2,t}&={\bf{\Lambda}}_{2}\eta_t+\varepsilon_t,
\end{align}
where $\{\eta_{t}\}_{t\geq 0}$ is a $k_2$-dimensional latent process, $\{\varepsilon_t\}_{t\geq 0}$ is a $p_2$-dimensional c{\`a}dl{\`a}g $(\mathcal{F}_t)$-adapted latent process on the stochastic basis, $k_2\leq p_2$, and ${\bf{\Lambda}}_{2}\in\mathbb{R}^{p_2\times k_2}$ is a constant loading matrix. In addition, we express the relationship between $\{\xi_{t}\}_{t\geq 0}$ and $\{\eta_{t}\}_{t\geq 0}$ as follows:
\begin{align}
    \eta_t={\bf{B}}_0\eta_t+{\bf{\Gamma}}\xi_t+\zeta_t,
\end{align}
where $\{\zeta_t\}_{t\geq 0}$ is a $k_2$-dimensional c{\`a}dl{\`a}g $(\mathcal{F}_t)$-adapted latent process on the stochastic basis, ${\bf{B}}_0\in\mathbb{R}^{k_2\times k_2}$ is a constant loading matrix whose diagonal elements are zero, and ${\bf{\Gamma}}\in\mathbb{R}^{k_2\times k_1}$ is a constant loading matrix. We suppose that ${\bf{\Lambda}}_{1}$ and ${\bf{\Lambda}}_{2}$ are full column rank, and ${\bf{\Psi}}=\mathbb{I}_{k_2}-{\bf{B}}_0$ is non-singular, where $\mathbb{I}_{k_2}$ is the identity matrix of size $k_2$. $\{\varepsilon_t\}_{t\geq 0}$ and $\{\zeta_t\}_{t\geq 0}$ satisfy the following stochastic differential equations:
\begin{align}
    d\varepsilon_{t}&=a_{3}(\varepsilon_{t-})dt+{\bf{S}}_{3}d W_{3,t}+\int_{E_3}c_3(\varepsilon_{t-},z_3)p_3(dt,dz_3),\quad \varepsilon_0=x_{3,0}, \label{epsillon}\\
    d\zeta_{t}&=a_{4}(\zeta_{t-})dt+{\bf{S}}_{4}d W_{4,t}+\int_{E_4}c_4(\zeta_{t-},z_4)p_4(dt,dz_4),\quad \zeta_0=x_{4,0},\label{zeta}
\end{align}
where $E_3=\mathbb{R}^{p_2} \backslash \{0\}$, $E_4=\mathbb{R}^{k_2} \backslash \{0\}$, and $x_{3,0}\in\mathbb{R}^{p_2}$ and $x_{4,0}\in\mathbb{R}^{k_2}$ are non-random vectors. $\{W_{3,t}\}_{t\geq 0}$ and $\{W_{4,t}\}_{t\geq 0}$ are $r_3$ and $r_4$-dimensional standard $(\mathcal{F}_t)$-Wiener processes, respectively. $p_3(dt,dz_3)$ and $p_4(dt,dz_4)$ are Poisson random measures on $\mathbb{R}_{+}\times E_3$ and $\mathbb{R}_{+}\times E_4$ with compensator $q_3(dt,dz_3)={\bf{E}}\bigl[p_3(dt,dz_3)\bigr]$ and $q_4(dt,dz_4)={\bf{E}}\bigl[p_4(dt,dz_4)\bigr]$, respectively. The functions $a_3:\mathbb{R}^{p_2}\longrightarrow\mathbb{R}^{p_2}$, $a_4:\mathbb{R}^{k_2}\longrightarrow\mathbb{R}^{k_2}$, 
$c_3:\mathbb{R}^{p_2}\times E_3\longrightarrow\mathbb{R}^{p_2}$,
$c_4:\mathbb{R}^{k_2}\times E_4\longrightarrow\mathbb{R}^{k_2}$
are unknown Borel functions, and ${\bf{S}}_{3}\in\mathbb{R}^{p_2\times r_3}$ and ${\bf{S}}_{4}\in\mathbb{R}^{k_2\times r_4}$ are constant matrices. 
Define the volatility matrices of $\{\varepsilon_t\}_{t\geq 0}$ and $\{\zeta_t\}_{t\geq 0}$ 
as ${\bf{\Sigma}}_{\varepsilon\varepsilon}={\bf{S}}_{3}{\bf{S}}_{3}^{\top}$ and 
${\bf{\Sigma}}_{\zeta\zeta}={\bf{S}}_{4}{\bf{S}}_{4}^{\top}$, respectively.
Suppose that for any $t\geq 0$, $\mathcal{F}_t$, $\sigma(W_{i,u}-W_{i,t};\ u\geq t) \ (i=1,2,3,4)$, and
\begin{align*}
    \sigma\bigl(p_i(A_i\cap ((t,\infty)\times E_i)); A_i\subset\mathbb{R}_{+}\times E_i \mbox{\ is a Borel set}\bigr)\quad (i=1,2,3,4)
\end{align*}
are independent. For $i=1,2,3,4$, we assume that $q_i(dt,dz_i)$ has a representation such that $q_i(dt,dz_i)=f_i(z_i)dz_idt$, and $f_i(z_i)=\lambda_{i,0}F_i(z_i)$, where $\lambda_{i,0}>0$ is a unknown value, and $F_i(z_i)$ is a unknown probability density. Set $p=p_1+p_2$, and $X_t=(X_{1,t}^{\top},X_{2,t}^{\top})^{\top}$. $(X_{t_i^n})_{i=0}^n$ are discrete observations, where $t_i^n=ih_n$. 
We suppose that $T=nh_n$ is fixed. Let
\begin{align*}
    {\bf{\Sigma}}=\begin{pmatrix}
    {\bf{\Sigma}}^{11} & {\bf{\Sigma}}^{12}\\
    {\bf{\Sigma}}^{12\top} & {\bf{\Sigma}}^{22}
    \end{pmatrix},
\end{align*}
where
\begin{align*}
    {\bf{\Sigma}}^{11}&={\bf{\Lambda}}_{1}{\bf{\Sigma}}_{\xi\xi}{\bf{\Lambda}}_{1}^{\top}+{\bf{\Sigma}}_{\delta\delta},\\
    {\bf{\Sigma}}^{12}&={\bf{\Lambda}}_{1}{\bf{\Sigma}}_{\xi\xi}{\bf{\Gamma}}^{\top}{\bf{\Psi}}^{-1\top}{\bf{\Lambda}}_{2}^{\top},\\
    {\bf{\Sigma}}^{22}&={\bf{\Lambda}}_{2}{\bf{\Psi}}^{-1}({\bf{\Gamma}}{\bf{\Sigma}}_{\xi\xi}{\bf{\Gamma}}^{\top}+{\bf{\Sigma}}_{\zeta\zeta}){\bf{\Psi}}^{-1\top}{\bf{\Lambda}}_{2}^{\top}+{\bf{\Sigma}}_{\varepsilon\varepsilon}.
\end{align*}
We assume that ${\bf{\Sigma}}$ is positive definite. For example, when ${\bf{\Sigma}}_{\delta\delta}$ and ${\bf{\Sigma}}_{\varepsilon\varepsilon}$ are positive definite, ${\bf{\Sigma}}$ is positive definite. ${\bf{P}}_{\bf{\Sigma}}$ denotes the law of the process $\{X_t\}_{t\geq 0}$ defined by (\ref{X1})-(\ref{zeta}).

SEM is a statistical method that analyzes the relationships between unobservable variables and has been used in a wide range of fields, e.g., psychology, biology, economics, and so on. The unobservable variables are called latent variables. Intelligence and economic conditions are one of the most important examples of latent variables. There are a lot of studies on SEM. For instance, see Shapiro
\cite{Shapiro(1983)}, \cite{Shapiro(1985)} for statistical inference, Jacobucci \cite{Jacobucci(2016)} and Huang \cite{Huang(2017)} for sparse inference, Huang \cite{Huang AIC(2017)} for model selection, and Czir{\' a}ky \cite{Cziraky(2004)} for SEM for time series data.

In recent years, we can easily access high-frequency data thanks to the development of information technology. For that reason, a lot of researchers have studied statistical inference for diffusion processes based on discrete observations; see, e.g., Yoshida \cite{Yoshida(1992)}, Kessler \cite{kessler(1997)}, Uchida and Yoshida \cite{Uchi-Yoshi(2012)} and references therein. Statistical inference for jump-diffusion processes has been also studied intensively
because we often observe discontinuous sample paths. As a result, jump-diffusion models have been used for modeling in various fields, including financial econometrics, physics and hydrology. Shimizu and Yoshida \cite{Shimizu-Yoshida_JP, Shimizu(2006)} investigated statistical inference for ergodic diffusion processes with jumps based on discrete observations. Using a threshold method based on the increment, they construct a quasi-likelihood and show the quasi-maximum likelihood estimator has consistency and asymptotic normality.
Notably, Mancini \cite{Mancini(2004)} independently developed a consistent estimator for the characteristics of jumps in the Poisson-diffusion model.
Ogihara and Yoshida \cite{Ogihara(2011)} proved the convergence of moments of the quasi-maximum likelihood estimator. For statistical inference of ergodic jump diffusion processes, see also Amorino and Gloter \cite{Amorino(2020), Amorino(2021)}. For another filtering method, see, e.g., Inatsugu and Yoshida \cite{Inatsugu(2021)}.

Recently, Kusano and Uchida \cite{Kusano(JJSD)} proposed SEM for diffusion processes. By using this method, we can examine the relationships between latent processes based on high-frequency data. For more details on the SEM, we also refer readers to Kusano and Uchida \cite{Kusano(2024c), Kusano(2024e)}. On the other hand, it is not appropriate to use this method when observed data may contain jumps. That is because Kusano and Uchida \cite{Kusano(JJSD)} suppose that 
the latent processes are diffusion processes, and do not consider jumps. For example, 
we cannot use the result of  Kusano and Uchida \cite{Kusano(JJSD)} for the data as shown in Figure \ref{jump}.
As mentioned above, jump-diffusion models are regarded as powerful tools to express a wide range of stochastic phenomena.
Therefore, we propose SEM for diffusion processes with jumps.

This paper is organized as follows. In Section \ref{notation}, we introduce notation and assumptions. In Section \ref{Main Theorem}, we consider the statistical inference of SEM for diffusion processes with jumps. First, we prepare some lemmas to justify the use of a threshold method. Using the threshold method, we define the estimator of ${\bf{\Sigma}}$, and investigate the asymptotic properties. Next, the quasi-likelihood of the SEM is proposed, and it is shown that the quasi-maximum likelihood estimator has consistency and asymptotic normality. Furthermore, to examine whether a specified parametric model is correct or not, we study the quasi-likelihood ratio test of the SEM. In Section \ref{Simulation}, we give examples and simulation results. Section \ref{Proof} is devoted to the proofs of the results stated in Section \ref{Main Theorem}.
\begin{figure}[h]
    \centering
    \includegraphics[width=0.85\columnwidth]{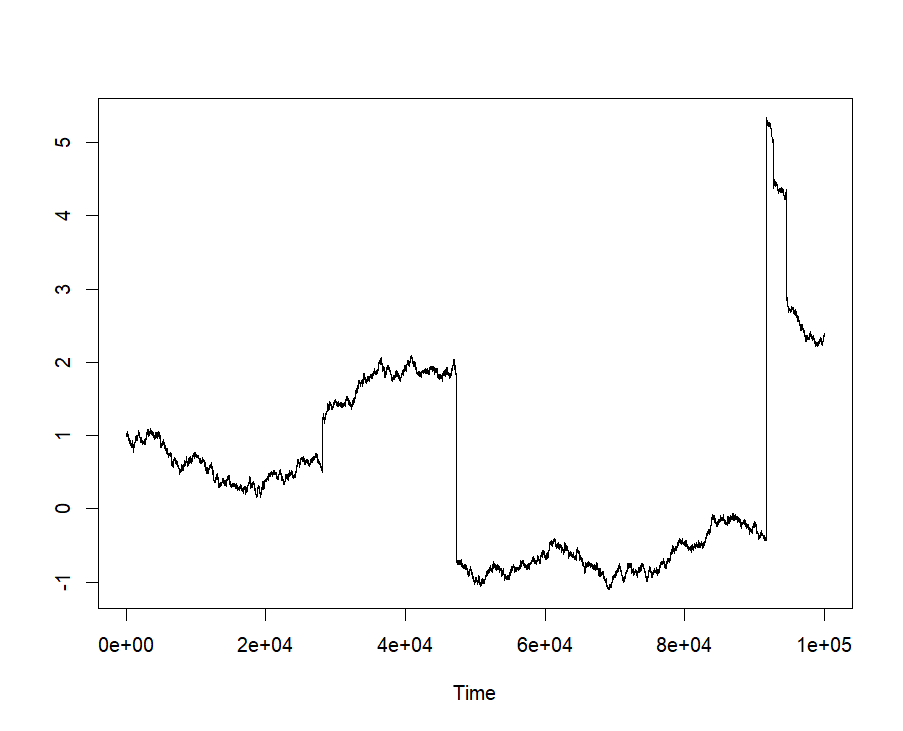}
\caption{Sample path from a diffusion process with jumps.}\label{jump}
\end{figure}


\section{Notation and Assumptions}\label{notation}
First, we introduce the notation as follows. For a vector $v$, $v^{(i)}$ stands for the $i$-th element of $v$ and let $|v|=\sqrt{\tr{vv^\top}}$. $\Diag v$ denotes the diagonal matrix whose $i$-th diagonal element is $v^{(i)}$. For a matrix $M$, $M_{ij}$ is the $(i,j)$-th component of $M$ and $|M|=\sqrt{\tr{MM^\top}}$. For a symmetric matrix $M$, $\vec M$ and $\vech M$ are the vectorization of $M$, and the half-vectorization of $M$, respectively.  For a $p$-dimensional symmetric matrix $M$, $\mathbb{D}_p$ is the $p^2\times \bar{p}$ matrix such that $\vec{M}=\mathbb{D}_{p}\vech{M}$, where $\bar{p}=p(p+1)/2$. Set $\mathbb{D}_p^{+}=\bigl(\mathbb{D}_p^{\top}\mathbb{D}_p\bigr)^{-1}
\mathbb{D}_p^{\top}$. Note that $\vech{M}=\mathbb{D}_{p}^{+}\vec{M}$; see, e.g., Harville \cite{Harville(1998)}. $\mathcal{M}_{p}^{+}$ denotes the set of all $p\times p$ real-valued positive definite matrices. For a positive sequence $u_{n}$, $R:[0,\infty)\times \mathbb{R}^d\rightarrow \mathbb{R}$ stands for the short notation for functions which satisfy $|R({u_{n}},x)|\leq u_{n}C(1+|x|)^C$ for some $C>0$. Let $ C^{k}_{\uparrow}(\mathbb R^{d})$ be the space of all functions $f$ satisfying the following conditions:
\begin{itemize}
    \item[(i)] $f$ is continuously differentiable with respect to $x\in \mathbb{R}^d$ up to order $k$. 
    \item[(ii)] $f$ and all its derivatives are of polynomial growth in $x\in \mathbb{R}^d$, where $f$ is of polynomial growth in $x\in \mathbb{R}^d$ if 
    $f(x)=R(1,x)$.
\end{itemize}
Let $N_{d}(\mu,\Sigma)$ be a $d$-dimensional normal random variable with mean $\mu\in\mathbb{R}^d$ and covariance matrix $\Sigma\in\mathcal{M}_d^{+}$. $Z_d$ denotes a $d$-dimensional standard normal random variable. $\chi^2_{r}$ represents the random variable which has the chi-squared distribution with $r$ degrees of freedom, and $\chi^2_{r}(\alpha)$ stands for an upper $\alpha\in[0,1]$ point of $\chi^2_r$. The symbols $\stackrel{p}{\longrightarrow}$ and $\stackrel{d}{\longrightarrow}$ express convergence in probability and convergence in distribution, respectively. For a stochastic process $\{S_t\}_{t\geq 0}$, $\Delta_i^n S=S_{t_i^n}-S_{t_{i-1}^n}$, $\Delta S_t=S_t-S_{t-}$, and $R_i(u_n,S)=R(u_n,S_{t_i^n})$. Set $(d_1,d_2,d_3,d_4)=(k_1,p_1,p_2,k_2)$. Define $\mathcal{F}_{i}^n=\mathcal{F}_{t_{i}^n}$ for $i=0,\ldots,n$. 
Let 
\begin{align*}
    R_{i-1}(u_n,\xi,\delta,\varepsilon,\zeta)=R_{i-1}({u_{n}},\xi)+
    R_{i-1}({u_{n}},\delta)+R_{i-1}({u_{n}},\varepsilon)+R_{i-1}({u_{n}},\zeta)
\end{align*}
and
\begin{align*}
    \tilde{R}_{i-1}(u_n,\xi,\delta,\varepsilon,\zeta)=1-R_{i-1}(u_n,\xi,\delta,\varepsilon,\zeta).
\end{align*}
Next, we make the following assumptions.
\begin{enumerate}
    \vspace{2mm}
    \item[\bf{[A1]}]
    For $i=1,2,3,4$, there exist $L_i>0$ and $\zeta_i(z_i)>0$ of at most polynomial growth in $z_i$ such that
    \begin{align*}
        |a_i(x_i)-a_i(y_i)|\leq L_i|x_i-y_i|,\ |c_i(x_i,z_i)-c_i(y_i,z_i)|\leq \zeta_i(z_i)|x_i-y_i|
    \end{align*}
    and
    \begin{align*}
        |c_i(x_i,z_i)|\leq \zeta_i(z_i)(1+|x_i|)
    \end{align*}
    for any $x_i,y_i\in \mathbb{R}^{d_i}$ and $z_i\in E_i$.
    \vspace{2mm}
    \item[\bf{[A2]}]
     For $i=1,2,3,4$, $a_i\in C_{\uparrow}^4(\mathbb{R}^{d_i})$.
     \vspace{2mm}
    \item[\bf{[A3]}] For $i=1,2,3,4$, there exist $r_i>0$ and $K_i>0$ such that
    \begin{align*}
        f_i(z_i){\bf{1}}_{\{|z_i|\leq r_i\}}\leq K_i|z_i|^{1-d_i}
    \end{align*}
    and
    \begin{align*}
        \int |z_i|^{l}f_i(z_i)dz_i<\infty
    \end{align*}
    for any $l\geq 1$.
    \vspace{2mm}
    \item[\bf{[A4]}]
    For $i=1,2,3,4$, it holds that
    \begin{align*}
        \inf_{x_i}|c_i(x_i,z_i)|\geq c_{i,0}|z_i|
    \end{align*}
    for some $c_{i,0}>0$ near the origin.
\end{enumerate}
\begin{remark}
{\bf{[A1]}} and {\bf{[A3]}} deduce
{\setlength{\abovedisplayskip}{8pt}
\setlength{\belowdisplayskip}{8pt}
\begin{align}
    \sup_{t\in[0,T]}{\bf{E}}\bigl[|\xi_{t}|^{l}\bigr]<\infty,\  
    \sup_{t\in[0,T]}{\bf{E}}\bigl[|\delta_{t}|^{l}\bigr]<\infty,\
    \sup_{t\in[0,T]}{\bf{E}}\bigl[|\varepsilon_{t}|^{l}\bigr]<\infty,\
    \sup_{t\in[0,T]}{\bf{E}}\bigl[|\zeta_{t}|^{l}\bigr]<\infty
    \label{moment}
\end{align}
for any $l\geq 0$ when $T$ is fixed; see, e.g., Platen and Bruti-Liberati \cite{Platen(2010)}.} In this paper, we only state the asymptotic results 
in the case where $T$ is fixed. On the other hand, 
we can also consider the case where $h_n\longrightarrow 0$, $T=nh_n\longrightarrow\infty$ and $nh_n^2\longrightarrow 0$ as $n\longrightarrow\infty$. In this case, we need to assume some regularity conditions including (\ref{moment}).
\end{remark}
\section{Main Theorem}\label{Main Theorem}
First, we consider the estimation of ${\bf{\Sigma}}$. Set the true values of ${\bf{\Lambda}}_{1}$, ${\bf{\Lambda}}_{2}$, ${\bf{\Gamma}}$, ${\bf{\Psi}}$, ${\bf{\Sigma}}_{\xi\xi}$, ${\bf{\Sigma}}_{\delta\delta}$, ${\bf{\Sigma}}_{\varepsilon\varepsilon}$ and ${\bf{\Sigma}}_{\zeta\zeta}$ as 
${\bf{\Lambda}}_{1,0}$, ${\bf{\Lambda}}_{2,0}$, ${\bf{\Gamma}}_0$, ${\bf{\Psi}}_0$, ${\bf{\Sigma}}_{\xi\xi,0}$, ${\bf{\Sigma}}_{\delta\delta,0}$, ${\bf{\Sigma}}_{\varepsilon\varepsilon,0}$ and ${\bf{\Sigma}}_{\zeta\zeta,0}$, respectively. Note that the true value of ${\bf{\Sigma}}$ is expressed as 
\begin{align*}
    {\bf{\Sigma}}_0=\begin{pmatrix}
    {\bf{\Sigma}}_0^{11} & {\bf{\Sigma}}_0^{12}\\
    {\bf{\Sigma}}_0^{12\top} & {\bf{\Sigma}}_0^{22}
    \end{pmatrix},
\end{align*}
where
\begin{align*}
    {\bf{\Sigma}}_0^{11}&={\bf{\Lambda}}_{1,0}{\bf{\Sigma}}_{\xi\xi,0}{\bf{\Lambda}}_{1,0}^{\top}+{\bf{\Sigma}}_{\delta\delta,0},\\
    {\bf{\Sigma}}_0^{12}&={\bf{\Lambda}}_{1,0}{\bf{\Sigma}}_{\xi\xi,0}{\bf{\Gamma}}_0^{\top}{\bf{\Psi}}_0^{-1\top}{\bf{\Lambda}}_{2,0}^{\top},\\
    {\bf{\Sigma}}_0^{22}&={\bf{\Lambda}}_{2,0}{\bf{\Psi}}_0^{-1}({\bf{\Gamma}}_0{\bf{\Sigma}}_{\xi\xi,0}{\bf{\Gamma}}_0^{\top}+{\bf{\Sigma}}_{\zeta\zeta,0}){\bf{\Psi}}_0^{-1\top}{\bf{\Lambda}}_{2,0}^{\top}+{\bf{\Sigma}}_{\varepsilon\varepsilon,0}.
\end{align*}
In this paper, we judge whether jumps occur or not in the interval $(t_{i-1}^n,t_i^n]$ from the increment $|\Delta_i^n X|$; see, e.g., Shimizu and Yoshida \cite{Shimizu(2006)} and Ogihara and Yoshida \cite{Ogihara(2011)}. In other words, if $|\Delta_i^n X|$ exceeds $Dh_n^{\rho}$, we judge that a jump occurs in the interval, and if not, we consider that no jumps occur in the interval, where $\rho\in[0,1/2)$ and $D>0$. Set two random times as
\begin{align*}
    \tau_i^n&=\inf\bigl\{t\in[t_{i-1}^n,t_i^n);\ |\Delta \xi_t|>0 \lor |\Delta \delta_t|>0\lor |\Delta \varepsilon_t|>0\lor |\Delta \zeta_t|>0 \bigr\}
\end{align*}
and
\begin{align*}  
    \eta_i^n&=\sup\bigl\{t\in[t_{i-1}^n,t_i^n);\ |\Delta \xi_t|>0 \lor |\Delta \delta_t|>0\lor |\Delta \varepsilon_t|>0\lor |\Delta \zeta_t|>0 \bigr\}
\end{align*}
for $i=1,\ldots,n$. Here, we define that $\tau_i^n=\eta_i^n=t_i^n$ if the infimum or supremum on the right-hand side does not exist. 
Note that $\tau_i^n$ and $\eta_i^n$ mean the first 
and the last jump time in the interval $[t_{i-1}^n,t_i^n)$, respectively. For the two random times, the following lemma is shown. 
\begin{lemma}\label{Pine}
Under ${\bf{[A1]}}$ and ${\bf{[A3]}}$, for $\rho\in[0,1/2)$, $D>0$ and all $p\geq 1$,
\begin{align}       
    {\bf{P}}_{{\bf{\Sigma}}_0}\biggl(\sup_{t\in[t_{i-1}^n,\tau_i^n)}|X_t-X_{t_{i-1}^n}|>Dh_n^{\rho}\Bigl|\mathcal{F}_{i-1}^n\biggr)=
    R_{i-1}(h_n^p,\xi,\delta,\varepsilon,\zeta) \label{supine1}
\end{align}
and
\begin{align}
    {\bf{P}}_{{\bf{\Sigma}}_0}
    \biggl(\sup_{t\in[\eta_{i}^n,t_i^n)}|X_{t_{i}^n}-X_t|
    >Dh_n^{\rho}\Bigl|\mathcal{F}_{i-1}^n\biggr)=R_{i-1}(h_n^p,\xi,\delta,\varepsilon,\zeta)\ \
    \label{supine2}
\end{align}
for $i=1,\ldots,n$, where $\sup\phi=-\infty$.
\end{lemma}
Next, we define the following events:
\begin{align*}
    C_{i,j,0}^n&=\bigl\{J_{i,j}^n=0,\ |\Delta X_{i}^n|\leq Dh_n^{\rho}\bigl\},\quad C_{i,j,1}^n=\bigl\{J_{i,j}^n=1,\ |\Delta X_{i}^n|\leq Dh_n^{\rho}\bigl\},\\
    C_{i,j,2}^n&=\bigl\{J_{i,j}^n\geq 2,\ |\Delta X_{i}^n|\leq Dh_n^{\rho}\bigl\}, \quad 
    D_{i,j,0}^n=\bigl\{J_{i,j}^n=0,\ |\Delta X_{i}^n|> Dh_n^{\rho}\bigl\},\\
    D_{i,j,1}^n&=\bigl\{J_{i,j}^n=1,\ |\Delta X_{i}^n|> Dh_n^{\rho}\bigl\},\quad
    D_{i,j,2}^n=\bigl\{J_{i,j}^n\geq 2,\ |\Delta X_{i}^n|> Dh_n^{\rho}\bigl\}
\end{align*}
for $i=1,\ldots,n$ and $j=1,2,3,4$, where $J_{i,j}^n=p_j((t_{i-1}^n,t_i^n]\times E_j)$. Let
\begin{align*}
    C_{i,k_1,k_2,k_3,k_4}^n=C_{i,1,k_1}^n\cap C_{i,2,k_2}^n\cap C_{i,3,k_3}^n \cap C_{i,4,k_4}^n
\end{align*}
and
\begin{align*}
    D_{i,k_1,k_2,k_3,k_4}^n=D_{i,1,k_1}^n\cap D_{i,2,k_2}^n\cap D_{i,3,k_3}^n \cap D_{i,4,k_4}^n
\end{align*}
for $i=1,\ldots,n$ and $k_1,k_2,k_3,k_4=0,1,2$. 
Note that
\begin{align*}
    \bigl\{|\Delta X_{i}^n\bigr|\leq Dh_n^{\rho}\bigr\}&=
    \bigcup_{k_1,k_2,k_3,k_4=0,1,2}C_{i,k_1,k_2,k_3,k_4}^n
\end{align*}
and
\begin{align*}
    \bigl\{|\Delta X_{i}^n\bigr|> Dh_n^{\rho}\bigr\}&=
    \bigcup_{k_1,k_2,k_3,k_4=0,1,2}
    D_{i,k_1,k_2,k_3,k_4}^n
\end{align*}
for $i=1,\ldots,n$. Furthermore, we define the sets as follows:
\begin{align*}
    K_1&=\Bigl\{(k_1,k_2,k_3,k_4)\in\{0,1,2\}^4;\ 
    \mbox{One element of}\ k_1, k_2, k_3, \mbox{and}\  k_4\ \mbox{is}\ 1,\ \mbox{and the others are 0}\Bigr\},\\
    K_2&=\Bigl\{(k_1,k_2,k_3,k_4)\in\{0,1,2\}^4;\ 
    \mbox{Two elements of}\ k_1, k_2, k_3, \mbox{and}\  k_4\ \mbox{are}\ 1,\ \mbox{and the others are 0}\Bigr\},\\
    K_3&=\Bigl\{(k_1,k_2,k_3,k_4)\in\{0,1,2\}^4;\ 
    \mbox{Three elements of}\ k_1, k_2, k_3, \mbox{and}\  k_4\ \mbox{are}\ 1,\ \mbox{and the other is 0} \Bigr\}
\end{align*}
and
\begin{align*}
    K_4&=\Bigl\{(k_1,k_2,k_3,k_4)\in\{0,1,2\}^4;\ 
    \mbox{At least one of}\ k_1, k_2, k_3, \mbox{and}\  k_4\  \mbox{is}\ 2\Bigr\}.
\end{align*}
Below, we always suppose that $\rho\in[1/3,1/2)$. Using Lemma \ref{Pine}, we can obtain the following lemmas.
\begin{lemma}\label{Clemma}
Under ${\bf{[A1]}}$, ${\bf{[A3]}}$ and ${\bf{[A4]}}$, for a sufficiently large $n$,
\begin{align}
    &{\bf{P}}_{{\bf{\Sigma}}_0}\Bigl(C^{n}_{i,0,0,0,0}\big|\mathcal{F}_{i-1}^n\Bigr)=\tilde{R}_{i-1}(h_n,\xi,\delta,\varepsilon,\zeta),\label{C0}\\
    &{\bf{P}}_{{\bf{\Sigma}}_0}\Bigl(C^{n}_{i,1,1,1,1}\big|\mathcal{F}_{i-1}^n\Bigr)\leq\lambda_{1,0}\lambda_{2,0}\lambda_{3,0}\lambda_{4,0} h_n^4,\label{C1}\\
    &{\bf{P}}_{{\bf{\Sigma}}_0}\Bigl(C^{n}_{i,k_1,k_2,k_3,k_4}\big|\mathcal{F}_{i-1}^n\Bigr)=R_{i-1}(h_n^{\rho+1},\xi,\delta,\varepsilon,\zeta)\quad (k_1,k_2,k_3,k_4)\in K_1\label{CK1},\\
    &{\bf{P}}_{{\bf{\Sigma}}_0}\Bigl(C^{n}_{i,k_1,k_2,k_3,k_4}\big|\mathcal{F}_{i-1}^n\Bigr)\leq \Bigl(\max_{j=1,2,3,4}\lambda_{j,0}\Bigr)^2 h_n^2\quad (k_1,k_2,k_3,k_4)\in K_2,\label{CK2}\\
    & {\bf{P}}_{{\bf{\Sigma}}_0}\Bigl(C^{n}_{i,k_1,k_2,k_3,k_4}\big|\mathcal{F}_{i-1}^n\Bigr)\leq 
    \Bigl(\max_{j=1,2,3,4}\lambda_{j,0}\Bigr)^3 h_n^3\quad (k_1,k_2,k_3,k_4)\in K_3 \label{CK3}
\end{align}
and
\begin{align}
    {\bf{P}}_{{\bf{\Sigma}}_0}\Bigl(C^{n}_{i,k_1,k_2,k_3,k_4}\big|\mathcal{F}_{i-1}^n\Bigr)\leq \Bigl(\max_{j=1,2,3,4}\lambda_{j,0}^2\Bigr)h_n^2
    \quad (k_1,k_2,k_3,k_4)\in K_4\qquad \label{CK4}
\end{align}
for $i=1,\ldots,n$.
\end{lemma}
\begin{lemma}\label{Dlemma}
Under ${\bf{[A1]}}$, ${\bf{[A3]}}$ and ${\bf{[A4]}}$, for any $p\geq 1$ and a sufficiently large $n$,
\begin{align}
    &{\bf{P}}_{{\bf{\Sigma}}_0}\Bigl(D^{n}_{i,0,0,0,0}\big|\mathcal{F}_{i-1}^n\Bigr)=R_{i-1}(h_n^p,\xi,\delta,\varepsilon,\zeta),\label{D0}\\
    &{\bf{P}}_{{\bf{\Sigma}}_0}\Bigl(D^{n}_{i,1,1,1,1}\big|\mathcal{F}_{i-1}^n\Bigr)\leq\lambda_{1,0}\lambda_{2,0}\lambda_{3,0}\lambda_{4,0} h_n^4,\label{D1}\\
    &{\bf{P}}_{{\bf{\Sigma}}_0}\Bigl(D^{n}_{i,1,0,0,0}\big|\mathcal{F}_{i-1}^n\Bigr)=\lambda_{1,0}h_n\tilde{R}_{i-1}(h_n^{\rho},\xi,\delta,\varepsilon,\zeta),\label{DK11}\\
    &{\bf{P}}_{{\bf{\Sigma}}_0}\Bigl(D^{n}_{i,0,1,0,0}\big|\mathcal{F}_{i-1}^n\Bigr)=\lambda_{2,0}h_n\tilde{R}_{i-1}(h_n^{\rho},\xi,\delta,\varepsilon,\zeta),\label{DK12}\\
    &{\bf{P}}_{{\bf{\Sigma}}_0}\Bigl(D^{n}_{i,0,0,1,0}\big|\mathcal{F}_{i-1}^n\Bigr)=\lambda_{3,0}h_n\tilde{R}_{i-1}(h_n^{\rho},\xi,\delta,\varepsilon,\zeta),\\
    &{\bf{P}}_{{\bf{\Sigma}}_0}\Bigl(D^{n}_{i,0,0,0,1}\big|\mathcal{F}_{i-1}^n\Bigr)=\lambda_{4,0}h_n\tilde{R}_{i-1}(h_n^{\rho},\xi,\delta,\varepsilon,\zeta),\label{DK14}\\
    &{\bf{P}}_{{\bf{\Sigma}}_0}\Bigl(D^{n}_{i,k_1,k_2,k_3,k_4}\big|\mathcal{F}_{i-1}^n\Bigr)\leq \Bigl(\max_{j=1,2,3,4}\lambda_{j,0}\Bigr)^2 h_n^2\quad (k_1,k_2,k_3,k_4)\in K_2,\label{DK2}\\
    & {\bf{P}}_{{\bf{\Sigma}}_0}\Bigl(D^{n}_{i,k_1,k_2,k_3,k_4}\big|\mathcal{F}_{i-1}^n\Bigr)\leq 
    \Bigl(\max_{j=1,2,3,4}\lambda_{j,0}\Bigr)^3 h_n^3\quad (k_1,k_2,k_3,k_4)\in K_3  \label{DK3}
\end{align}
and
\begin{align}
    &{\bf{P}}_{{\bf{\Sigma}}_0}\Bigl(D^{n}_{i,k_1,k_2,k_3,k_4}\big|\mathcal{F}_{i-1}^n\Bigr)\leq \Bigl(\max_{j=1,2,3,4}\lambda_{j,0}^2\Bigr)h_n^2\quad (k_1,k_2,k_3,k_4)\in K_4 \label{DK4}
\end{align}
for $i=1,\ldots,n$.
\end{lemma}
By Lemmas \ref{Clemma} and \ref{Dlemma}, we can judge that no jumps occur in the interval $[t_{i-1}^n,t_i^n)$ if $|\Delta X_{i}^n\bigr|\leq Dh_n^{\rho}$, and a single jump occurs in the interval $[t_{i-1}^n,t_i^n)$ if $|\Delta X_{i}^n\bigr|> Dh_n^{\rho}$. Thus, to estimate ${\bf{\Sigma}}_0$, we use the estimator
\begin{align*}
    \hat{\bf{\Sigma}}_n=\frac{1}{\tilde{N}_nh_n}\sum_{i=1}^n(\Delta_i^n X)(\Delta_i^n X)^{\top}{\bf{1}}_{\{|\Delta_i^n X|\leq Dh_n^{\rho}\}},
\end{align*}
where 
\begin{align*}
    N_n=\sum_{i=1}^n{\bf{1}}_{\{|\Delta_{i}^n X|\leq Dh_n^{\rho}\}},\quad \tilde{N}_n=\left\{
    \begin{array}{ll}
    N_n & (N_n\neq 0),\\
    n & (N_n=0).
    \end{array}\right.
\end{align*}
The asymptotic result of $\hat{\bf{\Sigma}}_n$ is as follows.
\begin{theorem}\label{Qtheorem}
Under {\rm{\textbf{[A1]}}}-{\rm{\textbf{[A4]}}}, as $n\longrightarrow\infty$,
\begin{align}
    \hat{\bf{\Sigma}}_n\overset{p}{\longrightarrow}{\bf{\Sigma}}_0 \label{Qcons}
\end{align}
and
\begin{align}
    \sqrt{n}(\vech \hat{\bf{\Sigma}}_n-\vech{\bf{\Sigma}}_0)\overset{d}{\longrightarrow}N_{\bar{p}}\Bigl(0,2\mathbb{D}_p^{+}({\bf{\Sigma}}_0\otimes {\bf{\Sigma}}_0)\mathbb{D}_p^{+\top}\Bigr) \label{Qasym}
\end{align}
under ${\bf{P}}_{{\bf{\Sigma}}_0}$.
\end{theorem}
Next, we consider the estimation of parameters
\begin{align}
    {\bf{\Lambda}}_1,\ {\bf{\Lambda}}_2,\ {\bf{\Gamma}},\ {\bf{\Psi}},\ {\bf{\Sigma}}_{\xi\xi},\ 
    {\bf{\Sigma}}_{\delta\delta},\ {\bf{\Sigma}}_{\varepsilon\varepsilon},\ {\bf{\Sigma}}_{\zeta\zeta}. \label{para}
\end{align}
Unfortunately, 
we cannot estimate all elements of (\ref{para}) due to non-identifiable problems. Hence, we impose constraints on the parameters (\ref{para}). For instance, statisticians may set some elements of (\ref{para}) to $0$ or $1$ to ensure identifiability. These constraints are made from the theoretical viewpoint of each research field. See Kusano and Uchida \cite{Kusano(JJSD)} for details of the constraints and the identifiability problem. Note that it is sufficient to estimate only unknown non-duplicated elements of
(\ref{para}). Set the vector of those elements as $\theta\in\Theta\subset\mathbb{R}^q$, where $q$ is the number of only unknown non-duplicated elements of
(\ref{para}), and $\Theta$ is a parameter space. For simplicity, we suppose that $\Theta$ is convex and compact. Hereafter, we write the parameters (\ref{para}) as
\begin{align*}
    {\bf{\Lambda}}_1^{\theta},\ {\bf{\Lambda}}_2^{\theta},\ {\bf{\Gamma}}^{\theta},\ {\bf{\Psi}}^{\theta},\ {\bf{\Sigma}}_{\xi\xi}^{\theta},\ 
    {\bf{\Sigma}}_{\delta\delta}^{\theta},\ {\bf{\Sigma}}_{\varepsilon\varepsilon}^{\theta},\ {\bf{\Sigma}}_{\zeta\zeta}^{\theta}. 
\end{align*}
In addition, we set 
\begin{align*}
    {\bf{\Sigma}}(\theta)=\begin{pmatrix}
    {\bf{\Sigma}}(\theta)^{11} & {\bf{\Sigma}}(\theta)^{12}\\
    {\bf{\Sigma}}(\theta)^{12\top} & {\bf{\Sigma}}(\theta)^{22}
    \end{pmatrix},
\end{align*}
where
\begin{align*}
    {\bf{\Sigma}}(\theta)^{11}&={\bf{\Lambda}}^{\theta}_{1}{\bf{\Sigma}}_{\xi\xi}^{\theta}{\bf{\Lambda}}_{1}^{\theta\top}+{\bf{\Sigma}}_{\delta\delta}^{\theta},\\
    {\bf{\Sigma}}(\theta)^{12}&={\bf{\Lambda}}_{1}^{\theta}{\bf{\Sigma}}_{\xi\xi}^{\theta}{\bf{\Gamma}}^{\theta\top}{\bf{\Psi}}^{\theta-1\top}{\bf{\Lambda}}_{2}^{\theta\top},\\
    {\bf{\Sigma}}(\theta)^{22}&={\bf{\Lambda}}_{2}^{\theta}{\bf{\Psi}}^{\theta-1}({\bf{\Gamma}}^{\theta}{\bf{\Sigma}}_{\xi\xi}^{\theta}{\bf{\Gamma}}^{\theta\top}+{\bf{\Sigma}}_{\zeta\zeta}^{\theta}){\bf{\Psi}}^{\theta-1\top}{\bf{\Lambda}}_{2}^{\theta\top}+{\bf{\Sigma}}_{\varepsilon\varepsilon}^{\theta}.
\end{align*}
Note that ${\bf{\Sigma}}(\theta)$ is two-times continuously differentiable
for $\theta$. To estimate $\theta$, we use the following quasi-likelihood:
\begin{align*}
    {\bf{L}}_n(\theta)=\exp\bigl\{{\bf{H}}_n(\theta)\bigr\},
\end{align*}
where 
\begin{align*}
    {\bf{H}}_n(\theta)
    =-\frac{1}{2}\sum_{i=1}^n\log\det {\bf{\Sigma}}(\theta){\bf{1}}_{\{|\Delta_{i}^n X|\leq h_n^{\rho}\}}-\frac{1}{2h_n}\sum_{i=1}^n(\Delta_{i}^n X)^{\top}{\bf{\Sigma}}(\theta)^{-1}(\Delta_{i}^n X){\bf{1}}_{\{|\Delta X_{i}^n|\leq h_n^{\rho}\}}.
\end{align*}
Define the quasi-maximum likelihood estimator as
\begin{align*}
    {\bf{H}}_n(\hat{\theta}_n)=\sup_{\theta\in\Theta}{\bf{H}}_n(\theta).
\end{align*}
Set the true value of $\theta$ as $\theta_0\in{{\rm{Int}}(\Theta)}$. Let
\begin{align*}
    \left.\Delta_0=\frac{\partial}{\partial\theta^{\top}}\vech{{\bf{\Sigma}}(\theta)}\right|_{\theta=\theta_0},\quad {\bf{W}}_0=2\mathbb{D}_p^{+}\bigl({\bf{\Sigma}}(\theta_0)\otimes {\bf{\Sigma}}(\theta_0)\bigr)\mathbb{D}_p^{+\top}
\end{align*}
and ${\bf{P}}_{\theta_0} ={\bf{P}}_{{\bf{\Sigma}}(\theta_0)}$. Furthermore, we make the following assumption.
\begin{enumerate}
    \item[\bf{[B1]}]
    \begin{enumerate}
        \item 
         ${\bf{\Sigma}}(\theta)={\bf{\Sigma}}(\theta_0)\Longrightarrow \theta=\theta_0$.
        \item $\Delta_0$ has full column rank.
    \end{enumerate}
\end{enumerate}
For the quasi-maximum likelihood estimator, the following theorem holds.
\begin{theorem}\label{thetatheorem}
Under {\rm{\textbf{[A1]}}}-{\rm{\textbf{[A4]}}} and {\rm{\textbf{[B1]}}}, as $n\longrightarrow\infty$,
\begin{align}
    \hat{\theta}_n\overset{p}{\longrightarrow}\theta_0 \label{thetacons}
\end{align}
and
\begin{align}
    \sqrt{n}\bigl(\hat{\theta}_n-\theta_0\bigr)\overset{d}{\longrightarrow}N_q\Bigl(0, \bigl(\Delta_0^{\top}{\bf{W}}_0^{-1}\Delta_0\bigr)^{-1} \Bigr) \label{thetaasym}
\end{align}
under ${\bf{P}}_{\theta_0}$.
\end{theorem}
Next, we consider the goodness-of-fit test:
\begin{align}
    \left\{
    \begin{array}{ll}
    H_0: {\bf{\Sigma}}={\bf{\Sigma}}(\theta),\\
    H_1: {\bf{\Sigma}}\neq{\bf{\Sigma}}(\theta).
    \end{array}
    \right. \label{test}
\end{align}
For ${\bf{\Sigma}}\in\mathcal{M}_p^+$, we set
\begin{align*}
    {\bf{L}}_n({\bf{\Sigma}})=\exp\bigl\{{\bf{H}}_n({\bf{\Sigma}})\bigr\},
\end{align*}
where 
\begin{align*}
    {\bf{H}}_n({\bf{\Sigma}})
    =-\frac{1}{2}\sum_{i=1}^n\log\det {\bf{\Sigma}}{\bf{1}}_{\{|\Delta_{i}^n X|\leq h_n^{\rho}\}}-\frac{1}{2h_n}\sum_{i=1}^n(\Delta_{i}^n X)^{\top}{\bf{\Sigma}}^{-1}(\Delta_{i}^n X){\bf{1}}_{\{|\Delta X_{i}^n|\leq h_n^{\rho}\}}.
\end{align*}
Let $J_n=\bigl\{\hat{\bf{\Sigma}}_n\ {\rm{is\ positive\ definite}}\bigr\}$. 
On $J_n$, we define the quasi-likelihood ratio as 
\begin{align*}
    \lambda_n=\frac{\sup_{\theta\in\Theta}{\bf{L}}_n({\bf{\Sigma}}(\theta))}{\sup_{{{\bf{\Sigma}}\in\mathcal{M}_p^{+}
    }}{\bf{L}}_n({\bf{\Sigma}})}. 
\end{align*}
Note that $\tilde{N}_n=N_n$ on $J_n$. Since 
\begin{align*}
    &\quad\ \frac{1}{2h_n}\sum_{i=1}^n(\Delta_{i}^n X)^{\top}{\bf{\Sigma}}^{-1}(\Delta_{i}^n X){\bf{1}}_{\{|\Delta X_{i}^n|\leq h_n^{\rho}\}}\\
    &=\frac{N_n}{2}\times \frac{1}{N_nh_n}\sum_{i=1}^n\tr\Bigl\{(\Delta_{i}^n X)^{\top}{\bf{\Sigma}}^{-1}(\Delta_{i}^n X){\bf{1}}_{\{|\Delta_{i}^n X|\leq Dh_n^{\rho}\}}\Bigr\}=\frac{N_n}{2}\tr\bigl\{{\bf{\Sigma}}^{-1}\hat{\bf{\Sigma}}_n\bigr\}
\end{align*}
on $J_n$, we have
\begin{align*}
    {\bf{H}}_n({\bf{\Sigma}})
    &=-\frac{N_n}{2}\log\det {\bf{\Sigma}}-\frac{N_n}{2}\tr\bigl\{{\bf{\Sigma}}^{-1}\hat{\bf{\Sigma}}_n\bigr\}
\end{align*}
on $J_n$. Furthermore, ${\bf{H}}_n({\bf{\Sigma}})$ has a maximum value 
\begin{align*}
     {\bf{H}}_n(\hat{\bf{\Sigma}}_n)
    &=-\frac{N_n}{2}\log\det \hat{\bf{\Sigma}}_n-\frac{N_np}{2}
\end{align*}
at ${\bf{\Sigma}}=\hat{{\bf{\Sigma}}}_n$ on $J_n$, so that
\begin{align*}
    -2\log\lambda_n&=-2\sup_{\theta\in\Theta}{\bf{H}}_{n}({\bf{\Sigma}}(\theta))+2
    \sup_{{{\bf{\Sigma}}\in\mathcal{M}}_p^+} {\bf{H}}_{n}({\bf{\Sigma}})\\
    &=-2{\bf{H}}_{n}({\bf{\Sigma}}(\hat{\theta}_n))+2{\bf{H}}_n(\hat{\bf{\Sigma}}_n)\\
    &=N_n\log\det{\bf{\Sigma}}(\hat{\theta}_n)-N_n\log\det\hat{\bf{\Sigma}}_n+N_n \tr\bigl\{{\bf{\Sigma}}(\hat{\theta}_n)^{-1}\hat{\bf{\Sigma}}_n\bigr\}-N_np
\end{align*}
on $J_n$. Therefore, we define the quasi-likelihood ratio test statistics as follows:
\begin{align*}
    {\bf{T}}_n&={\bf{T}}_n(\hat{\theta}_n)\\
    &=N_n\log\det {\bf{\Sigma}}(\hat{{\bf{\theta}}}_n)-N_n\log\det \tilde{\bf{\Sigma}}_n+N_n
    \tr\bigl\{{\bf{\Sigma}}(\hat{\theta}_n)^{-1}\tilde{\bf{\Sigma}}_n\bigr\}-N_np,
\end{align*}
where 
\begin{align*}
   \tilde{\bf{\Sigma}}_n=\left\{
    \begin{array}{ll}
    \hat{\bf{\Sigma}}_n & (on\ J_n),\\
    \mathbb{I}_p & (on\ J_n^c).
    \end{array}\right. 
\end{align*}
Note that ${\bf{T}}_n$ is still well-defined on $J_n^c$. See also Kusano and Uchida \cite{Kusano(JJSD)} for the details of the goodness-of-fit test (\ref{test}). Set the true value of ${\bf{\Sigma}}$ under $H_0$ as ${\bf{\Sigma}}(\theta_0)$, where $\theta_0\in{\rm{Int}}(\Theta)$. Under $H_0$ (i.e., under ${\bf{P}}_{\theta_0}$), the following asymptotic property is shown.
\begin{theorem}\label{testtheorem1}
Under {\rm{\textbf{[A1]}}}-{\rm{\textbf{[A4]}}} and {\rm{\textbf{[B1]}}}, as $n\longrightarrow\infty$,
\begin{align*}
    {\bf{T}}_n\overset{d}{\longrightarrow}\chi^2_{\bar{p}-q}
\end{align*}
under $H_0$.
\end{theorem}
By Theorem \ref{testtheorem1}, we can consider the test of asymptotic significance level $\alpha\in(0,1)$ with the following rejection region:
\begin{align*}
    \bigl\{{\bf{T}}_n>\chi^2_{\bar{p}-q}(\alpha)\bigr\}.
\end{align*}
Next, we study the consistency of the test. Fix ${\bf{\Sigma}}^*$ as the true value of ${\bf{\Sigma}}$ under $H_1$. Note that ${\bf{\Sigma}}^*\neq {\bf{\Sigma}}(\theta)$ for any $\theta\in\Theta$. Set the optimal parameter $\theta^*$ as 
\begin{align*}
    {\bf{U}}(\theta^*)=\inf_{\theta\in\Theta}{\bf{U}}(\theta),
\end{align*}
where
\begin{align*}
    {\bf{U}}(\theta)=\log\det {\bf{\Sigma}}(\theta)-\log\det {\bf{\Sigma}}^*+
    \tr\bigl\{{\bf{\Sigma}}(\theta)^{-1}{\bf{\Sigma}}^*\bigr\}-p.
\end{align*}
In addition, to prove the consistency of the test, the following assumption is made.
\begin{enumerate}
    \item[\bf{[B2]}]
    ${\bf{U}}(\theta)={\bf{U}}(\theta^*)\Longrightarrow \theta=\theta^*$.
\end{enumerate}
Under $H_1$ (i.e., under ${\bf{P}}_{{\bf{\Sigma}}^*})$, the following theorem holds.
\begin{theorem}\label{testtheorem2}
Under {\rm{\textbf{[A1]}}}-{\rm{\textbf{[A4]}}} and {\rm{\textbf{[B2]}}}, as $n\longrightarrow\infty$,
\begin{align*}
    {\bf{P}}\Bigl({\bf{T}}_n>\chi^2_{\bar{p}-q}(\alpha)\Bigr)\longrightarrow 1
\end{align*}
under $H_1$.
\end{theorem}
\section{Numerical Simulations}\label{Simulation}
\subsection{True model}
The four-dimensional observable process $\{X^0_{1,t}\}_{t\geq 0}$ is defined by the true factor model as follows:
\begin{align*}
    X_{1,t}^0&={\bf{\Lambda}}_{1,0}\xi_t^0+\delta_t^0,
\end{align*}
where $\{\xi_{t}^0\}_{t\geq 0}$ and $\{\delta^0_t\}_{t\geq 0}$ are one and four-dimensional latent processes, respectively, 
and
\begin{align*}
    {\bf{\Lambda}}_{1,0}=\begin{pmatrix}
    1 & 0.7 & 1.3 & 0.9
    \end{pmatrix}^{\top}.
\end{align*}
The stochastic process $\{\xi_t^0\}_{t\geq 0}$ satisfies the following one-dimensional L\'{e}vy-OU process:
\begin{align*}
    d\xi^{0}_{t}&=-2\bigl(\xi^{0}_{t-}-1)dt+1.2dW_{1,t}+\int_{\mathbb{R}\backslash \{0\}}zp_1(dt,dz), \quad 
    \xi^{0}_{0}=1,
\end{align*}
where $\{W_{1,t}\}_{t\geq 0}$ is a one-dimensional standard Wiener process, and $p_1(dt,dz)$ is a Poisson random measure on $\mathbb{R}^+\times \mathbb{R}\backslash\{0\}$ with the jump density $f_{1}(z)=3g(z|0,5)$. Here, $g(z|\mu,\sigma^2)$ denotes the probability density function of the normal distribution with mean $\mu\in\mathbb{R}$ and variance $\sigma^2>0$, i.e.,
\begin{align*}
    g(z|\mu,\sigma^2)=\frac{1}{\sqrt{2\pi\sigma^2}}\exp{\biggl\{-\frac{(z-\mu)^2}{2\sigma^2}\biggr\}}.
\end{align*}
The stochastic process $\{\delta_t^0\}_{t\geq 0}$ is defined by the following four-dimensional L\'{e}vy-OU process:
\begin{align*}
    d\delta^{0(1)}_{t}&=-0.8\delta^{0(1)}_{t-}dt+1.6 dW^{(1)}_{2,t}
    +\int_{\mathbb{R}\backslash \{0\}}z p_{2,1}(dt,dz), \quad 
    \delta^{0(1)}_{0}=0,\\
    d\delta^{0(2)}_{t}&=-0.5\delta^{0(2)}_{t-}dt+0.7dW^{(2)}_{2,t}
    +\int_{\mathbb{R}\backslash \{0\}}zp_{2,2}(dt,dz), \quad 
    \delta^{0(2)}_{0}=0,\\
    d\delta^{0(3)}_{t}&=-0.9\delta^{0(3)}_{t-}dt+1.2dW^{(3)}_{2,t}
    +\int_{\mathbb{R}\backslash \{0\}}zp_{2,3}(dt,dz), \quad 
    \delta^{0(3)}_{0}=0
\end{align*}
and
\begin{align*}
    d\delta^{0(4)}_{t}&=-0.7\delta^{0(4)}_{t-}dt+0.9dW^{(4)}_{2,t}
    +\int_{\mathbb{R}\backslash \{0\}}zp_{2,4}(dt,dz), \quad 
    \delta^{0(4)}_{0}=0,
\end{align*}
where $\{W_{2,t}\}_{t\geq 0}$ is a four-dimensional standard Wiener process, and $p_{2,i}(dt,dz)$ for $i=1,2,3,4$ are Poisson random measures on $\mathbb{R}^+\times \mathbb{R}\backslash\{0\}$ with the jump densities $f_{2,1}(z)=2g(z|0,3)$, $f_{2,2}(z)=g(z|0,2)$, $f_{2,3}(z)=g(z|0,3)$, and $f_{2,4}(z)=2g(z|0,2)$, respectively. The eight-dimensional observable process $\{X_{2,t}^0\}_{t\geq 0}$ satisfies the following true factor model:
\begin{align*}
    X_{2,t}^0={\bf{\Lambda}}_{2,0}\eta_t^0+\varepsilon_t^0,
\end{align*}
where $\{\eta^0_{t}\}_{t\geq 0}$ and $\{\varepsilon^0_t\}_{t\geq 0}$ are two and eight-dimensional latent processes, respectively,  and 
\begin{align*}
    {\bf{\Lambda}}_{2,0}=\begin{pmatrix}
    1 & 0.8 & 1.4 & 1.2 & 0 & 0 & 0 & 0\\
    0 & 0 & 0 & 0 & 1 & 0.6 & 1.3 & 0.9
    \end{pmatrix}^{\top}.
\end{align*}
The relationship between $\{\xi^0_{t}\}_{t\geq 0}$ and $\{\eta^0_{t}\}_{t\geq 0}$ is expressed as follows:
\begin{align*}
    \eta_t^0={\bf{\Gamma}}_0\xi_t^0+\zeta_t^0,
\end{align*}
where $\{\zeta^0_t\}_{t\geq 0}$ is a two-dimensional latent process and
\begin{align*}
    {\bf{\Gamma}}_0=\begin{pmatrix}
    0.7 & -0.8
    \end{pmatrix}^{\top}.
\end{align*}
The stochastic process $\{\varepsilon^0_t\}_{t\geq 0}$ satisfies the following eight-dimensional L\'{e}vy-OU process:
\begin{align*}
    d\varepsilon^{0(1)}_{t}&=-0.8\varepsilon^{0(1)}_{t-}dt+0.9dW^{(1)}_{3,t}
    +\int_{\mathbb{R}\backslash \{0\}}z p_{3,1}(dt,dz),\quad 
    \varepsilon^{0(1)}_{0}=0,\\
    d\varepsilon^{0(2)}_{t}&=-1.5\varepsilon^{0(2)}_{t-}dt+1.2dW^{(2)}_{3,t}
    +\int_{\mathbb{R}\backslash \{0\}}zp_{3,2}(dt,dz),\quad 
    \varepsilon^{0(2)}_{0}=0,\\
    d\varepsilon^{0(3)}_{t}&=-0.9\varepsilon^{0(3)}_{t-}dt+0.8dW^{(3)}_{3,t}
    +\int_{\mathbb{R}\backslash \{0\}}zp_{3,3}(dt,dz),\quad 
    \varepsilon^{0(3)}_{0}=0,\\
    d\varepsilon^{0(4)}_{t}&=-0.7\varepsilon^{0(4)}_{t-}dt+1.1
    dW^{(4)}_{3,t}+\int_{\mathbb{R}\backslash \{0\}}zp_{3,4}(dt,dz),\quad 
    \varepsilon^{0(4)}_{0}=0,\\
    d\varepsilon^{0(5)}_{t}&=-1.2\varepsilon^{0(5)}_{t-}dt
    +1.5dW^{(5)}_{3,t}+\int_{\mathbb{R}\backslash \{0\}}zp_{3,5}(dt,dz),\quad 
    \varepsilon^{0(5)}_{0}=0,\\
    d\varepsilon^{0(6)}_{t}&=-0.5\varepsilon^{0(6)}_{t-}dt
    +1.3dW^{(6)}_{3,t}+\int_{\mathbb{R}\backslash \{0\}}zp_{3,6}(dt,dz),\quad 
    \varepsilon^{0(6)}_{0}=0,\\
    d\varepsilon^{0(7)}_{t}&=-1.3\varepsilon^{0(7)}_{t-}dt
    +0.7dW^{(7)}_{3,t}+\int_{\mathbb{R}\backslash \{0\}}zp_{3,7}(dt,dz),\quad 
    \varepsilon^{0(7)}_{0}=0
\end{align*}
and
\begin{align*}
    d\varepsilon^{0(8)}_{t}&=-0.6\varepsilon^{0(8)}_{t-}dt
    +1.4dW^{(8)}_{3,t}+\int_{\mathbb{R}\backslash \{0\}}zp_{3,8}(dt,dz),\quad 
    \varepsilon^{0(8)}_{0}=0,
\end{align*}
where $\{W_{3,t}\}_{t\geq 0}$ is an eight-dimensional standard Wiener process, and 
$p_{3,i}(dt,dz)$ for $i=1,\ldots,8$ are Poisson random measures on $\mathbb{R}^+\times \mathbb{R}\backslash\{0\}$ with the jump densities $f_{3,1}(z)=2g(z|0,2)$, $f_{3,2}(z)=g(z|0,3)$, $f_{3,3}(z)=g(z|0,2)$, $f_{3,4}(z)=2g(z|0,3)$, $f_{3,5}(z)=2g(z|0,3)$, $f_{3,6}(z)=g(z|0,3)$, $f_{3,7}(z)=g(z|0,2)$, and $f_{3,8}(z)=2g(z|0,3)$, respectively. The stochastic process $\{\zeta^0_t\}_{t\geq 0}$ is defined by the following two-dimensional L\'{e}vy-OU process:
\begin{align*}
    d\zeta^{0(1)}_{t}&=-0.8\zeta^{0(1)}_{t-}dt+0.9dW^{(1)}_{4,t}
    +\int_{\mathbb{R}\backslash \{0\}}z p_{4,1}(dt,dz), \quad 
    \zeta^{0(1)}_{0}=0
\end{align*}
and
\begin{align*}
    d\zeta^{0(2)}_{t}&=-1.4\zeta^{0(2)}_{t-}dt+1.1dW^{(2)}_{4,t}
    +\int_{\mathbb{R}\backslash \{0\}}zp_{4,2}(dt,dz), \quad 
    \zeta^{0(2)}_{0}=0,
\end{align*}
where $\{W_{4,t}\}_{t\geq 0}$  is a two-dimensional standard Wiener process, and $p_{4,1}(dt,dz)$ and $p_{4,2}(dt,dz)$ are Poisson random measures on $\mathbb{R}^+\times \mathbb{R}\backslash\{0\}$ with the jump densities $f_{4,1}(z)=2g(z|0,2)$ and $f_{4,2}(z)=g(z|0,3)$, respectively. Figure \ref{truefigure} is the path diagram of the true model at time $t$.
\begin{figure}[h]
    \centering
    \includegraphics[width=0.8\columnwidth]{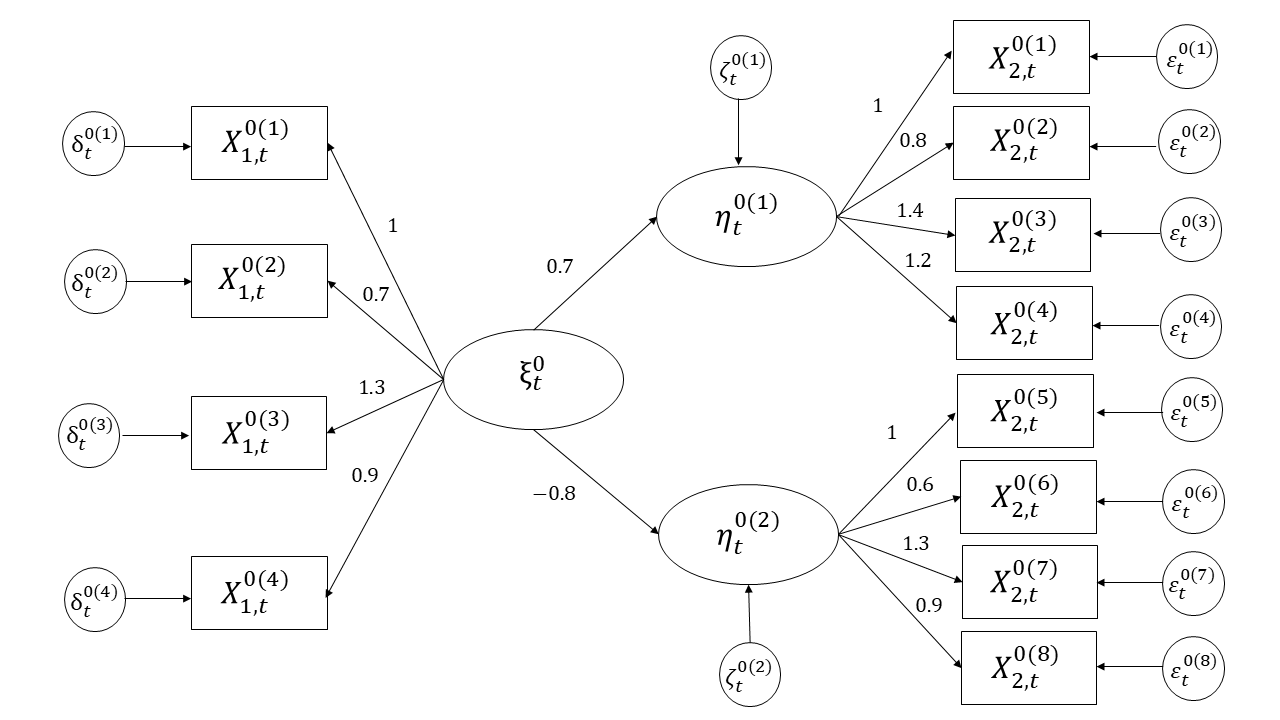}
    \caption{Path diagram of the true model at time $t$.}
    \label{truefigure}
\end{figure}
\subsection{Correctly specified model}
Set $p_1=4$, $p_2=8$, $k_1=1$, $k_2=2$, and $q=26$. We assume the loading matrices as 
\begin{align*}
    {\bf{\Lambda}}_{1}^{\theta}=\begin{pmatrix}
    1 & \theta^{(1)} & \theta^{(2)} & \theta^{(3)}
    \end{pmatrix}^{\top}
\end{align*}
and
\begin{align*}
    {\bf{\Lambda}}_{2}^{\theta}=\begin{pmatrix}
    1 & \theta^{(4)} & \theta^{(5)} & \theta^{(6)} & 0 & 0 & 0 & 0\\
    0 & 0 & 0 & 0 & 1 & \theta^{(7)} & \theta^{(8)} & \theta^{(9)}
    \end{pmatrix}^{\top},\quad {\bf{\Gamma}}^{\theta}=\begin{pmatrix}
    \theta^{(10)}\\
    \theta^{(11)}
    \end{pmatrix},\quad {\bf{\Psi}}^{\theta}=\mathbb{I}_2
\end{align*}
where $\theta^{(i)}$ for $i=1,\ldots,11$ are not zero. In addition, we suppose the volatility matrices as 
\begin{align*}
    {\bf{\Sigma}}_{\xi\xi}^{\theta}=\theta^{(12)},\quad 
    {\bf{\Sigma}}_{\delta\delta}^{\theta}=\Diag\bigl(\theta^{(13)}, \theta^{(14)},\theta^{(15)},\theta^{(16)}\bigr)^{\top}
\end{align*}
and
\begin{align*}
    {\bf{\Sigma}}_{\varepsilon\varepsilon}^{\theta}=\Diag\bigl( \theta^{(17)},\theta^{(18)},\theta^{(19)}, \theta^{(20)}, \theta^{(21)}, \theta^{(22)}, \theta^{(23)}, \theta^{(24)}\bigr)^{\top},\quad {\bf{\Sigma}}_{\zeta\zeta}^{\theta}=\Diag\bigl(\theta^{(25)}, \theta^{(26)}\bigr)^{\top},
\end{align*}
where $\theta^{(i)}$ for $i=12,\ldots,26$ are positive. This model is a correctly specified model since 
\begin{align*}
    {\bf{\Sigma}}_0={\bf{\Sigma}}(\theta_0),
\end{align*}
where 
\begin{align*}
    \theta_0&=\bigl(0.7, 1.3, 0.9, 0.8, 1.4, 1.2, 0.6, 1.3, 0.9, 0.7, -0.8, 1.44, 2.56, 0.49,  \\
    &\qquad\qquad 1.44, 0.81, 0.81, 1.44, 0.64, 1.21, 2.25, 1.69, 0.49, 1.96, 0.81, 1.21\bigr)^{\top}.
\end{align*}
Figure \ref{corfigure} is the path diagram of the correctly specified model at time $t$.
\begin{figure}[h]
    \centering
    \includegraphics[width=0.8\columnwidth]{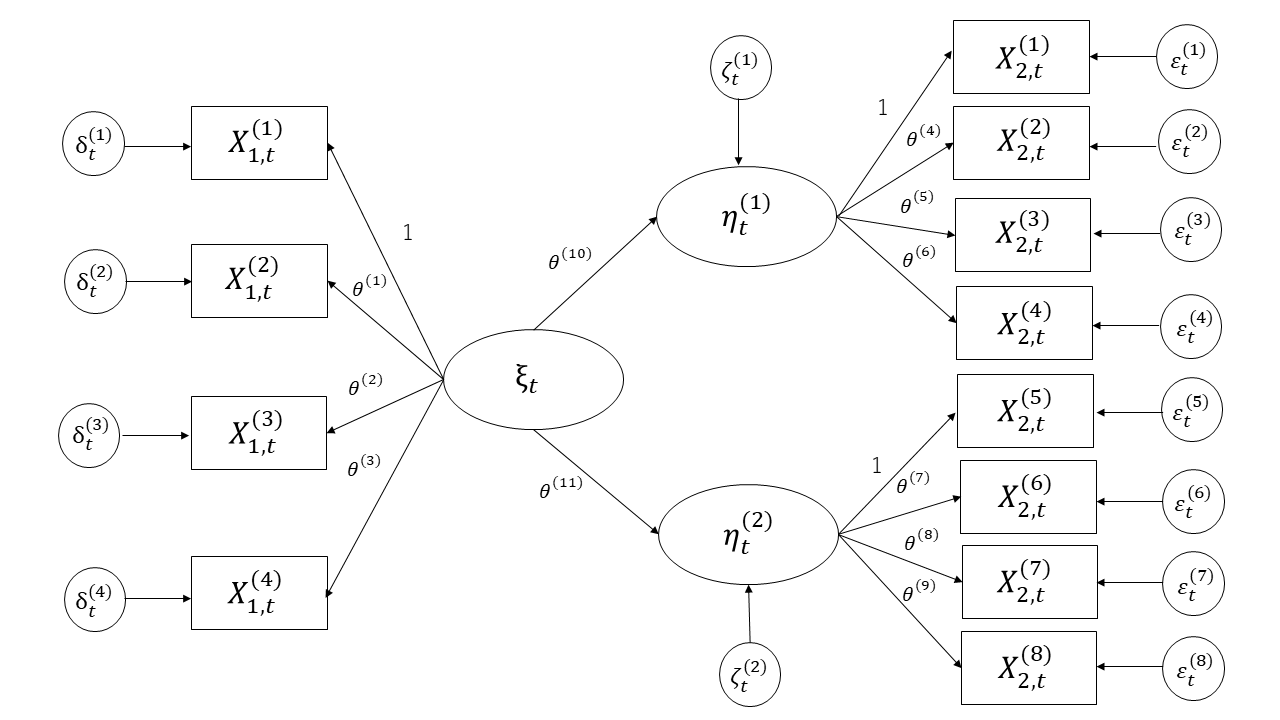}
    \caption{Path diagram of the correctly specified model at time $t$.}
    \label{corfigure}
\end{figure}
\subsection{Misspecified model}
Set $p_1=4$, $p_2=8$, $k_1=1$, $k_2=1$, and $q=25$. In this model, the loading matrices are defined by
\begin{align*}
    {\bf{\Lambda}}_{1}^{\theta}=\begin{pmatrix}
    1 & \theta^{(1)} & \theta^{(2)} & \theta^{(3)}
    \end{pmatrix}^{\top}
\end{align*}
and
\begin{align*}
    {\bf{\Lambda}}_{2}^{\theta}=\begin{pmatrix}
    1 & \theta^{(4)} & \theta^{(5)} & \theta^{(6)} &\theta^{(7)} & \theta^{(8)} & \theta^{(9)} & \theta^{(10)}
    \end{pmatrix}^{\top},\quad {\bf{\Gamma}}^{\theta}=
    \theta^{(11)}
\end{align*}
where $\theta^{(i)}$ for $i=1,\ldots,11$ are not zero. In addition, the volatility matrices are set as follows:
\begin{align*}
    {\bf{\Sigma}}_{\xi\xi}^{\theta}=\theta^{(12)},\quad 
    {\bf{\Sigma}}_{\delta\delta}^{\theta}=\Diag\bigl(\theta^{(13)}, \theta^{(14)},\theta^{(15)},\theta^{(16)} \bigr)^{\top}
\end{align*}
and
\begin{align*}
    {\bf{\Sigma}}_{\varepsilon\varepsilon}^{\theta}=\Diag\bigl( \theta^{(17)},\theta^{(18)},\theta^{(19)}, \theta^{(20)}, \theta^{(21)}, \theta^{(22)}, \theta^{(23)}, \theta^{(24)}\bigr)^{\top},\quad {\bf{\Sigma}}_{\zeta\zeta}^{\theta}=\theta^{(25)},
\end{align*}
where $\theta^{(i)}$ for $i=12,\ldots,25$ are positive. Figure \ref{missfigure} is the path diagram of the misspecified model at time $t$.
\begin{figure}[h]
    \centering
    \includegraphics[width=0.8\columnwidth]{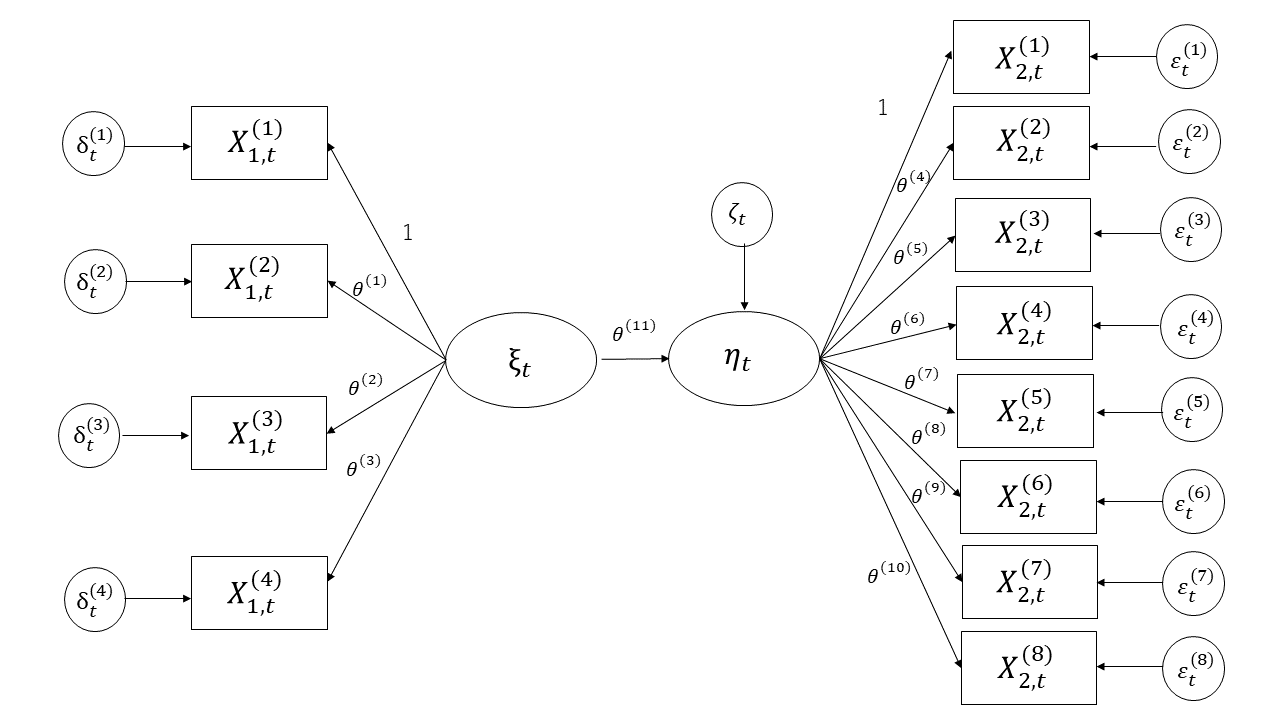}
    \caption{Path diagram of the misspecified model at time $t$.}
    \label{missfigure}
\end{figure}
\subsection{Simulation results}
In this simulation, we set $n=10^5$ and $T=1$, and generated $10000$ independent sample paths from the true model.  To maximize ${\bf{H}}_n(\theta)$, we used optim() with the BFGS method in R language, 
and chose  the true value $\theta_0$ as an initial value.
Let $D=10$ and $\rho=0.4$. 
\subsubsection{Correctly specified model}
Tables \ref{sigmatable} and \ref{thetatable} show the sample means and the sample standard deviations (S.D.s) of $\hat{\bf{\Sigma}}_n$ and $\hat{\theta}_n$, and Figures \ref{sigmafigure} and \ref{thetafigure} show the histograms, the Q-Q plots, and the empirical distributions of $\sqrt{n}((\hat{\bf{\Sigma}}_n)_{11}-({\bf{\Sigma}}_0)_{11})$ and $\sqrt{n}(\hat{\theta}_n^{(1)}-{\theta}_0^{(1)})$, which imply that $\hat{\bf{\Sigma}}_n$ and $\hat{\theta}_n$ have consistency and asymptotic normality. Figure \ref{testfigure} shows the histogram, the Q-Q plot, and the empirical distribution of ${\bf{T}}_n$. 
It seems from Figure \ref{testfigure} that the asymptotic distribution of ${\bf{T}}_n$ is the chi-squared distribution with degrees of freedom $52$.
Hence, we can see that Theorems \ref{Qtheorem}-\ref{testtheorem1} hold true for this example. 
\subsubsection{Misspecified model}
Table \ref{TableH1} shows the number of rejections of the quasi-likelihood ratio test in the misspecified model. Table \ref{TableH1} implies that the test has consistency.
\begin{longtable}{ccccccc}
    \\
    & $(\hat{{\bf{\Sigma}}}_{n})_{(1,1)}$ & $(\hat{{\bf{\Sigma}}}_{n})_{(1,2)}$ &
    $(\hat{{\bf{\Sigma}}}_{n})_{(1,3)}$ &
    $(\hat{{\bf{\Sigma}}}_{n})_{(1,4)}$ & 
    $(\hat{{\bf{\Sigma}}}_{n})_{(1,5)}$ &
    $(\hat{{\bf{\Sigma}}}_{n})_{(1,6)}$ 
    \\ \hline
    True value & 4.000 & 1.008 & 1.872 & 1.296 & 1.008 & 0.806
    \\
    Mean & 4.001 & 1.008 & 1.873 & 1.297 & 1.009 & 0.807\\
    S.D. & 0.018 & 0.008 & 0.014 & 0.010 & 0.010 & 0.010
    \\
    $(\hat{{\bf{\Sigma}}}_{n})_{(1,7)}$
    & $(\hat{{\bf{\Sigma}}}_{n})_{(1,8)}$ & $(\hat{{\bf{\Sigma}}}_{n})_{(1,9)}$  & $(\hat{{\bf{\Sigma}}}_{n})_{(1,10)}$ & $(\hat{{\bf{\Sigma}}}_{n})_{(1,11)}$ & 
    $(\hat{{\bf{\Sigma}}}_{n})_{(1,12)}$ &
    $(\hat{{\bf{\Sigma}}}_{n})_{(2,2)}$ \\ \hline
    1.411 & 1.210 & -1.152 & -0.691 & -1.498 & -1.037 & 1.196 \\
    1.412 & 1.210 & -1.153 & -0.691 & -1.498 & -1.037 & 1.196\\
    0.013 & 0.012 & 0.014 & 0.010 & 0.014 & 0.013 & 0.005
    \\
    $(\hat{{\bf{\Sigma}}}_{n})_{(2,3)}$
    & $(\hat{{\bf{\Sigma}}}_{n})_{(2,4)}$ & $(\hat{{\bf{\Sigma}}}_{n})_{(2,5)}$  & $(\hat{{\bf{\Sigma}}}_{n})_{(2,6)}$ & $(\hat{{\bf{\Sigma}}}_{n})_{(2,7)}$ & 
    $(\hat{{\bf{\Sigma}}}_{n})_{(2,8)}$ &
    $(\hat{{\bf{\Sigma}}}_{n})_{(2,9)}$ \\ \hline
    1.310 & 0.907 & 0.706 & 0.564 & 0.988 & 0.847 & -0.806\\
    1.311 & 0.908 & 0.706 & 0.565 & 0.988 & 0.847 & -0.807
    \\
    0.008 & 0.006 & 0.006 & 0.006 & 0.007 & 0.007 & 0.008
    \\
    $(\hat{{\bf{\Sigma}}}_{n})_{(2,10)}$
    & $(\hat{{\bf{\Sigma}}}_{n})_{(2,11)}$ & $(\hat{{\bf{\Sigma}}}_{n})_{(2,12)}$  & $(\hat{{\bf{\Sigma}}}_{n})_{(3,3)}$ & $(\hat{{\bf{\Sigma}}}_{n})_{(3,4)}$ & 
    $(\hat{{\bf{\Sigma}}}_{n})_{(3,5)}$ &
    $(\hat{{\bf{\Sigma}}}_{n})_{(3,6)}$ \\ \hline
    -0.484 & -1.048 & -0.726 & 3.874 & 1.685 & 1.310 & 1.048
     \\
    -0.484 & -1.049 & -0.726 & 3.875 & 1.686 & 1.311 & 1.049
     \\
    0.006 & 0.008 & 0.007 & 0.018 & 0.010 & 0.010 & 0.010
    \\
    $(\hat{{\bf{\Sigma}}}_{n})_{(3,7)}$
    & $(\hat{{\bf{\Sigma}}}_{n})_{(3,8)}$ & $(\hat{{\bf{\Sigma}}}_{n})_{(3,9)}$  & $(\hat{{\bf{\Sigma}}}_{n})_{(3,10)}$ & $(\hat{{\bf{\Sigma}}}_{n})_{(3,11)}$ & 
    $(\hat{{\bf{\Sigma}}}_{n})_{(3,12)}$ &
    $(\hat{{\bf{\Sigma}}}_{n})_{(4,4)}$ \\ \hline
    1.835 & 1.572 & -1.498 & -0.899 & -1.947 & -1.348 & 1.976
     \\
    1.835 & 1.573 & -1.498 & -0.899 & -1.948 & -1.348 & 1.977
    \\
    0.013 & 0.013 & 0.014 & 0.010 & 0.014 & 0.013 & 0.009
    \\
    $(\hat{{\bf{\Sigma}}}_{n})_{(4,5)}$
    & $(\hat{{\bf{\Sigma}}}_{n})_{(4,6)}$ & $(\hat{{\bf{\Sigma}}}_{n})_{(4,7)}$  & $(\hat{{\bf{\Sigma}}}_{n})_{(4,8)}$ & $(\hat{{\bf{\Sigma}}}_{n})_{(4,9)}$ & 
    $(\hat{{\bf{\Sigma}}}_{n})_{(4,10)}$ &
    $(\hat{{\bf{\Sigma}}}_{n})_{(4,11)}$ \\ \hline
    0.907 & 0.726 & 1.270 & 1.089 & -1.037 & -0.622 & -1.348\\
    0.908 & 0.726 & 1.271 & 1.089 & -1.037 & -0.622 & -1.348
   \\
    0.007 & 0.007 & 0.009 & 0.009 & 0.010 & 0.007 & 0.010
    \\
    $(\hat{{\bf{\Sigma}}}_{n})_{(4,12)}$
    & $(\hat{{\bf{\Sigma}}}_{n})_{(5,5)}$ & $(\hat{{\bf{\Sigma}}}_{n})_{(5,6)}$  & $(\hat{{\bf{\Sigma}}}_{n})_{(5,7)}$ & $(\hat{{\bf{\Sigma}}}_{n})_{(5,8)}$ & 
    $(\hat{{\bf{\Sigma}}}_{n})_{(5,9)}$ &
    $(\hat{{\bf{\Sigma}}}_{n})_{(5,10)}$ \\ \hline
    -0.933 & 2.326 & 1.212 & 2.122 & 1.819 & -0.806 & -0.484
    \\
    -0.934 & 2.326 & 1.213 & 2.122 & 1.819 & -0.807 & -0.484
     \\
    0.009 & 0.011 & 0.008 & 0.011 & 0.011 & 0.010 & 0.008
    \\
    $(\hat{{\bf{\Sigma}}}_{n})_{(5,11)}$
    & $(\hat{{\bf{\Sigma}}}_{n})_{(5,12)}$ & $(\hat{{\bf{\Sigma}}}_{n})_{(6,6)}$  & $(\hat{{\bf{\Sigma}}}_{n})_{(6,7)}$ & $(\hat{{\bf{\Sigma}}}_{n})_{(6,8)}$ & 
    $(\hat{{\bf{\Sigma}}}_{n})_{(6,9)}$ &
    $(\hat{{\bf{\Sigma}}}_{n})_{(6,10)}$ \\ \hline
    -1.048 & -0.726 & 2.410 & 1.697 & 1.455 & -0.645 & -0.387
    \\
    -1.049 & -0.726 & 2.410 & 1.698 & 1.455 & -0.645 & -0.387
    \\
    0.010 & 0.010 & 0.011 & 0.011 & 0.010 & 0.010
    & 0.008 \\
    $(\hat{{\bf{\Sigma}}}_{n})_{(6,11)}$
    & $(\hat{{\bf{\Sigma}}}_{n})_{(6,12)}$ & $(\hat{{\bf{\Sigma}}}_{n})_{(7,7)}$  & $(\hat{{\bf{\Sigma}}}_{n})_{(7,8)}$ & $(\hat{{\bf{\Sigma}}}_{n})_{(7,9)}$ & 
    $(\hat{{\bf{\Sigma}}}_{n})_{(7,10)}$ &
    $(\hat{{\bf{\Sigma}}}_{n})_{(7,11)}$ \\ 
    \hline
    -0.839 & -0.581 & 3.611 & 2.546 & -1.129 & -0.677 & -1.468 \\
    -0.839 & -0.581 & 3.611 & 2.547 & -1.130 & -0.678 & 
    -1.468 \\
    0.010 & 0.010 & 0.016 & 0.014 & 0.013 & 0.010 & 0.013\\
    $(\hat{{\bf{\Sigma}}}_{n})_{(7,12)}$
    & $(\hat{{\bf{\Sigma}}}_{n})_{(8,8)}$ & $(\hat{{\bf{\Sigma}}}_{n})_{(8,9)}$  & $(\hat{{\bf{\Sigma}}}_{n})_{(8,10)}$ & $(\hat{{\bf{\Sigma}}}_{n})_{(8,11)}$ & 
    $(\hat{{\bf{\Sigma}}}_{n})_{(8,12)}$ &
    $(\hat{{\bf{\Sigma}}}_{n})_{(9,9)}$ \\
    \hline
    -1.016 & 3.392 & -0.968 & -0.581 & -1.258 & -0.871 & 4.382 \\
    -1.017 & 3.393 & -0.968 & -0.581 & -1.259 & -0.871 & 4.382 \\
    0.012 & 0.015 & 0.012 & 0.009 & 0.012 & 0.012 & 0.020\\
    $(\hat{{\bf{\Sigma}}}_{n})_{(9,10)}$
    & $(\hat{{\bf{\Sigma}}}_{n})_{(9,11)}$ & $(\hat{{\bf{\Sigma}}}_{n})_{(9,12)}$  & $(\hat{{\bf{\Sigma}}}_{n})_{(10,10)}$ & $(\hat{{\bf{\Sigma}}}_{n})_{(10,11)}$ & 
    $(\hat{{\bf{\Sigma}}}_{n})_{(10,12)}$ &
    $(\hat{{\bf{\Sigma}}}_{n})_{(11,11)}$ \\
    \hline
    1.279 & 2.771 & 1.918 & 2.457 & 1.663 & 1.151 & 4.092 \\
    1.279 & 2.772 & 1.919 & 2.457 & 1.663 & 1.151 & 4.093 \\  
    0.011 & 0.016 & 0.014 & 0.011 & 0.011 & 0.010 & 0.018\\
    $(\hat{{\bf{\Sigma}}}_{n})_{(11,12)}$
    & $(\hat{{\bf{\Sigma}}}_{n})_{(12,12)}$ &  \\
    \hline
    2.494 & 3.687\\
    2.494 & 3.687 \\
    0.015 & 0.016\\
    \caption{Sample mean and sample standard deviation (S.D.) of $\hat{\bf{\Sigma}}_n$.}\label{sigmatable}
\end{longtable}
\begin{longtable}[h]{cccccccccc}
    \\
    & $\hat{\theta}_n^{(1)}$ & 
    $\hat{\theta}_n^{(2)}$
    &
    $\hat{\theta}_n^{(3)}$ &
    $\hat{\theta}_n^{(4)}$ & 
    $\hat{\theta}_n^{(5)}$ &
    $\hat{\theta}_n^{(6)}$ &
    $\hat{\theta}_n^{(7)}$ &
    $\hat{\theta}_n^{(8)}$ &
    $\hat{\theta}_n^{(9)}$ 
    \\ \hline
    True value & 0.700 & 1.300 & 0.900 & 0.800  & 1.400 & 1.200 & 0.600  & 1.300 & 0.900\\
    Mean & 0.700 & 1.300 & 0.900 & 0.800 & 1.400 & 1.200 & 0.600 & 1.300 &  0.900 \\
    S.D. & 0.004 & 0.007 & 0.005 & 0.004 & 0.004 & 0.004 & 0.004 & 0.005 & 0.005 \\
    $\hat{\theta}_n^{(10)}$ & 
    $\hat{\theta}_n^{(11)}$
    &
    $\hat{\theta}_n^{(12)}$ &
    $\hat{\theta}_n^{(13)}$ & 
    $\hat{\theta}_n^{(14)}$ &
    $\hat{\theta}_n^{(15)}$ &
    $\hat{\theta}_n^{(16)}$ &
    $\hat{\theta}_n^{(17)}$ &
    $\hat{\theta}_n^{(18)}$ &
    $\hat{\theta}_n^{(19)}$
    \\ \hline
    0.700 & -0.800 & 1.440 & 2.560 & 0.490 & 1.440 & 0.810 & 0.810 & 1.440 & 0.640 
    \\
    0.700 & -0.800 & 1.441 & 2.560 & 0.490 & 1.440 & 0.810 & 0.810 & 1.440 & 0.640 \\
    0.004 & 0.006 & 0.014 & 0.013 & 0.003 & 0.009 & 0.005 & 0.005 & 0.007 & 0.006
    \\
    $\hat{\theta}_n^{(20)}$ & 
    $\hat{\theta}_n^{(21)}$
    &
    $\hat{\theta}_n^{(22)}$ &
    $\hat{\theta}_n^{(23)}$ & 
    $\hat{\theta}_n^{(24)}$ &
    $\hat{\theta}_n^{(25)}$ &
    $\hat{\theta}_n^{(26)}$ & &\\ \hline
    1.210 & 2.250 & 1.690 & 0.490 & 1.960 & 0.810 & 1.210
    \\
    1.210 & 2.250 & 1.690 & 0.490 & 1.960 & 0.810 & 1.210
    \\
    0.007 & 0.012 & 0.008 & 0.009 & 0.010 & 0.006 & 0.011
    \vspace{2mm}\\
    \caption{Sample mean and sample standard deviation (S.D.) of $\hat{\theta}_n$.}\label{thetatable}
\end{longtable}
\vspace{2mm}
\begin{longtable}[h]{c}
    The number of rejections  \\ \hline
    10000\\     
    \caption{The number of rejections of the misspecified model.}\label{TableH1}
\end{longtable}
\clearpage
\ \\ \ \\ \ \\
\begin{figure}[h]
    \centering
    \includegraphics[width=0.27\columnwidth]{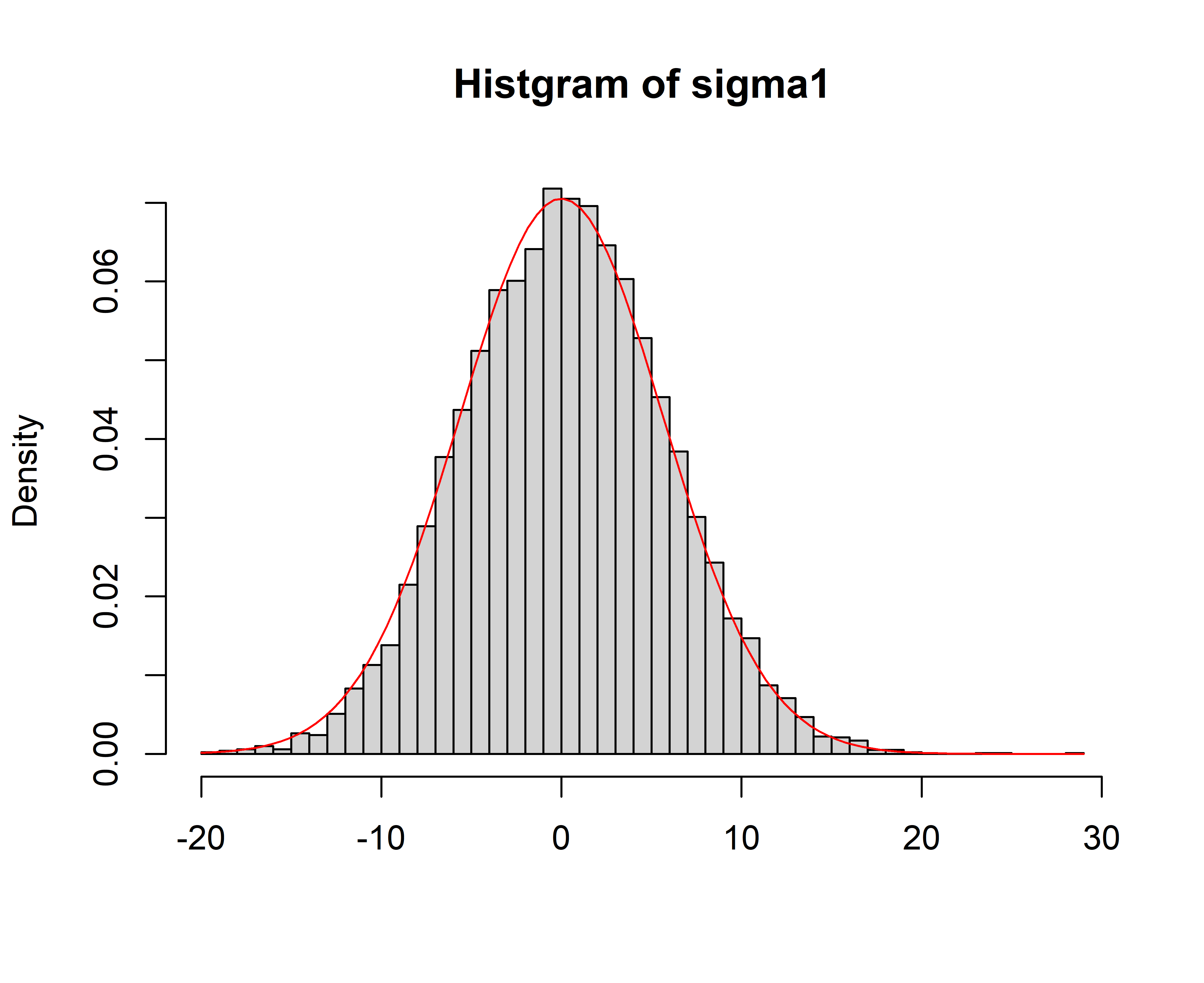}
    \includegraphics[width=0.27\columnwidth]{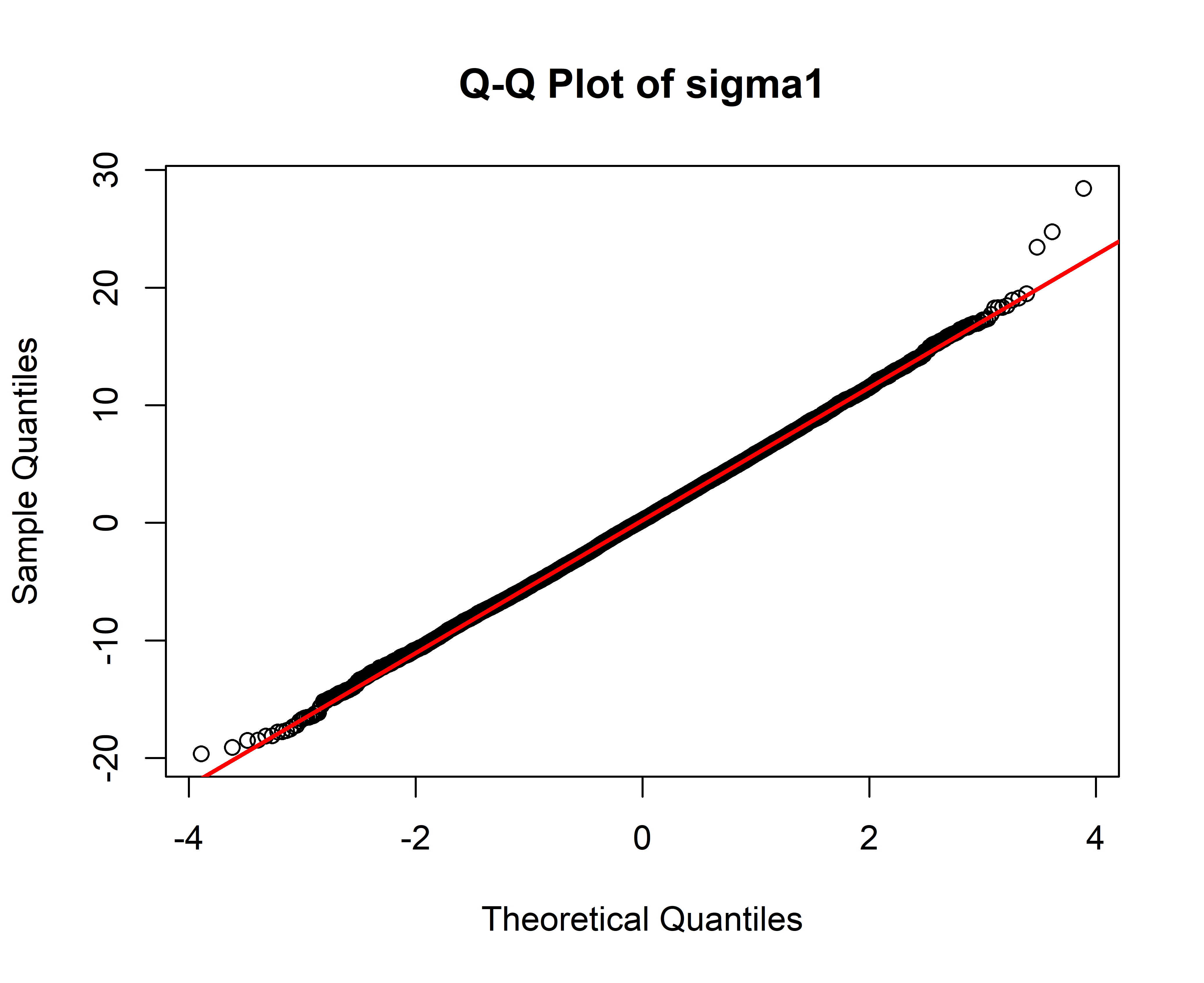}
    \includegraphics[width=0.27\columnwidth]{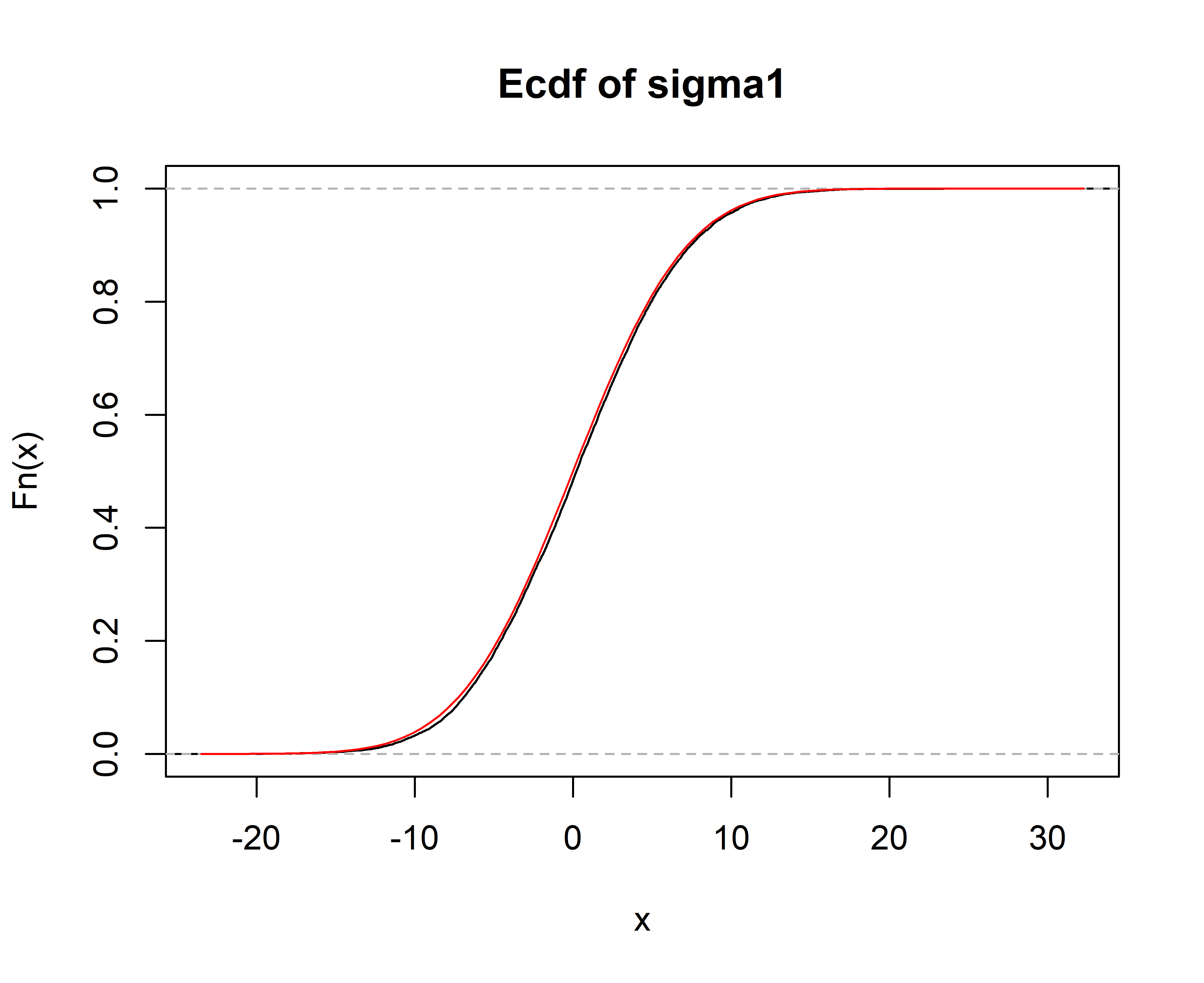}
    \caption{Histogram (left), Q-Q plot (middle) and empirical distribution (right) of $\sqrt{n}((\hat{\bf{\Sigma}}_{n})_{11}-({\bf{\Sigma}}_0)_{11})$. The red lines are theoretical curves.}
    \label{sigmafigure}
\end{figure}
\ \\ 
\begin{figure}[h]
    \centering
    \includegraphics[width=0.27\columnwidth]{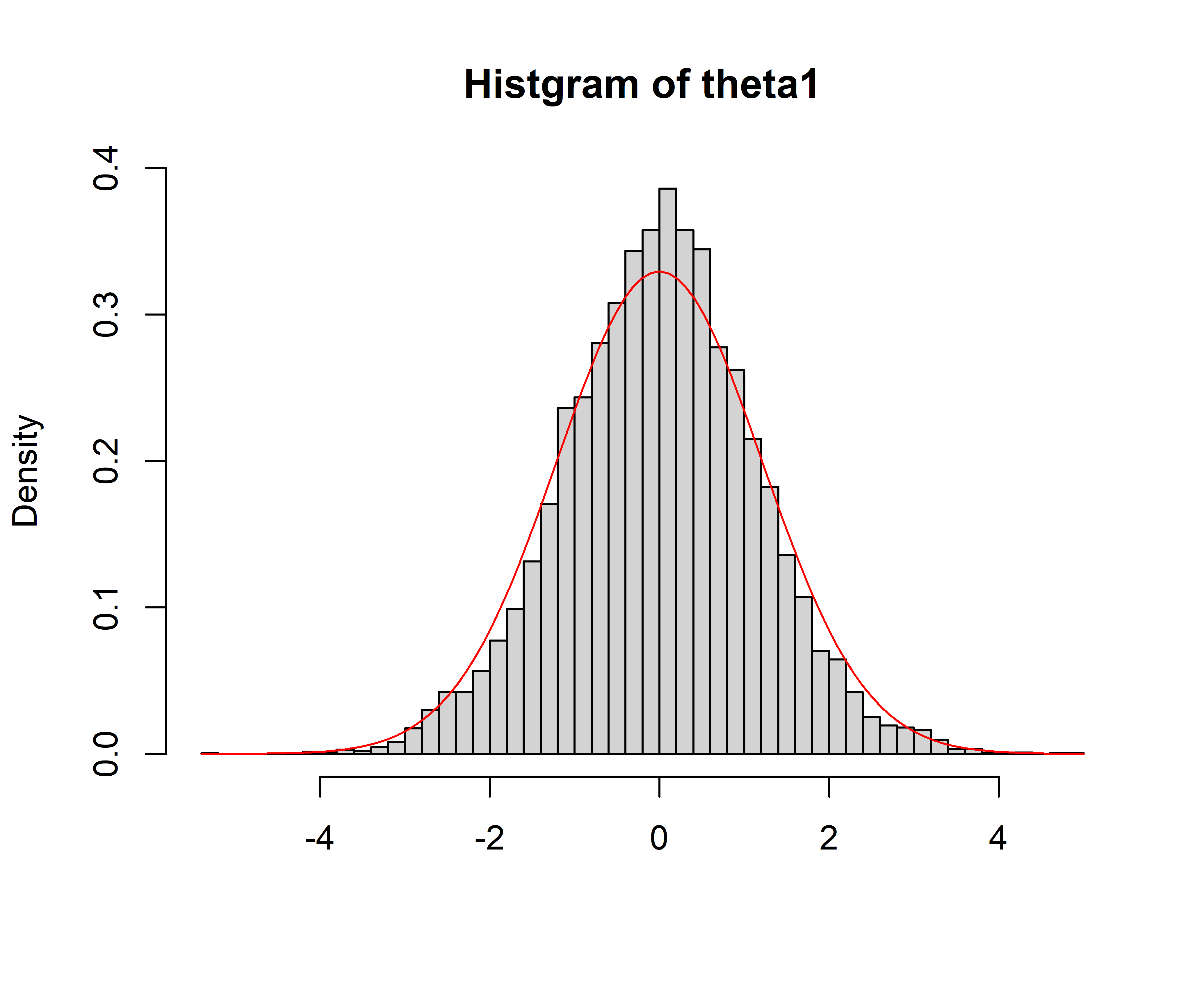}
    \includegraphics[width=0.27\columnwidth]{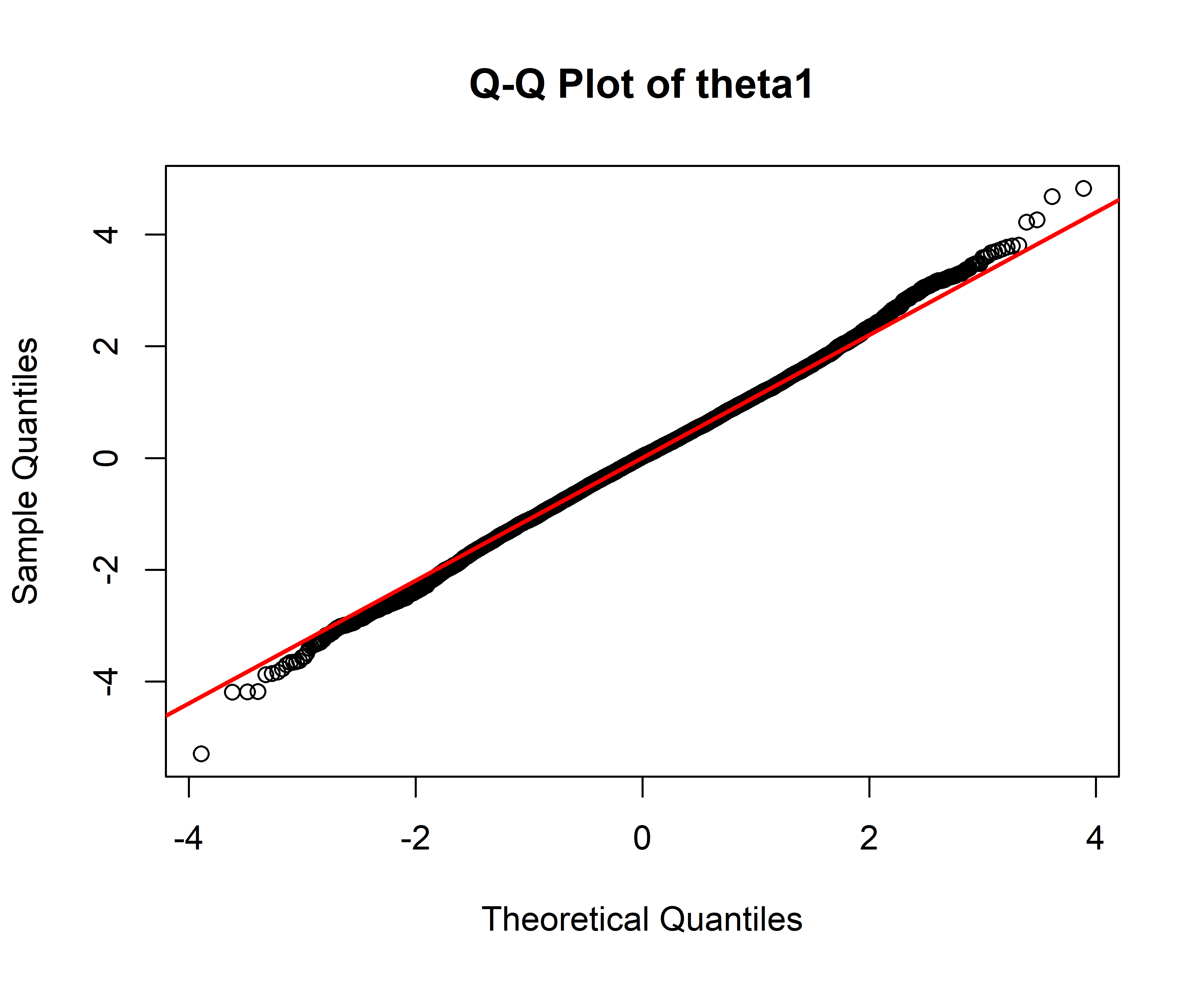}
    \includegraphics[width=0.27\columnwidth]{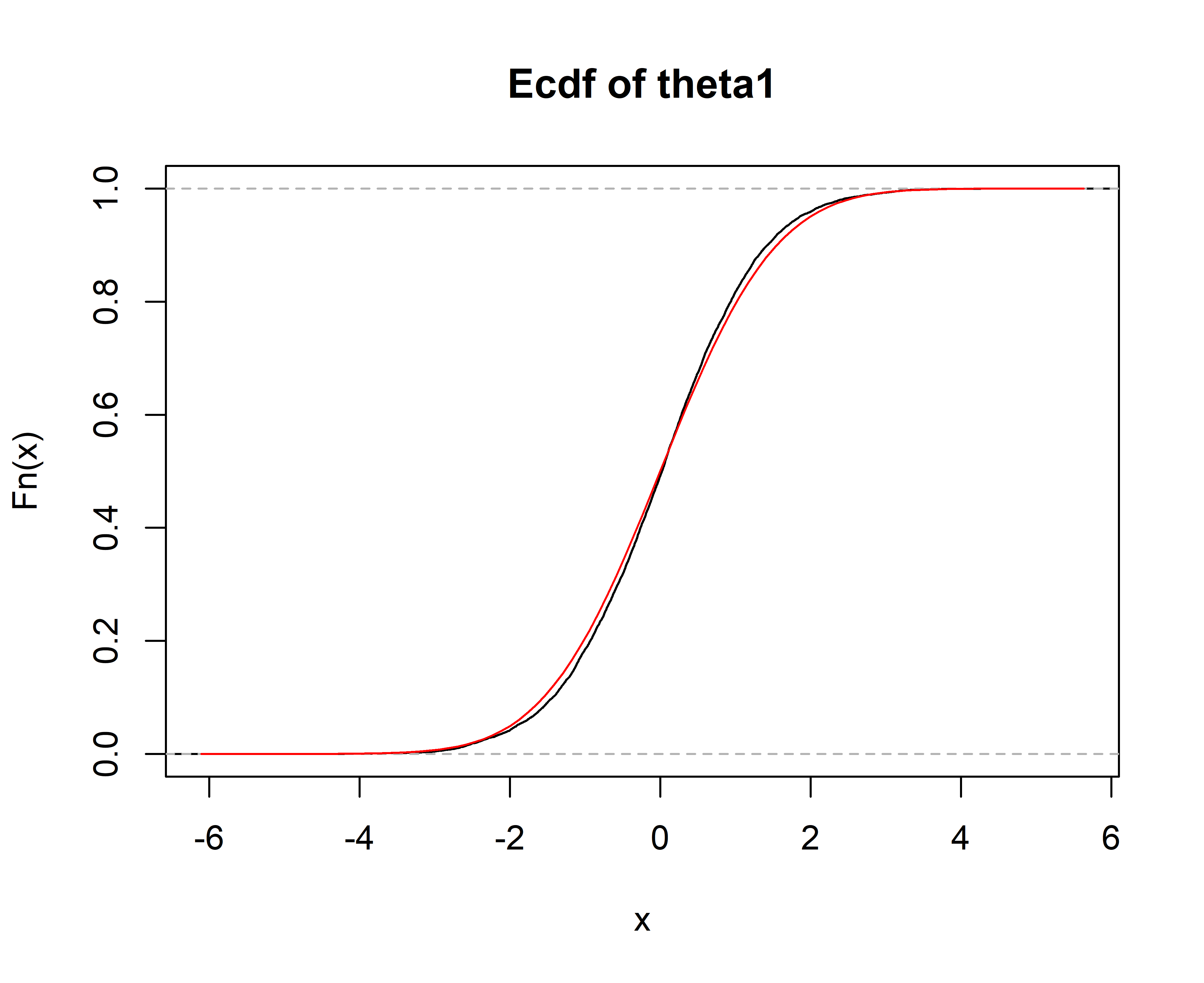}
    \caption{Histogram (left), Q-Q plot (middle) and empirical distribution (right) of $\sqrt{n}(\hat{\theta}_{n}^{(1)}-\theta_{0}^{(1)})$. The red lines are theoretical curves.}\label{thetafigure}
\end{figure}
\ \\
\begin{figure}[h]
    \centering
    \includegraphics[width=0.27\columnwidth]{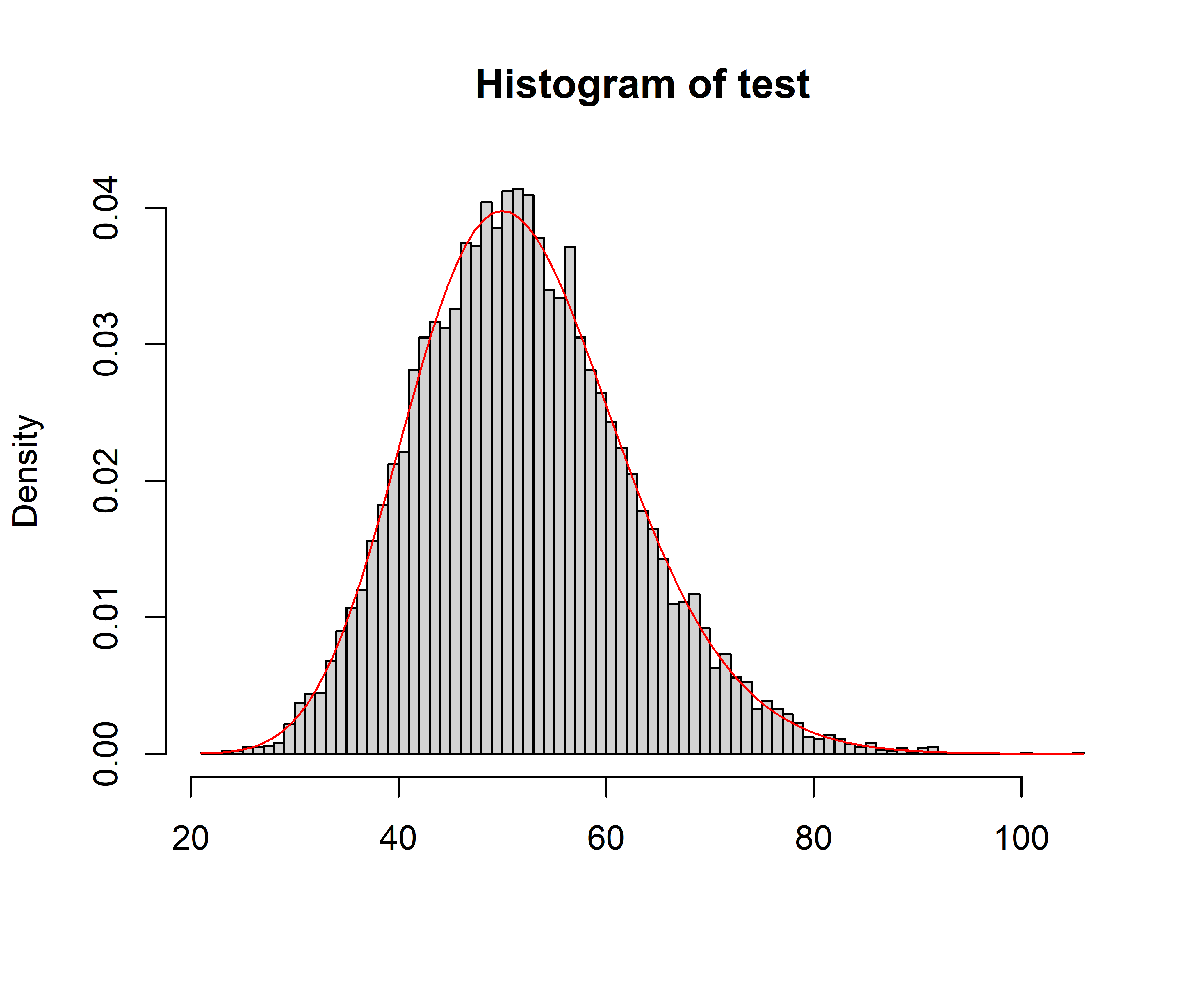}
    \includegraphics[width=0.27\columnwidth]{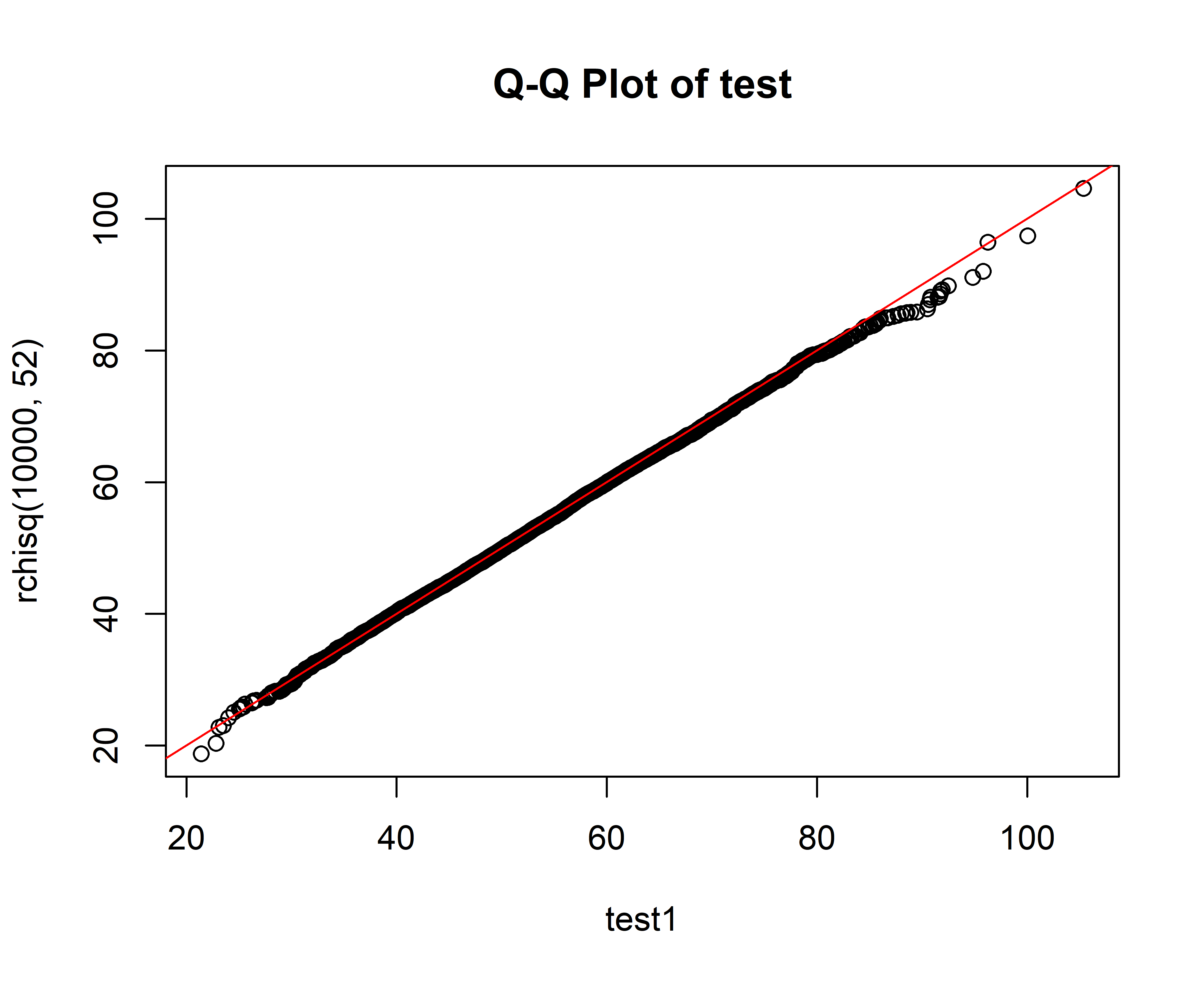}
    \includegraphics[width=0.27\columnwidth]{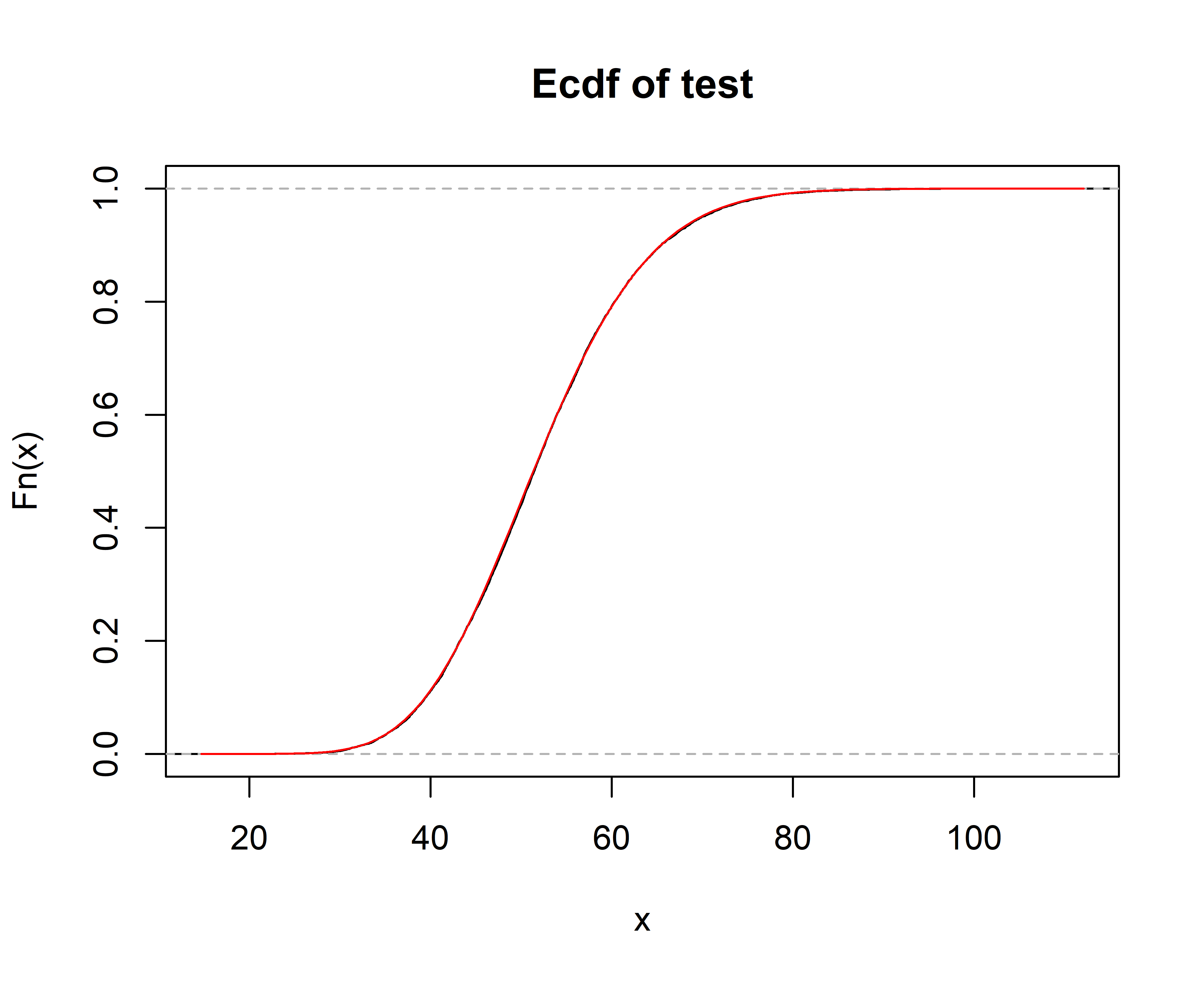}
\caption{Histogram (left), Q-Q plot (middle) and empirical distribution (right) of ${\bf{T}}_{n}$. The red lines are theoretical curves.}
\label{testfigure}
\end{figure}
\clearpage
\fontsize{10pt}{16pt}\selectfont
\section{Proofs} \label{Proof}
In Lemmas \ref{Pine}-\ref{Dlemma}, Propositions \ref{Nprop}-\ref{Econvprop2}, and Theorem \ref{Qtheorem}, we simply write ${\bf{P}}_{{\bf{\Sigma}}_0}$ as ${\bf{P}}$, and 
${\bf{E}}$ stands for expectation under ${\bf{P}}$. $C$ denotes a general positive constant whose value can  vary depending on the context without being specifically mentioned. Let $\lambda_0=\sum_{i=1}^4\lambda_{i,0}$.
\begin{proof}[\textbf{Proof of Lemma \ref{Pine}}] 
First, we show (\ref{supine1}). Since
\begin{align*}
    |X_t-X_{t_{i-1}^n}|&\leq |X_{1,t}-X_{1,t_{i-1}^n}|+|X_{2,t}-X_{2,t_{i-1}^n}|\\
    &=|{\bf{\Lambda}}_{1,0}(\xi_t-\xi_{t_{i-1}^n})+\delta_t-\delta_{t_{i-1}^n}|\\
    &\qquad+|
    {\bf{\Lambda}}_{2,0}{\bf{\Psi}}_0^{-1}{\bf{\Gamma}}_0(\xi_t-\xi_{t_{i-1}^n})+{\bf{\Lambda}}_{2,0}{\bf{\Psi}}_0^{-1}(\zeta_t-\zeta_{t_{i-1}^n})+\varepsilon_t-\varepsilon_{t_{i-1}^n}|
    \\
    &\leq C|\xi_t-\xi_{t_{i-1}^n}|+|\delta_t-\delta_{t_{i-1}^n}|+|\varepsilon_t-\varepsilon_{t_{i-1}^n}|+C|\zeta_t-\zeta_{t_{i-1}^n}|,
\end{align*}
we obtain
\begin{align*}
    \sup_{t\in[t_{i-1}^n,\tau_i^n)}|X_t-X_{t_{i-1}^n}|
    &\leq C\sup_{t\in[t_{i-1}^n,\tau_i^n)}|\xi_t-\xi_{t_{i-1}^n}|+\sup_{t\in[t_{i-1}^n,\tau_i^n)}|\delta_t-\delta_{t_{i-1}^n}|\\
    &\qquad\qquad\quad+\sup_{t\in[t_{i-1}^n,\tau_i^n)}|\varepsilon_t-\varepsilon_{t_{i-1}^n}|+C\sup_{t\in[t_{i-1}^n,\tau_i^n)}|\zeta_t-\zeta_{t_{i-1}^n}|.
\end{align*}
In a similar manner to Lemma 2.1 in Shimizu and Yoshida \cite{Shimizu(2006)}, it is shown that
\begin{align*}
    {\bf{P}}\biggl(C\sup_{t\in[t_{i-1}^n,\tau_i^n)}|\xi_{t}-\xi_{t_{i-1}^n}|>Dh_n^{\rho} \Bigl|\mathcal{F}_{i-1}^n\biggr)&=R_{i-1}(h_n^p,\xi),\\
    {\bf{P}}\biggl(\sup_{t\in[t_{i-1}^n,\tau_i^n)}|\delta_t-\delta_{t_{i-1}^n}|>Dh_n^{\rho} \Bigl|\mathcal{F}_{i-1}^n\biggr)&=R_{i-1}(h_n^p,\delta),\\
    {\bf{P}}\biggl(\sup_{t\in[t_{i-1}^n,\tau_i^n)}|\varepsilon_t-\varepsilon_{t_{i-1}^n}|>Dh_n^{\rho} \Bigl|\mathcal{F}_{i-1}^n\biggr)&=R_{i-1}(h_n^p,\varepsilon)
\end{align*}
and
\begin{align*}
    {\bf{P}}\biggl(C\sup_{t\in[t_{i-1}^n,\tau_i^n)}|\zeta_t-\zeta_{t_{i-1}^n}|>Dh_n^{\rho} \Bigl|\mathcal{F}_{i-1}^n\biggr)&=R_{i-1}(h_n^p,\zeta)
\end{align*}
for any $p\geq 1$. Therefore, one gets
\begin{align*}
    &\quad\ {\bf{P}}\biggl(\sup_{t\in[t_{i-1}^n,\tau_i^n)}|X_t-X_{t_{i-1}^n}|>Dh_n^{\rho} \Bigl|\mathcal{F}_{i-1}^n\biggr)\\
    &\leq {\bf{P}}\biggl(C\sup_{t\in[t_{i-1}^n,\tau_i^n)}|\xi_t-\xi_{t_{i-1}^n}|>\frac{Dh_n^{\rho}}{4} \Bigl|\mathcal{F}_{i-1}^n\biggr)+{\bf{P}}\biggl(\sup_{t\in[t_{i-1}^n,\tau_i^n)}|\delta_t-\delta_{t_{i-1}^n}|>\frac{Dh_n^{\rho}}{4} \Bigl|\mathcal{F}_{i-1}^n\biggr)\\
    &\qquad+{\bf{P}}\biggl(\sup_{t\in[t_{i-1}^n,\tau_i^n)}|\varepsilon_t-\varepsilon_{t_{i-1}^n}|>\frac{Dh_n^{\rho}}{4} \Bigl|\mathcal{F}_{i-1}^n\biggr)+{\bf{P}}\biggl(C\sup_{t\in[t_{i-1}^n,\tau_i^n)}|\zeta_t-\zeta_{t_{i-1}^n}|>\frac{Dh_n^{\rho}}{4} \Bigl|\mathcal{F}_{i-1}^n\biggr)\\
    &=R_{i-1}(h_n^p,\xi,\delta,\varepsilon,\zeta)
\end{align*}
for all $p\geq 1$. In an analogous manner to Lemma 2.1 in Shimizu and Yoshida \cite{Shimizu(2006)}, we have
\begin{align*}
    {\bf{P}}\biggl(C\sup_{t\in[\eta_{i}^n,t_i^n)}|\xi_{t_{i}^n}-\xi_t|>Dh_n^{\rho}\ \Bigl|\mathcal{F}_{i-1}^n\biggr)&=R_{i-1}(h_n^p,\xi),\\
    {\bf{P}}\biggl(\sup_{t\in[\eta_{i}^n,t_i^n)}|\delta_{t_{i}^n}-\delta_t|>Dh_n^{\rho}\ \Bigl|\mathcal{F}_{i-1}^n\biggr)&=R_{i-1}(h_n^p,\delta),\\
    {\bf{P}}\biggl(\sup_{t\in[\eta_{i}^n,t_i^n)}|\varepsilon_{t_{i}^n}-\varepsilon_t|>Dh_n^{\rho}\ \Bigl|\mathcal{F}_{i-1}^n\biggr)&=R_{i-1}(h_n^p,\varepsilon)
\end{align*}
and
\begin{align*}
    {\bf{P}}\biggl(C\sup_{t\in[\eta_{i}^n,t_i^n)}|\zeta_{t_{i}^n}-\zeta_t|>Dh_n^{\rho}\ \Bigl|\mathcal{F}_{i-1}^n\biggr)&=R_{i-1}(h_n^p,\zeta)
\end{align*}
for any $p\geq 1$, which implies (\ref{supine2}).
\end{proof}
\begin{lemma}\label{Mlemma}
Let $m\geq n$. For a matrix $M\in\mathbb{R}^{m\times n}$ and a vector $x\in\mathbb{R}^n$, if $M$ is full column rank,
\begin{align*}
    |Mx|\geq \Bigl(\min_{i} \sigma_i\Bigr)|x|,
\end{align*}
where for $i=1,\ldots,n$, $\sigma_i>0$ and $\sigma_i$  is a singular value of $M$.
\end{lemma}
\begin{proof}
First, we consider the singular value decomposition of $M$:
\begin{align*}
    M=U\Lambda V^{\top},
\end{align*}
where $U$ and $V$ are $m\times n$ and $n\times n$ matrices such that $U^{\top}U=V^{\top}V=\mathbb{I}_n$, and
\begin{align*}
   \Lambda=\Diag(\sigma_1,\ldots,\sigma_n)^{\top}.
\end{align*}
Using the fact that $U^{\top}U=\mathbb{I}_n$, one gets
\begin{align*}
    |Mx|^2=x^{\top}M^{\top}Mx
    &=x^{\top}V\Lambda U^{\top}U\Lambda V^{\top}x\\
    &=x^{\top}V\Lambda^2 V^{\top}x=\sum_{i=1}^n\sigma_i^2 (V^{\top}x)^{(i)2}\geq\Bigl(\min_{i} \sigma_i^2\Bigr)|V^{\top}x|^2.
\end{align*}
Note that $VV^{\top}=\mathbb{I}_n$ since $V$ is a square matrix and $V^{\top}V=\mathbb{I}_n$. Consequently, we have
\begin{align*}
    |V^{\top}x|^2=x^{\top}VV^{\top}x=|x|^2,
\end{align*}
which completes the proof.
\end{proof}
\begin{proof}[\textbf{Proofs of Lemmas \ref{Clemma}-\ref{Dlemma}}]
First of all, we prove (\ref{CK4}). It is sufficient to show the case where $(k_1,k_2,k_3,k_4)=(2,0,0,0)\in K_4$. Since
\begin{align*}
    {\bf{P}}\Bigl(C^{n}_{i,2,0,0,0}\big|\mathcal{F}_{i-1}^n\Bigr)
    &={\bf{P}}\Bigl(|\Delta_{i}^n X|\leq Dh_n^{\rho}, J_{i,1}^n\geq 2, J_{i,2}^n=0, J_{i,3}^n=0, J_{i,4}^n=0\big|\mathcal{F}_{i-1}^n\Bigr)\\
    &\leq {\bf{P}}\Bigl(J_{i,1}^n\geq 2|\mathcal{F}_{i-1}^n\Bigr)
\end{align*}
and 
\begin{align*}
    {\bf{P}}\Bigl(J_{i,1}^n\geq 2|\mathcal{F}_{i-1}^n\Bigr)
    &=\sum_{k=2}^{\infty}\frac{(\lambda_{1,0}h_n)^k}{k!}e^{-\lambda_{1,0} h_n}\\
    &=\sum_{l=0}^{\infty}\frac{(\lambda_{1,0}h_n)^{l+2}}{(l+2)!}e^{-\lambda_{1,0} h_n}\\
    &\leq (\lambda_{1,0}h_n)^{2}\sum_{l=0}^{\infty}\frac{(\lambda_{1,0}h_n)^{l}}{l!}e^{-\lambda_{1,0}h_n}=\lambda_{1,0}^{2}h_n^2\leq \Bigl(\max_{j=1,2,3,4}\lambda_{j,0}^2\Bigr) h_n^2
\end{align*}
for $i=1,\ldots,n$, we obtain
\begin{align*}
     {\bf{P}}\Bigl(C^{n}_{i,2,0,0,0}\big|\mathcal{F}_{i-1}^n\Bigr)\leq \Bigl(\max_{j=1,2,3,4}\lambda_{j,0}^2\Bigr) h_n^2.
\end{align*}
Similarly, (\ref{DK4}) can be shown. Furthermore, we see
\begin{align*}
    {\bf{P}}\Bigl(C^{n}_{i,1,1,1,1}\big|\mathcal{F}_{i-1}^n\Bigr)
    &={\bf{P}}\Bigl(|\Delta_{i}^n X|\leq Dh_n^{\rho}, J_{i,1}^n=1, J_{i,2}^n=1, J_{i,3}^n=1, J_{i,4}^n=1\big|\mathcal{F}_{i-1}^n\Bigr)\\
    &\leq {\bf{P}}\Bigl(J_{i,1}^n=1, J_{i,2}^n=1, J_{i,3}^n=1, J_{i,4}^n=1\big|\mathcal{F}_{i-1}^n\Bigr)\\
    &={\bf{P}}\Bigl(J_{i,1}^n=1\big|\mathcal{F}_{i-1}^n\Bigr){\bf{P}}\Bigl(J_{i,2}^n=1\big|\mathcal{F}_{i-1}^n\Bigr){\bf{P}}\Bigl(J_{i,3}^n=1\big|\mathcal{F}_{i-1}^n\Bigr){\bf{P}}\Bigl(J_{i,4}^n=1\big|\mathcal{F}_{i-1}^n\Bigr)\\
    &=e^{-\lambda_{0}h_n}\lambda_{1,0}\lambda_{2,0}\lambda_{3,0}\lambda_{4,0}h_n^4\\
    &\leq \lambda_{1,0}\lambda_{2,0}\lambda_{3,0}\lambda_{4,0} h_n^4,
\end{align*}
which yields (\ref{C1}). In the same way, we can obtain (\ref{CK2})-(\ref{CK3}), (\ref{D1}), and (\ref{DK2})-(\ref{DK3}). Note that $\tau_i^n=t_{i}^n$ on 
\begin{align*}
    \bigl\{J_{i,1}^n=0, J_{i,2}^n=0, J_{i,3}^n=0, J_{i,4}^n=0\bigr\},
\end{align*}
and Taylor’s theorem implies 
\begin{align*}
    1-e^{-\lambda_{0} h_n}\leq \lambda_0 h_n.
\end{align*}
It holds from Lemma \ref{Pine} that
\begin{align*}
    {\bf{P}}\Bigl(D^{n}_{i,0,0,0,0}\big|\mathcal{F}_{i-1}^n\Bigr)
    &={\bf{P}}\Bigl(|\Delta_{i}^n X|> Dh_n^{\rho}, J_{i,1}^n=0, J_{i,2}^n=0, J_{i,3}^n=0, J_{i,4}^n=0\big|\mathcal{F}_{i-1}^n\Bigr)\\
    &\leq {\bf{P}}\Bigl(|X_{\tau_i^n}-X_{t_{i-1}^n}|> Dh_n^{\rho} \big|\mathcal{F}_{i-1}^n \Bigr)\\
    &\leq {\bf{P}}\biggl(\sup_{t\in [t_{i-1}^n,\tau_i^n)}|X_{t}-X_{t_{i-1}^n}|> Dh_n^{\rho}\ \Big|\mathcal{F}_{i-1}^n \biggr)=R_{i-1}(h_n^{p},\xi,\delta,\varepsilon,\zeta)
\end{align*}
and 
\begin{align*}
    {\bf{P}}\Bigl(C^{n}_{i,0,0,0,0}\big|\mathcal{F}_{i-1}^n\Bigr)
    &={\bf{P}}\Bigl(|\Delta_{i}^n X|\leq Dh_n^{\rho}, J_{i,1}^n=0, J_{i,2}^n=0, J_{i,3}^n=0, J_{i,4}^n=0\big|\mathcal{F}_{i-1}^n\Bigr)\\
    &={\bf{P}}\Bigl(J_{i,1}^n=0, J_{i,2}^n=0, J_{i,3}^n=0, J_{i,4}^n=0\big|\mathcal{F}_{i-1}^n\Bigr)\\
    &\qquad\quad-{\bf{P}}\Bigl(|\Delta_{i}^n X|> Dh_n^{\rho}, J_{i,1}^n=0, J_{i,2}^n=0, J_{i,3}^n=0, J_{i,4}^n=0\big|\mathcal{F}_{i-1}^n\Bigr)\\
    &=e^{-\lambda_0 h_n}-R_{i-1}(h_n^{p},\xi,\delta,\varepsilon,\zeta)\\
    &=\tilde{R}_{i-1}(h_n,\xi,\delta,\varepsilon,\zeta)
\end{align*}
for any $p\geq 1$, which completes the proof of (\ref{C0}) and (\ref{D0}). Finally, we consider (\ref{CK1}) and (\ref{DK11})-(\ref{DK14}). It is sufficient to show the case where $(k_1,k_2,k_3,k_4)=(1,0,0,0)\in K_1$. First, we see
\begin{align*}
    &\quad\ {\bf{P}}\Bigl(C^{n}_{i,1,0,0,0}\big|\mathcal{F}_{i-1}^n\Bigr)\\
    &={\bf{P}}\biggl(|\Delta_{i}^n X|\leq Dh_n^{\rho}, J_{i,1}^n=1, J_{i,2}^n=0, J_{i,3}^n=0, J_{i,4}^n=0, |\Delta Z_{1,\tau_i^n}|> \frac{2Dh_n^{\rho}}{\sigma_0 c_{1,0}}\Big|\mathcal{F}_{i-1}^n\biggr)\\
    &\qquad+{\bf{P}}\biggl(|\Delta_{i}^n X|\leq Dh_n^{\rho}, J_{i,1}^n=1, J_{i,2}^n=0, J_{i,3}^n=0, J_{i,4}^n=0, |\Delta Z_{1,\tau_i^n}|\leq \frac{2Dh_n^{\rho}}{\sigma_0 c_{1,0}}\Big|\mathcal{F}_{i-1}^n\biggr),
\end{align*}
where $\sigma_0$ is a minimum singular value of ${\bf{\Lambda}}_{1,0}$, $c_{1,0}$ is the positive constant in ${\bf{[A4]}}$, 
\begin{align*}
    Z_{1,t}=\int_{[0,t]\times E_1}z_1 p_1(dt,dz_1)
\end{align*}
and $\Delta Z_{1,\tau_i^n}$ has the density $F_1$ under $\mathcal{F}_{i-1}^n$. Note that $\sigma_0>0$ since ${\bf{\Lambda}}_{1,0}$ is full column rank.
On
\begin{align*}
    \bigl\{J_{i,1}^n=1, J_{i,2}^n=0, J_{i,3}^n=0, J_{i,4}^n=0\bigr\},
\end{align*}
it follows from Lemma \ref{Mlemma} that
\begin{align*}
   |X_{\tau_i^n}-X_{\tau_i^n-}|&=\Biggl|\begin{pmatrix}
    {\bf{\Lambda}}_{1,0}\\
    {\bf{\Lambda}}_{2,0}{\bf{\Psi}}_0^{-1}{\bf{\Gamma}}_0
    \end{pmatrix}(\xi_{\tau_i^n}-\xi_{\tau_i^n-})
    \Biggr|\\
    &\geq \Bigl| {\bf{\Lambda}}_{1,0}(\xi_{\tau_i^n}-\xi_{\tau_i^n-})\Bigr|\\
    &\geq \sigma_0|\xi_{\tau_i^n}-\xi_{\tau_i^n-}|
    =\sigma_{0}\bigl|c_1(\xi_{\tau_i^n},\Delta Z_{1,\tau_{i}^n})\bigr|\geq \sigma_0 c_{1,0}|\Delta Z_{1,\tau_{i}^n}|
\end{align*}
when $|\Delta \xi_{\tau_{i}^n}|$ is small enough, and it holds from the triangle inequality and $\tau_i^n=\eta_{i}^n$ that
\begin{align*}
    |X_{\tau_i^n}-X_{\tau_i^n-}|
    &\leq |X_{\tau_i^n}-X_{t_{i}^n}|+|X_{t_{i}^n}-X_{t_{i-1}^n}|+|X_{t_{i-1}^n}-X_{\tau_i^n-}|\\
    &=|X_{t_{i}^n}-X_{\eta_i^n}|+|\Delta_i^n X|+|X_{\tau_i^n-}-X_{t_{i-1}^n}|.
\end{align*}
Thus, one gets
\begin{align*}
    \sup_{t\in[\eta_{i}^n,t_i^n)}|X_{t_i^n}-X_{t}|+\sup_{t\in [t_{i-1}^n,\tau_i^n)}|X_t-X_{t_{i-1}^n}|
    &\geq |X_{t_i^n}-X_{\eta_i^n}\bigr|+\bigl|X_{\tau_i^n-}-X_{t_{i-1}^n}|\\
    &\geq |X_{\tau_i^n}-X_{\tau_i^n-}\bigr|-\bigl|\Delta_i^n X|\\
    &\geq \sigma_0 c_{1,0}|\Delta Z_{1,\tau_{i}^n}|-Dh_n^{\rho}>Dh_n^{\rho}
\end{align*}
when $|\Delta \xi_{\tau_{i}^n}|$ is small enough on 
\begin{align*}
    \biggl\{|\Delta_{i}^n X|\leq Dh_n^{\rho}, J_{i,1}^n=1, J_{i,2}^n=0, J_{i,3}^n=0, J_{i,4}^n=0, |\Delta Z_{1,\tau_i^n}|> \frac{2Dh_n^{\rho}}{\sigma_0 c_{1,0}}\biggr\}.
\end{align*}
Consequently, Lemma \ref{Pine} shows
\begin{align*}
    &\quad\ {\bf{P}}\biggl(|\Delta_{i}^n X|\leq Dh_n^{\rho}, J_{i,1}^n=1, J_{i,2}^n=0, J_{i,3}^n=0, J_{i,4}^n=0, |\Delta Z_{1,\tau_i^n}|> \frac{2Dh_n^{\rho}}{\sigma_0 c_{1,0}}\Big|\mathcal{F}_{i-1}^n\biggr)\\
    &\leq {\bf{P}}\biggl(\sup_{t\in[\eta_{i}^n,t_i^n)}|X_{t_i^n}-X_{t}|+\sup_{t\in [t_{i-1}^n,\tau_i^n)}|X_t-X_{t_{i-1}^n}|>Dh_n^{\rho}\Big|\mathcal{F}_{i-1}^n\biggr)\\
    &\leq {\bf{P}}\biggl(\sup_{t\in[\eta_{i}^n,t_i^n)}|X_{t_i^n}-X_{t}|>\frac{Dh_n^{\rho}}{2}\Big|\mathcal{F}_{i-1}^n\biggr)+{\bf{P}}\biggl(
    \sup_{t\in [t_{i-1}^n,\tau_i^n)}|X_t-X_{t_{i-1}^n}|>\frac{Dh_n^{\rho}}{2}\Big|\mathcal{F}_{i-1}^n\biggr)\\
    &=R_{i-1}(h_n^{p},\xi,\delta,\varepsilon,\zeta)
\end{align*}
for all $p\geq 1$. Since $h_n^{\rho}\longrightarrow 0$ as $n\longrightarrow\infty$, we note that
\begin{align*}
    \frac{2Dh_n^{\rho}}{\sigma_0 c_{1,0}}\leq r_1
\end{align*}
for a sufficiently large $n$, where $r_{1}$ is the positive constant in ${\bf{[A3]}}$. It follows that 
\begin{align*}
   {\bf{P}}\biggl(|\Delta Z_{1,\tau_i^n}|\leq \frac{2Dh_n^{\rho}}{\sigma_0 c_{1,0}}\Big|\mathcal{F}_{i-1}^n, J_{i,1}^n=1\biggr)
    &=\int_{\{|z_1|\leq 2\sigma_0^{-1}c_{1,0}^{-1}Dh_n^{\rho}\} \backslash\{0\}}F_1(z_1)dz_1\\
    &=\frac{1}{\lambda_{1,0}}\int_{\{|z_1|\leq 2\sigma_0^{-1}c_{1,0}^{-1}Dh_n^{\rho}\}\backslash\{0\}}f_1(z_1){\bf{1}}_{\{|z_1|\leq r_1\}}dz_1\\
    &\leq \frac{K_1}{\lambda_{1,0}}\int_{\{|z_1|\leq 2\sigma_0^{-1}c_{1,0}^{-1}Dh_n^{\rho}\}\backslash\{0\}}|z_1|^{1-k_1}dz_1\\
    &\leq \frac{Ch_n^{\rho}}{\lambda_{1,0}}
\end{align*}
for a sufficiently large $n$, so that
\begin{align*}
    &\quad\ {\bf{P}}\biggl(|\Delta_{i}^n X|\leq Dh_n^{\rho}, J_{i,1}^n=1, J_{i,2}^n=0, J_{i,3}^n=0, J_{i,4}^n=0, |\Delta Z_{1,\tau_i^n}|\leq \frac{2Dh_n^{\rho}}{\sigma_0 c_{1,0}}\Big|\mathcal{F}_{i-1}^n\biggr)\\
    &\leq {\bf{P}}\biggl(J_{i,1}^n=1, |\Delta Z_{1,\tau_i^n}|\leq \frac{2Dh_n^{\rho}}{\sigma_0 c_{1,0}}\Big|\mathcal{F}_{i-1}^n\biggr)\\
    &={\bf{P}}\biggl(|\Delta Z_{1,\tau_i^n}|\leq \frac{2Dh_n^{\rho}}{\sigma_0 c_{1,0}}\Big|\mathcal{F}_{i-1}^n, J_{i,1}^n=1\biggr){\bf{P}}\Bigl(J_{i,1}^n=1\Bigr)\leq C h_n^{\rho+1}.
\end{align*}
Therefore, since $p$ is arbitrary, we obtain
\begin{align*}
    {\bf{P}}\Bigl(C^{n}_{i,1,0,0,0}\big|\mathcal{F}_{i-1}^n\Bigr)=R_{i-1}(h_n^{\rho+1},\xi,\delta,\varepsilon,\zeta)
\end{align*}
for a sufficiently large $n$. Moreover, for a sufficiently large $n$, one gets
\begin{align*}
    {\bf{P}}\Bigl(D^{n}_{i,1,0,0,0}\big|\mathcal{F}_{i-1}^n\Bigr)
    &={\bf{P}}\Bigl(|\Delta_{i}^n X|> Dh_n^{\rho}, J_{i,1}^n=1, J_{i,2}^n=0, J_{i,3}^n=0, J_{i,4}^n=0\big|\mathcal{F}_{i-1}^n\Bigr)\\
    &={\bf{P}}\Bigl(J_{i,1}^n=1, J_{i,2}^n=0, J_{i,3}^n=0, J_{i,4}^n=0\big|\mathcal{F}_{i-1}^n\Bigr)\\
    &\qquad\quad-{\bf{P}}\Bigl(|\Delta_{i}^n X|\leq Dh_n^{\rho}, J_{i,1}^n=1, J_{i,2}^n=0, J_{i,3}^n=0, J_{i,4}^n=0\big|\mathcal{F}_{i-1}^n\Bigr)\\
    &=e^{-\lambda_0 h_n}\lambda_{1,0}h_n-R_{i-1}(h_n^{\rho+1},\xi,\delta,\varepsilon,\zeta)\\
    &=\lambda_{1,0}h_n\tilde{R}_{i-1}(h_n^{\rho},\xi,\delta,\varepsilon,\zeta),
\end{align*}
which yields (\ref{DK11}). Similarly, (\ref{DK12})-(\ref{DK14}) can be shown.
\end{proof}
\begin{proposition}\label{Nprop}
Under {\rm{\textbf{[A1]}}}, {\rm{\textbf{[A3]}}} and {\rm{\textbf{[A4]}}}, as $n\longrightarrow\infty$,
\begin{align}
    \frac{\tilde{N}_n}{n}\overset{p}{\longrightarrow}1 \label{Nprob}
\end{align}
and
\begin{align}
    \sqrt{n}\Bigl(\frac{\tilde{N}_n}{n}-1\Bigr)\overset{p}{\longrightarrow}0
    \label{Nasym} 
\end{align}
under ${\bf{P}}$.
\end{proposition}
\begin{proof}
First, we show (\ref{Nprob}). Lemma \ref{Clemma} implies
\begin{align*}
    {\bf{P}}\Bigl(|\Delta_{i}^n X|\leq Dh_n^{\rho}\big|\mathcal{F}_{i-1}^n\Bigr)&={\bf{P}}\biggl(\bigcup_{k_1,k_2,k_3,k_4=0,1,2}C^{n}_{i,k_1,k_2,k_3,k_4}\Big|\mathcal{F}_{i-1}^n\biggr)\\
    &=\sum_{k_1,k_2,k_3,k_4=0,1,2}{\bf{P}}\Bigl(C^{n}_{i,k_1,k_2,k_3,k_4}\big|\mathcal{F}_{i-1}^n\Bigr)=\tilde{R}_{i-1}(h_n,\xi,\delta,\varepsilon,\zeta)
\end{align*}
for a sufficiently large $n$. In other words, there exists $N_1\in\mathbb{N}$ such that
\begin{align*}
    {\bf{P}}\Bigl(|\Delta_{i}^n X|\leq Dh_n^{\rho}\big|\mathcal{F}_{i-1}^n\Bigr)=\tilde{R}_{i-1}(h_n,\xi,\delta,\varepsilon,\zeta)
\end{align*}
for any $n\geq N_1$. Fix $\varepsilon>0$ arbitrarily. Moreover, we have
\begin{align*}
    \frac{1}{n}\sum_{i=1}^n \tilde{R}_{i-1}(h_n,\xi,\delta,\varepsilon,\zeta)=1-\frac{h_n}{n}\sum_{i=1}^n R_{i-1}(1,\xi,\delta,\varepsilon,\zeta)\stackrel{p}{\longrightarrow}1,
\end{align*}
so that there exists $N_2\in\mathbb{N}$ such that
\begin{align*}
    {\bf{P}}\biggl(\biggl|\frac{1}{n}\sum_{i=1}^n \tilde{R}_{i-1}(h_n,\xi,\delta,\varepsilon,\zeta)-1\biggr|>\varepsilon\biggr)<\varepsilon
\end{align*}
for all $n\geq N_2$. Thus, for any $n\geq\max{\{N_1,N_2\}}$, one gets
\begin{align*}
    &\quad\ {\bf{P}}\biggl(\biggl|\frac{1}{n}\sum_{i=1}^n{\bf{P}}\Bigl(|\Delta_{i}^n X|\leq Dh_n^{\rho}\big|\mathcal{F}_{i-1}^n\Bigr)-1\biggr|>\varepsilon\biggr)\\
    &={\bf{P}}\biggl(\biggl|\frac{1}{n}\sum_{i=1}^n \tilde{R}_{i-1}(h_n,\xi,\delta,\varepsilon,\zeta)-1\biggr|>\varepsilon\biggr)<\varepsilon,
\end{align*}
which yields
\begin{align}
    \frac{1}{n}\sum_{i=1}^n{\bf{P}}\Bigl(|\Delta_{i}^n X|\leq Dh_n^{\rho}\big|\mathcal{F}_{i-1}^n\Bigr)\stackrel{p}{\longrightarrow}1.  \label{Pprob}
\end{align}
Since 
\begin{align*}
    {\bf{E}}\biggl[\frac{1}{n}\sum_{i=1}^n{\bf{1}}_{\{|\Delta_{i}^n X|\leq Dh_n^{\rho}\}}\Big|\mathcal{F}_{i-1}^n\biggr]
    &=\frac{1}{n}\sum_{i=1}^n{\bf{P}}\Bigl(|\Delta_{i}^n X|\leq Dh_n^{\rho}\big|\mathcal{F}_{i-1}^n\Bigr)\stackrel{p}{\longrightarrow}1
\end{align*}
and
\begin{align*}
     {\bf{E}}\biggl[\frac{1}{n^2}\sum_{i=1}^n{\bf{1}}^2_{\{|\Delta_{i}^n X|\leq Dh_n^{\rho}\}}\Big|\mathcal{F}_{i-1}^n\biggr]
     &=\frac{1}{n}\times \frac{1}{n}\sum_{i=1}^n{\bf{P}}\Bigl(|\Delta_{i}^n X|\leq Dh_n^{\rho}\big|\mathcal{F}_{i-1}^n\Bigr)\stackrel{p}{\longrightarrow}0,
\end{align*}
it follows from Lemma 9 in Genon-Catalot and Jacod \cite{Genon(1993)} that
\begin{align*}
    \frac{N_n}{n}=\frac{1}{n}\sum_{i=1}^n{\bf{1}}_{\{|\Delta_{i}^n X|\leq Dh_n^{\rho}\}}\stackrel{p}{\longrightarrow}1.
\end{align*}
Consequently, we see 
\begin{align*}
    {\bf{P}}\Bigl(N_n=0\Bigr)&={\bf{P}}\biggl(\frac{N_n}{n}=0\biggr)\leq {\bf{P}}\biggl(\Bigl|\frac{N_n}{n}-1\Bigr|>\frac{1}{2}\biggr)\longrightarrow 0, 
\end{align*}
so that we get
\begin{align}
\begin{split}
    {\bf{P}}\biggl(\Bigl|\frac{\tilde{N}_n}{n}-1\Bigr|>\varepsilon\biggr)&={\bf{P}}\biggl(\Bigl\{\Bigl|\frac{\tilde{N}_n}{n}-1\Bigr|>\varepsilon\Bigr\}\cap \Bigl\{N_n\neq 0\Bigr\} \biggr) \\
    &\qquad\quad+{\bf{P}}\biggl(\Bigl\{\Bigl|\frac{\tilde{N}_n}{n}-1\Bigr|>\varepsilon\Bigr\}\cap \Bigl\{N_n= 0\Bigr\} \biggr)\\
    &\leq {\bf{P}}\biggl(\Bigl|\frac{N_n}{n}-1\Bigr|>\varepsilon\biggr)+{\bf{P}}\Bigl(N_n=0\Bigr)\longrightarrow 0, \label{Nnprob}
\end{split}
\end{align}
which implies (\ref{Nprob}). Next, we prove (\ref{Nasym}). Since
\begin{align*}
    \frac{1}{\sqrt{n}}\sum_{i=1}^n R_{i-1}(h_n,\xi,\delta,\varepsilon,\zeta)=\sqrt{nh_n^2}\times \frac{1}{n}\sum_{i=1}^n R_{i-1}(1,\xi,\delta,\varepsilon,\zeta)\stackrel{p}{\longrightarrow}0,
\end{align*}
it is shown that
\begin{align*}
    \frac{1}{\sqrt{n}}\sum_{i=1}^n\Bigl\{{\bf{P}}\Bigl(|\Delta_{i}^n X|\leq Dh_n^{\rho}\big|\mathcal{F}_{i-1}^n\Bigr)-1\Bigr\}\stackrel{p}{\longrightarrow}0
\end{align*}
in an analogous manner to the proof of (\ref{Pprob}). Thus, it follows from (\ref{Pprob}) that 
\begin{align*}
    {\bf{E}}\biggl[\frac{1}{\sqrt{n}}\sum_{i=1}^n\bigl({\bf{1}}_{\{|\Delta_{i}^n X|\leq Dh_n^{\rho}\}}-1\bigr)\Big|\mathcal{F}_{i-1}^n\biggr]
    &=\frac{1}{\sqrt{n}}\sum_{i=1}^n\Bigl\{{\bf{P}}\Bigl(|\Delta_{i}^n X|\leq Dh_n^{\rho}\big|\mathcal{F}_{i-1}^n\Bigr)-1\Bigr\}\stackrel{p}{\longrightarrow}0
\end{align*}
and
\begin{align*}
    {\bf{E}}\biggl[\frac{1}{n}\sum_{i=1}^n\bigl({\bf{1}}_{\{|\Delta_{i}^n X|\leq Dh_n^{\rho}\}}-1\bigr)^2\Big|\mathcal{F}_{i-1}^n\biggr]
    &=\frac{1}{n}\sum_{i=1}^n{\bf{E}}\Bigl[{\bf{1}}_{\{|\Delta_{i}^n X|\leq Dh_n^{\rho}\}}\big|\mathcal{F}_{i-1}^n\Bigr]\\
    &\qquad\quad-\frac{2}{n}\sum_{i=1}^n{\bf{E}}\Bigl[{\bf{1}}_{\{|\Delta_{i}^n X|\leq Dh_n^{\rho}\}}\big|\mathcal{F}_{i-1}^n\Bigr]+1\\
    &=-\frac{1}{n}\sum_{i=1}^n{\bf{P}}\Bigl(|\Delta_{i}^n X|\leq Dh_n^{\rho}\big|\mathcal{F}_{i-1}^n\Bigr)+1\stackrel{p}{\longrightarrow}0,
\end{align*}
so that it holds from Lemma 9 in Genon-Catalot and Jacod \cite{Genon(1993)} that
\begin{align*}
    \sqrt{n}\Bigl(\frac{N_n}{n}-1\Bigr)=\frac{1}{\sqrt{n}}\sum_{i=1}^n\bigl({\bf{1}}_{\{|\Delta_{i}^n X|\leq Dh_n^{\rho}\}}-1\bigr)\overset{p}{\longrightarrow}0.
\end{align*}
In a similar way to (\ref{Nnprob}), we complete the proof of  (\ref{Nasym}).
\end{proof}
\begin{proposition}\label{X1momentprop}
Under {\rm{\textbf{[A1]}}}-{\rm{\textbf{[A4]}}}, for a sufficiently large $n$, 
\begin{align}
    \begin{split}
    {\bf{E}}\Bigl[(\Delta_i^n X_1)^{(j_1)}(\Delta_i^n X_1)^{(j_2)}{\bf{1}}_{\{|\Delta_i^n X|\leq Dh_n^{\rho}\}}\big|\mathcal{F}_{i-1}^n\Bigr]=h_n ({\bf{\Sigma}}^{11}_0)_{j_1j_2}
    + R_{i-1}(h_n^{2},\xi,\delta,\varepsilon,\zeta)
    \end{split}\label{EX1X1}
\end{align}
and
\begin{align}
    \begin{split}
    &\quad\ {\bf{E}}\Bigl[(\Delta_i^n X_1)^{(j_1)}(\Delta_i^n X_1)^{(j_2)}(\Delta_i^n X_1)^{(j_3)}(\Delta_i^n X_1)^{(j_4)}{\bf{1}}_{\{|\Delta_i^n X|\leq Dh_n^{\rho}\}}\big|\mathcal{F}_{i-1}^n\Bigr]\\
    &=h_n^2\bigl\{({\bf{\Sigma}}^{11}_0)_{j_1j_2}({\bf{\Sigma}}^{11}_0)_{j_3j_4}+({\bf{\Sigma}}^{11}_0)_{j_1j_3}({\bf{\Sigma}}^{11}_0)_{j_2j_4}+({\bf{\Sigma}}^{11}_0)_{j_1j_4}({\bf{\Sigma}}^{11}_0)_{j_2j_3}\bigr\}\\
    &\qquad\qquad\qquad\qquad\qquad\qquad\qquad\qquad+R_{i-1}(h_n^{5\rho+1},\xi,\delta,\varepsilon,\zeta)+
    R_{i-1}(h_n^3,\xi,\delta,\varepsilon,\zeta)
    \end{split}\label{EX1X1X1X1}
\end{align}
for $j_1,j_2,j_3,j_4=1,\ldots,p_1$.
\end{proposition}
\begin{proof}
We only prove (\ref{EX1X1X1X1}). By noting that $3\rho+1\geq 2$, in an analogous manner, (\ref{EX1X1}) can be shown. 
First, we define the stochastic processes $\{\xi^c_t\}_{t\geq 0}$ and $\{\delta^c_t\}_{t\geq 0}$ by the following diffusion processes
\begin{align*}
    d\xi^c_{t}&=a_{1}(\xi^c_{t})dt+{\bf{S}}_{1,0}d W_{1,t},\quad \xi_0^c=x_{1,0},\\
    d\delta^c_{t}&=a_{2}(\delta^c_{t})dt+{\bf{S}}_{2,0}d W_{2,t},\quad \delta_0^c=x_{2,0}
\end{align*}
independent of $J_{i,j}^n$. Let $X^c_{1,t}={\bf{\Lambda}}_{1,0}\xi_t^{c}+\delta_t^c$. In a similar way to Lemma 21 in Kusano and Uchida \cite{Kusano(2023)}, it is shown that
\begin{align}
\begin{split}
    &\quad\ {\bf{E}}\Bigl[(\Delta_i^n X^c_1)^{(j_1)}(\Delta_i^n X^c_1)^{(j_2)}(\Delta_i^n X^c_1)^{(j_3)}(\Delta_i^n X^c_1)^{(j_4)}\big|\mathcal{F}_{i-1}^n\Bigr]\\
    &=h_n^2 \bigl\{({\bf{\Sigma}}^{11}_{0})_{j_1j_2}({\bf{\Sigma}}^{11}_{0})_{j_3j_4}+({\bf{\Sigma}}^{11}_{0})_{j_1j_3}({\bf{\Sigma}}^{11}_{0})_{j_2j_4}+({\bf{\Sigma}}^{11}_{0})_{j_1j_4}({\bf{\Sigma}}^{11}_{0})_{j_2j_3}\bigr\}+
    R_{i-1}(h_n^3,\xi,\delta),
\end{split}\label{EX1X1X1X1c}
\end{align}
where 
\begin{align*}
    R_{i-1}(h_n^3,\xi,\delta)=R_{i-1}(h_n^3,\xi)+R_{i-1}(h_n^3,\delta).
\end{align*}
Since
\begin{align*}
    \bigl\{J_{i,1}^n=0, J_{i,2}^n=0, J_{i,3}^n=0, J_{i,4}^n=0\bigr\}=C^{n}_{i,0,0,0,0}\cup
    D^{n}_{i,0,0,0,0}
\end{align*}
and $C^{n}_{i,0,0,0,0}\cap D^{n}_{i,0,0,0,0}=\phi$, we have
\begin{align*}
   {\bf{1}}_{\{J_{i,1}^n=0, J_{i,2}^n=0, J_{i,3}^n=0, J_{i,4}^n=0\}}
    &={\bf{1}}_{C^{n}_{i,0,0,0,0}}+{\bf{1}}_{D^{n}_{i,0,0,0,0}},
\end{align*}
which implies that
\begin{align}
\begin{split}
    &\quad\ {\bf{E}}\Bigl[(\Delta_i^n X_1)^{(j_1)}(\Delta_i^n X_1)^{(j_2)}(\Delta_i^n X_1)^{(j_3)}(\Delta_i^n X_1)^{(j_4)}{\bf{1}}_{C^{n}_{i,0,0,0,0}}\big|\mathcal{F}_{i-1}^n\Bigr]\\
    &={\bf{E}}\Bigl[(\Delta_i^n X_1)^{(j_1)}(\Delta_i^n X_1)^{(j_2)}(\Delta_i^n X_1)^{(j_3)}(\Delta_i^n X_1)^{(j_4)}{\bf{1}}_{\{J_{i,1}^n=0, J_{i,2}^n=0, J_{i,3}^n=0, J_{i,4}^n=0\}}\big|\mathcal{F}_{i-1}^n\Bigr]\\
    &\qquad\qquad-{\bf{E}}\Bigl[(\Delta_i^n X_1)^{(j_1)}(\Delta_i^n X_1)^{(j_2)}(\Delta_i^n X_1)^{(j_3)}(\Delta_i^n X_1)^{(j_4)}{\bf{1}}_{D^{n}_{i,0,0,0,0}}\big|\mathcal{F}_{i-1}^n\Bigr].
\end{split}\label{ECJD}
\end{align}
The independence of $\{\xi_t\}_{t\geq 0}$, $\{\delta_t\}_{t\geq 0}$, $\{\varepsilon_t\}_{t\geq 0}$ and $\{\zeta_t\}_{t\geq 0}$ yields 
\begin{align*}
    &\quad\ {\bf{P}}\Bigl(J_{i,1}^n=0, J_{i,2}^n=0, J_{i,3}^n=0, J_{i,4}^n=0
    \big|\mathcal{F}_{i-1}^n\Bigr)\\
    &={\bf{P}}\Bigl(J_{i,1}^n=0 \big|\mathcal{F}_{i-1}^n\Bigr){\bf{P}}\Bigl(J_{i,2}^n=0 \big|\mathcal{F}_{i-1}^n\Bigr){\bf{P}}\Bigl(J_{i,3}^n=0 \big|\mathcal{F}_{i-1}^n\Bigr){\bf{P}}\Bigl(J_{i,4}^n=0 \big|\mathcal{F}_{i-1}^n\Bigr)=e^{-\lambda_{0}h_n},
\end{align*}
so that it holds from (\ref{EX1X1X1X1c}) that
\begin{align*}
    &\quad\ {\bf{E}}\Bigl[(\Delta_i^n X_1)^{(j_1)}(\Delta_i^n X_1)^{(j_2)}(\Delta_i^n X_1)^{(j_3)}(\Delta_i^n X_1)^{(j_4)}{\bf{1}}_{\{J_{i,1}^n=0, J_{i,2}^n=0, J_{i,3}^n=0, J_{i,4}^n=0\}}\big|\mathcal{F}_{i-1}^n\Bigr]\\
    &={\bf{E}}\Bigl[(\Delta_i^n X^c_1)^{(j_1)}(\Delta_i^n X^c_1)^{(j_2)}(\Delta_i^n X^c_1)^{(j_3)}(\Delta_i^n X^c_1)^{(j_4)}{\bf{1}}_{\{J_{i,1}^n=0, J_{i,2}^n=0, J_{i,3}^n=0, J_{i,4}^n=0\}}\big|\mathcal{F}_{i-1}^n\Bigr]\\
    &={\bf{E}}\Bigl[(\Delta_i^n X^c_1)^{(j_1)}(\Delta_i^n X^c_1)^{(j_2)}(\Delta_i^n X^c_1)^{(j_3)}(\Delta_i^n X^c_1)^{(j_4)}\big|\mathcal{F}_{i-1}^n\Bigr]\\
    &\qquad\qquad\qquad\qquad\qquad\times{\bf{P}}\Bigl(J_{i,1}^n=0, J_{i,2}^n=0, J_{i,3}^n=0, J_{i,4}^n=0
    \big|\mathcal{F}_{i-1}^n\Bigr)\\
    &=h_n^2 \bigl\{({\bf{\Sigma}}^{11}_{0})_{j_1j_2}({\bf{\Sigma}}^{11}_{0})_{j_3j_4}+({\bf{\Sigma}}^{11}_{0})_{j_1j_3}({\bf{\Sigma}}^{11}_{0})_{j_2j_4}+({\bf{\Sigma}}^{11}_{0})_{j_1j_4}({\bf{\Sigma}}^{11}_{0})_{j_2j_3}\bigr\}\\
    &\quad\ +h_n^2\bigl(e^{-\lambda_{0} h_n}-1\bigr)\bigl\{({\bf{\Sigma}}^{11}_{0})_{j_1j_2}({\bf{\Sigma}}^{11}_{0})_{j_3j_4}+({\bf{\Sigma}}^{11}_{0})_{j_1j_3}({\bf{\Sigma}}^{11}_{0})_{j_2j_4}
    +({\bf{\Sigma}}^{11}_{0})_{j_1j_4}({\bf{\Sigma}}^{11}_{0})_{j_2j_3}\bigr\}\\
    &\qquad\qquad\qquad\qquad\qquad\quad+h_n^3 R_{i-1}(1,\xi,\delta)\\
    &=h_n^2 \bigl\{({\bf{\Sigma}}^{11}_{0})_{j_1j_2}({\bf{\Sigma}}^{11}_{0})_{j_3j_4}+({\bf{\Sigma}}^{11}_{0})_{j_1j_3}({\bf{\Sigma}}^{11}_{0})_{j_2j_4}+({\bf{\Sigma}}^{11}_{0})_{j_1j_4}({\bf{\Sigma}}^{11}_{0})_{j_2j_3}\bigr\}
    +R_{i-1}(h_n^3,\xi,\delta).
\end{align*}
On the other hand, the Cauchy-Schwartz inequality and Lemma \ref{Dlemma} yield 
\begin{align*}
    &\quad\ \Bigl|{\bf{E}}\Bigl[(\Delta_i^n X_1)^{(j_1)}(\Delta_i^n X_1)^{(j_2)}(\Delta_i^n X_1)^{(j_3)}(\Delta_i^n X_1)^{(j_4)}{\bf{1}}_{D^{n}_{i,0,0,0,0}}\big|\mathcal{F}_{i-1}^n\Bigr]\Bigr|\\
    &\leq {\bf{E}}\Bigl[\bigl|(\Delta_i^n X_1)^{(j_1)}(\Delta_i^n X_1)^{(j_2)}(\Delta_i^n X_1)^{(j_3)}(\Delta_i^n X_1)^{(j_4)}\bigr|^2\big|\mathcal{F}_{i-1}^n\Bigr]^{\frac{1}{2}}{\bf{P}}\Bigl(D^{n}_{i,0,0,0,0}\big|\mathcal{F}_{i-1}^n\Bigr)^{\frac{1}{2}}\\
    &\leq {\bf{E}}\Bigl[|\Delta_i^n X_{1}|^8\big|\mathcal{F}_{i-1}^n\Bigr]^{\frac{1}{2}}{\bf{P}}\Bigl(D^{n}_{i,0,0,0,0}\big|\mathcal{F}_{i-1}^n\Bigr)^{\frac{1}{2}}\\
    &\leq \Bigl\{C{\bf{E}}\Bigl[|\Delta_{i}^n \xi|^8\big|\mathcal{F}_{i-1}^n\Bigr]^{\frac{1}{2}}+C{\bf{E}}\Bigl[|\Delta_{i}^n \delta|^8\big|\mathcal{F}_{i-1}^n\Bigr]^{\frac{1}{2}}\Bigr\}
    R_{i-1}(h_n^{\frac{p}{2}},\xi,\delta,\varepsilon,\zeta)\\
    &\leq R_{i-1}(h_n^{\frac{p}{2}+\frac{1}{2}},\xi,\delta,\varepsilon,\zeta)
\end{align*}
for any $p\geq 1$ since it is shown that
\begin{align*}
    {\bf{E}}\Bigl[\bigl|\Delta_{i}^n \xi\bigr|^l\big|\mathcal{F}_{i-1}^n\Bigr]=R_{i-1}(h_n,\xi),\quad 
    {\bf{E}}\Bigl[\bigl|\Delta_{i}^n \delta\bigr|^l\big|\mathcal{F}_{i-1}^n\Bigr]=R_{i-1}(h_n,\delta)
\end{align*}
for any $l\geq 2$ in an analogous manner to Proposition 3.1 in Shimizu and Yoshida \cite{Shimizu(2006)}.
Consequently, it follows from (\ref{ECJD}) that
\begin{align}
\begin{split}
    &\quad\ {\bf{E}}\Bigl[(\Delta_i^n X_1)^{(j_1)}(\Delta_i^n X_1)^{(j_2)}(\Delta_i^n X_1)^{(j_3)}(\Delta_i^n X_1)^{(j_4)}{\bf{1}}_{C^{n}_{i,0,0,0,0}}\big|\mathcal{F}_{i-1}^n\Bigr]\\
    &=h_n^2 \bigl\{({\bf{\Sigma}}^{11}_{0})_{j_1j_2}({\bf{\Sigma}}^{11}_{0})_{j_3j_4}+({\bf{\Sigma}}^{11}_{0})_{j_1j_3}({\bf{\Sigma}}^{11}_{0})_{j_2j_4}+({\bf{\Sigma}}^{11}_{0})_{j_1j_4}({\bf{\Sigma}}^{11}_{0})_{j_2j_3}\bigr\}
    +R_{i-1}(h_n^3,\xi,\delta,\varepsilon,\zeta).
\end{split}\label{C0000}
\end{align}
Recall that
\begin{align*}
    \bigl\{|\Delta_{i}^n X|\leq Dh_n^{\rho}\bigr\}=\bigcup_{k_1,k_2,k_3,k_4=0,1,2}C^{n}_{i,k_1,k_2,k_3,k_4}
\end{align*}
and $C^{n}_{i,k_1,k_2,k_3,k_4}$ are disjoint. Since
\begin{align*}
    {\bf{1}}_{\{|\Delta_i^n X|\leq Dh_n^{\rho}\}}=\sum_{k_1=0}^2 \sum_{k_2=0}^2 \sum_{k_3=0}^2 \sum_{k_4=0}^2{\bf{1}}_{C^{n}_{i,k_1,k_2,k_3,k_4}},
\end{align*}
we see
\begin{align*}
    &\quad\ {\bf{E}}\Bigl[(\Delta_i^n X_1)^{(j_1)}(\Delta_i^n X_1)^{(j_2)}(\Delta_i^n X_1)^{(j_3)}(\Delta_i^n X_1)^{(j_4)}{\bf{1}}_{\{|\Delta_i^n X|\leq Dh_n^{\rho}\}}\big|\mathcal{F}_{i-1}^n\Bigr]\\
    &={\bf{E}}\Bigl[(\Delta_i^n X_1)^{(j_1)}(\Delta_i^n X_1)^{(j_2)}(\Delta_i^n X_1)^{(j_3)}(\Delta_i^n X_1)^{(j_4)}{\bf{1}}_{C^{n}_{i,0,0,0,0}}\big|\mathcal{F}_{i-1}^n\Bigr]\\
    &\quad+{\bf{E}}\Bigl[(\Delta_i^n X_1)^{(j_1)}(\Delta_i^n X_1)^{(j_2)}(\Delta_i^n X_1)^{(j_3)}(\Delta_i^n X_1)^{(j_4)}{\bf{1}}_{C^{n}_{i,1,1,1,1}}\big|\mathcal{F}_{i-1}^n\Bigr]\\
    &\quad+\sum_{(k_1,k_2,k_3,k_4)\in K_1} {\bf{E}}\Bigl[(\Delta_i^n X_1)^{(j_1)}(\Delta_i^n X_1)^{(j_2)}(\Delta_i^n X_1)^{(j_3)}(\Delta_i^n X_1)^{(j_4)}{\bf{1}}_{C^{n}_{i,k_1,k_2,k_3,k_4}}\big|\mathcal{F}_{i-1}^n\Bigr]\\
    &\quad+\sum_{(k_1,k_2,k_3,k_4)\in K_2} {\bf{E}}\Bigl[(\Delta_i^n X_1)^{(j_1)}(\Delta_i^n X_1)^{(j_2)}(\Delta_i^n X_1)^{(j_3)}(\Delta_i^n X_1)^{(j_4)}{\bf{1}}_{C^{n}_{i,k_1,k_2,k_3,k_4}}\big|\mathcal{F}_{i-1}^n\Bigr]\\
    &\quad+\sum_{(k_1,k_2,k_3,k_4)\in K_3} {\bf{E}}\Bigl[(\Delta_i^n X_1)^{(j_1)}(\Delta_i^n X_1)^{(j_2)}(\Delta_i^n X_1)^{(j_3)}(\Delta_i^n X_1)^{(j_4)}{\bf{1}}_{C^{n}_{i,k_1,k_2,k_3,k_4}}\big|\mathcal{F}_{i-1}^n\Bigr]\\
    &\quad+\sum_{(k_1,k_2,k_3,k_4)\in K_4} {\bf{E}}\Bigl[(\Delta_i^n X_1)^{(j_1)}(\Delta_i^n X_1)^{(j_2)}(\Delta_i^n X_1)^{(j_3)}(\Delta_i^n X_1)^{(j_4)}{\bf{1}}_{C^{n}_{i,k_1,k_2,k_3,k_4}}\big|\mathcal{F}_{i-1}^n\Bigr].
\end{align*}
It holds from Lemma \ref{Clemma} that
\begin{align*}
    &\quad\ \Bigl|{\bf{E}}\Bigl[(\Delta_i^n X_1)^{(j_1)}(\Delta_i^n X_1)^{(j_2)}(\Delta_i^n X_1)^{(j_3)}(\Delta_i^n X_1)^{(j_4)}{\bf{1}}_{C^{n}_{i,1,1,1,1}}\big|\mathcal{F}_{i-1}^n\Bigr]\Bigr|\\
    &\leq {\bf{E}}\Bigl[\bigl|(\Delta_i^n X_1)^{(j_1)}\bigr|\bigl|(\Delta_i^n X_1)^{(j_2)}\bigr|\bigl|(\Delta_i^n X_1)^{(j_3)}\bigr|\bigl|(\Delta_i^n X_1)^{(j_4)}\bigr|{\bf{1}}_{C^{n}_{i,1,1,1,1}}\big|\mathcal{F}_{i-1}^n\Bigr]\\
    &\leq {\bf{E}}\Bigl[|\Delta_i^n X|^4{\bf{1}}_{C^{n}_{i,1,1,1,1}}\big|\mathcal{F}_{i-1}^n\Bigr]\\
    &\leq D^4h_n^{4\rho}{\bf{P}}\Bigl(C^{n}_{i,1,1,1,1}\big|\mathcal{F}_{i-1}^n\Bigr)\\
    &\leq D^4\lambda_{1,0}\lambda_{2,0}
    \lambda_{3,0}\lambda_{4,0}h_n^{4\rho+4},
\end{align*}
which implies
\begin{align*}
    {\bf{E}}\Bigl[(\Delta_i^n X_1)^{(j_1)}(\Delta_i^n X_1)^{(j_2)}(\Delta_i^n X_1)^{(j_3)}(\Delta_i^n X_1)^{(j_4)}{\bf{1}}_{C^{n}_{i,1,1,1,1}}\big|\mathcal{F}_{i-1}^n\Bigr]=R_{i-1}(h_n^{4\rho+4},\xi,\delta,\varepsilon,\zeta).
\end{align*}
Similarly, for a sufficiently large $n$, Lemma \ref{Clemma} implies that 
\begin{align*}
    {\bf{E}}\Bigl[(\Delta_i^n X_1)^{(j_1)}(\Delta_i^n X_1)^{(j_2)}(\Delta_i^n X_1)^{(j_3)}(\Delta_i^n X_1)^{(j_4)}{\bf{1}}_{C^{n}_{i,k_1,k_2,k_3,k_4}}\big|\mathcal{F}_{i-1}^n\Bigr]=R_{i-1}(h_n^{5\rho+1},\xi,\delta,\varepsilon,\zeta)
\end{align*}
for $(k_1,k_2,k_3,k_4)\in K_1$, 
\begin{align*}
    {\bf{E}}\Bigl[(\Delta_i^n X_1)^{(j_1)}(\Delta_i^n X_1)^{(j_2)}(\Delta_i^n X_1)^{(j_3)}(\Delta_i^n X_1)^{(j_4)}{\bf{1}}_{C^{n}_{i,k_1,k_2,k_3,k_4}}\big|\mathcal{F}_{i-1}^n\Bigr]=R_{i-1}(h_n^{4\rho+2},\xi,\delta,\varepsilon,\zeta)
\end{align*}
for $(k_1,k_2,k_3,k_4)\in K_2$, 
\begin{align*}
    {\bf{E}}\Bigl[(\Delta_i^n X_1)^{(j_1)}(\Delta_i^n X_1)^{(j_2)}(\Delta_i^n X_1)^{(j_3)}(\Delta_i^n X_1)^{(j_4)}{\bf{1}}_{C^{n}_{i,k_1,k_2,k_3,k_4}}\big|\mathcal{F}_{i-1}^n\Bigr]=R_{i-1}(h_n^{4\rho+3},\xi,\delta,\varepsilon,\zeta)
\end{align*}
for $(k_1,k_2,k_3,k_4)\in K_3$, and
\begin{align*}
    {\bf{E}}\Bigl[(\Delta_i^n X_1)^{(j_1)}(\Delta_i^n X_1)^{(j_2)}(\Delta_i^n X_1)^{(j_3)}(\Delta_i^n X_1)^{(j_4)}{\bf{1}}_{C^{n}_{i,k_1,k_2,k_3,k_4}}\big|\mathcal{F}_{i-1}^n\Bigr]=R_{i-1}(h_n^{4\rho+2},\xi,\delta,\varepsilon,\zeta)
\end{align*}
for $(k_1,k_2,k_3,k_4)\in K_4$. Since $5\rho+1<4\rho+2$, it follows from (\ref{C0000}) that
\begin{align*}
    &\quad\ {\bf{E}}\Bigl[(\Delta_i^n X_1)^{(j_1)}(\Delta_i^n X_1)^{(j_2)}(\Delta_i^n X_1)^{(j_3)}(\Delta_i^n X_1)^{(j_4)}{\bf{1}}_{\{|\Delta_i^n X|\leq Dh_n^{\rho}\}}\big|\mathcal{F}_{i-1}^n\Bigr]\\
    &=h_n^2\bigl\{({\bf{\Sigma}}^{11}_0)_{j_1j_2}({\bf{\Sigma}}^{11}_0)_{j_3j_4}+({\bf{\Sigma}}^{11}_0)_{j_1j_3}({\bf{\Sigma}}^{11}_0)_{j_2j_4}+({\bf{\Sigma}}^{11}_0)_{j_1j_4}({\bf{\Sigma}}^{11}_0)_{j_2j_3}\bigr\}\\
    &\qquad+R_{i-1}(h_n^{5\rho+1},\xi,\delta,\varepsilon,\zeta)+
    R_{i-1}(h_n^3,\xi,\delta,\varepsilon,\zeta)
\end{align*}
for a sufficiently large $n$, which completes the proof.
\end{proof}
\begin{proposition}\label{X2momentprop}
Under {\rm{\textbf{[A1]}}}-{\rm{\textbf{[A4]}}}, for a sufficiently large $n$, 
\begin{align*}
    {\bf{E}}\Bigl[(\Delta_i^n X_2)^{(j_1)}(\Delta_i^n X_2)^{(j_2)}{\bf{1}}_{\{|\Delta_i^n X|\leq Dh_n^{\rho}\}}\big|\mathcal{F}_{i-1}^n\Bigr]=h_n ({\bf{\Sigma}}^{22}_0)_{j_1j_2}
    +R_{i-1}(h_n^2,\xi,\delta,\varepsilon,\zeta)
\end{align*}
and
\begin{align*}
    &\quad\ {\bf{E}}\Bigl[(\Delta_i^n X_2)^{(j_1)}(\Delta_i^n X_2)^{(j_2)}(\Delta_i^n X_2)^{(j_3)}(\Delta_i^n X_2)^{(j_4)}{\bf{1}}_{\{|\Delta_i^n X|\leq Dh_n^{\rho}\}}\big|\mathcal{F}_{i-1}^n\Bigr]\\
    &=h_n^2\bigl\{({\bf{\Sigma}}^{22}_0)_{j_1j_2}({\bf{\Sigma}}^{22}_0)_{j_3j_4}+({\bf{\Sigma}}^{22}_0)_{j_1j_3}({\bf{\Sigma}}^{22}_0)_{j_2j_4}+({\bf{\Sigma}}^{22}_0)_{j_1j_4}({\bf{\Sigma}}^{22}_0)_{j_2j_3}\bigr\}\\
    &\qquad\qquad\qquad\qquad\qquad\qquad\qquad\qquad+R_{i-1}(h_n^{5\rho+1},\xi,\delta,\varepsilon,\zeta)+
    R_{i-1}(h_n^3,\xi,\delta,\varepsilon,\zeta)
\end{align*}
for $j_1,j_2,j_3,j_4=1,\ldots,p_2$.
\end{proposition}
\begin{proposition}\label{X12momentprop}
Under {\bf{[A1]}}-{\bf{[A4]}}, for a sufficient large $n$, 
\begin{align*}
    {\bf{E}}\Bigl[(\Delta_i^n X_1)^{(j_1)}(\Delta_i^n X_2)^{(j_2)}{\bf{1}}_{\{|\Delta_i^n X|\leq Dh_n^{\rho}\}}\big|\mathcal{F}_{i-1}^n\Bigr]
    =h_n ({\bf{\Sigma}}^{12}_{0})_{j_1j_2}+h_n^2 R_{i-1}(1,\xi,\delta,\varepsilon,\zeta)
\end{align*}
for $j_1=1,\ldots,p_1,\ j_2=1,\ldots,p_2$, 
\begin{align*}
    &\quad\ {\bf{E}}\Bigl[(\Delta_i^n X_1)^{(j_1)}(\Delta_i^n X_1)^{(j_2)}(\Delta_i^n X_1)^{(j_3)}(\Delta_i^n X_2)^{(j_4)}{\bf{1}}_{\{|\Delta_i^n X|\leq Dh_n^{\rho}\}}\big|\mathcal{F}_{i-1}^n\Bigr]\\
    &=h_n^2\bigl\{({\bf{\Sigma}}^{11}_{0})_{j_1j_2}({\bf{\Sigma}}^{12}_{0})_{j_3j_4}+({\bf{\Sigma}}^{11}_{0})_{j_1j_3}({\bf{\Sigma}}^{12}_{0})_{j_2j_4}+({\bf{\Sigma}}^{12}_{0})_{j_1j_4}({\bf{\Sigma}}^{11}_{0})_{j_2j_3}\bigr\}\\
    &\qquad\qquad\qquad\qquad\qquad\qquad\qquad\qquad+R_{i-1}(h_n^{5\rho+1},\xi,\delta,\varepsilon,\zeta)+R_{i-1}(h_n^3,\xi,\delta,\varepsilon,\zeta)
\end{align*}
for $j_1,j_2,j_3=1,\ldots,p_1,\ j_4=1,\ldots,p_2$,
\begin{align*}
    &\quad\ {\bf{E}}\Bigl[(\Delta_i^n X_1)^{(j_1)}(\Delta_i^n X_1)^{(j_2)}(\Delta_i^n X_2)^{(j_3)}(\Delta_i^n X_2)^{(j_4)}{\bf{1}}_{\{|\Delta_i^n X|\leq Dh_n^{\rho}\}}\big|\mathcal{F}_{i-1}^n\Bigr]\\
    &=h_n^2\bigl\{({\bf{\Sigma}}^{11}_{0})_{j_1j_2}({\bf{\Sigma}}^{22}_{0})_{j_3j_4}+({\bf{\Sigma}}^{12}_{0})_{j_1j_3}({\bf{\Sigma}}^{12}_{0})_{j_2j_4}+({\bf{\Sigma}}^{12}_{0})_{j_1j_4}({\bf{\Sigma}}^{12}_{0})_{j_2j_3}\bigr\}\\
    &\qquad\qquad\qquad\qquad\qquad\qquad\qquad\qquad+R_{i-1}(h_n^{5\rho+1},\xi,\delta,\varepsilon,\zeta)+R_{i-1}(h_n^3,\xi,\delta,\varepsilon,\zeta)
\end{align*}
for $j_1,j_2=1,\ldots,p_1,\ j_3,j_4=1,\ldots,p_2$, and
\begin{align*}
    &\quad\ {\bf{E}}\Bigl[(\Delta_i^n X_1)^{(j_1)}(\Delta_i^n X_2)^{(j_2)}(\Delta_i^n X_2)^{(j_3)}(\Delta_i^n X_2)^{(j_4)}{\bf{1}}_{\{|\Delta_i^n X|\leq Dh_n^{\rho}\}}\big|\mathcal{F}_{i-1}^n\Bigr]\\
    &=h_n^2\bigl\{({\bf{\Sigma}}^{12}_{0})_{j_1j_2}({\bf{\Sigma}}^{22}_{0})_{j_3j_4}+({\bf{\Sigma}}^{12}_{0})_{j_1j_3}({\bf{\Sigma}}^{22}_{0})_{j_2j_4}+({\bf{\Sigma}}^{12}_{0})_{j_1j_4}({\bf{\Sigma}}^{22}_{0})_{j_2j_3}\bigr\}\\
    &\qquad\qquad\qquad\qquad\qquad\qquad\qquad\qquad+R_{i-1}(h_n^{5\rho+1},\xi,\delta,\varepsilon,\zeta)+R_{i-1}(h_n^3,\xi,\delta,\varepsilon,\zeta)
\end{align*}
for $j_1=1,\ldots,p_1,\ j_2,j_3,j_4=1,\ldots,p_2$.
\end{proposition}
\begin{proof}[\textbf{Proofs of Propositions \ref{X2momentprop}-\ref{X12momentprop}}]
In an analogous manner to Proposition \ref{X1momentprop}, these results can be shown.
\end{proof}
\begin{proposition}\label{Econvprop1}
Under {\bf{[A1]}}-{\bf{[A4]}}, as $n\longrightarrow\infty$, 
\begin{align}
    \frac{1}{nh_n}\sum_{i=1}^n{\bf{E}}\Bigl[(\Delta_i^n X)^{(j_1)}(\Delta_i^n X)^{(j_2)}{\bf{1}}_{\{|\Delta_i^n X|\leq Dh_n^{\rho}\}}\big|\mathcal{F}_{i-1}^n\Bigr]\stackrel{p}{\longrightarrow}({\bf{\Sigma}}_0)_{j_1j_2}\label{XXprob} 
\end{align}
and
\begin{align}
    \frac{1}{n^2h_n^2}\sum_{i=1}^n{\bf{E}}\Bigl[(\Delta_i^n X)^{(j_1)}(\Delta_i^n X)^{(j_2)}(\Delta_i^n X)^{(j_3)}(\Delta_i^n X)^{(j_4)}{\bf{1}}_{\{|\Delta_i^n X|\leq Dh_n^{\rho}\}}\big|\mathcal{F}_{i-1}^n\Bigr]\stackrel{p}{\longrightarrow}0\label{XXXXprob1}
\end{align}
for $j_1,j_2,j_3,j_4=1,\ldots,p$.
\end{proposition}
\begin{proof}
Proposition \ref{X1momentprop} implies
\begin{align*}
    \frac{1}{nh_n}\sum_{i=1}^n {\bf{E}}\Bigl[(\Delta_i^n X_1)^{(j_1)}(\Delta_i^n X_1)^{(j_2)}{\bf{1}}_{\{|\Delta_i^n X|\leq Dh_n^{\rho}\}}\big|\mathcal{F}_{i-1}^n\Bigr]
    \stackrel{p}{\longrightarrow}({\bf{\Sigma}}_0^{11})_{j_1j_2}
\end{align*}
for $j_1,j_2=1,\ldots,p_1$. 
Propositions \ref{X2momentprop}-\ref{X12momentprop} imply that
\begin{align*}
    \frac{1}{nh_n}\sum_{i=1}^n {\bf{E}}\Bigl[(\Delta_i^n X_1)^{(j_1)}(\Delta_i^n X_2)^{(j_2)}{\bf{1}}_{\{|\Delta_i^n X|\leq Dh_n^{\rho}\}}\big|\mathcal{F}_{i-1}^n\Bigr]\stackrel{p}{\longrightarrow}({\bf{\Sigma}}_0^{12})_{j_1j_2}
\end{align*}
for $j_1=1,\ldots,p_1$ and $j_2=1,\ldots,p_2$, and
\begin{align*}
    \frac{1}{nh_n}\sum_{i=1}^n {\bf{E}}\Bigl[(\Delta_i^n X_2)^{(j_1)}(\Delta_i^n X_2)^{(j_2)}{\bf{1}}_{\{|\Delta_i^n X|\leq Dh_n^{\rho}\}}\big|\mathcal{F}_{i-1}^n\Bigr]\stackrel{p}{\longrightarrow}({\bf{\Sigma}}_0^{22})_{j_1j_2}
\end{align*}
$j_1,j_2=1,\ldots,p_2$, which yields (\ref{XXprob}). Since $5\rho+1>2$, by using Propositions \ref{X1momentprop}-\ref{X12momentprop}, (\ref{XXXXprob1}) can be shown in an analogous way.
\end{proof}
\begin{proposition}\label{Econvprop2}
Under {\bf{[A1]}}-{\bf{[A4]}}, as $n\longrightarrow\infty$,
\begin{align}
    &\frac{1}{\sqrt{n}h_n}\sum_{i=1}^n \Bigl\{{\bf{E}}\Bigl[(\Delta_i^n X)^{(j_1)}(\Delta_i^n X)^{(j_2)}{\bf{1}}_{\{|\Delta_i^n X|\leq Dh_n^{\rho}\}}\big|\mathcal{F}_{i-1}^n\Bigr]-h_n({\bf{\Sigma}}_0)_{j_1j_2}\Bigr\}\stackrel{p}{\longrightarrow}0,\label{sqrtXXprob}\\
    \begin{split}
    &\frac{1}{nh_n^2}\sum_{i=1}^n{\bf{E}}\Bigl[(\Delta_i^n X)^{(j_1)}(\Delta_i^n X)^{(j_2)}{\bf{1}}_{\{|\Delta_i^n X|\leq Dh_n^{\rho}\}}\big|\mathcal{F}_{i-1}^n\Bigr]\\
    &\qquad\qquad\quad\times {\bf{E}}\Bigl[(\Delta_i^n X)^{(j_3)}(\Delta_i^n X)^{(j_4)}{\bf{1}}_{\{|\Delta_i^n X|\leq Dh_n^{\rho}\}}\big|\mathcal{F}_{i-1}^n\Bigr]\stackrel{p}{\longrightarrow}
    ({\bf{\Sigma}}_0)_{j_1j_2}({\bf{\Sigma}}_0)_{j_3j_4},
    \end{split}\\
    \begin{split}
    &\frac{1}{nh_n^2}\sum_{i=1}^n{\bf{E}}\Bigl[(\Delta_i^n X)^{(j_1)}(\Delta_i^n X)^{(j_2)}(\Delta_i^n X)^{(j_3)}(\Delta_i^n X)^{(j_4)}{\bf{1}}_{\{|\Delta_i^n X|\leq Dh_n^{\rho}\}}\big|\mathcal{F}_{i-1}^n\Bigr]\\
    &\qquad\qquad\qquad\qquad\quad\stackrel{p}{\longrightarrow}({\bf{\Sigma}}_{0})_{j_1j_2}({\bf{\Sigma}}_{0})_{j_3j_4}+({\bf{\Sigma}}_{0})_{j_1j_3}({\bf{\Sigma}}_{0})_{j_2j_4}+({\bf{\Sigma}}_{0})_{j_1j_4}({\bf{\Sigma}}_{0})_{j_2j_3} \label{XXXXprob2}
    \end{split}
\end{align}
and
\begin{align}
    &\frac{1}{n^2h_n^4}\sum_{i=1}^n{\bf{E}}\Bigl[\bigl|(\Delta_i^n X)^{(j_1)}(\Delta_i^n X)^{(j_2)}\bigr|^4{\bf{1}}_{\{|\Delta_i^n X|\leq Dh_n^{\rho}\}}\big|\mathcal{F}_{i-1}^n\Bigr]\stackrel{p}{\longrightarrow}0\label{X4prob}
\end{align}
for $j_1,j_2,j_3,j_4=1,\ldots,p$.
\end{proposition}
\begin{proof}
By using Propositions \ref{X1momentprop}-\ref{X12momentprop}, (\ref{sqrtXXprob})-(\ref{XXXXprob2}) can be shown in a similar way to Lemma $6$ in Kusano and Uchida \cite{Kusano(JJSD)}. We prove (\ref{X4prob}). In an analogous manner to the proof of Proposition \ref{X1momentprop}, one gets
\begin{align*}
    &\quad\ {\bf{E}}\Bigl[\bigl|(\Delta_i^n X)^{(j_1)}(\Delta_i^n X)^{(j_2)}\bigr|^4{\bf{1}}_{\{|\Delta_i^n X|\leq Dh_n^{\rho}\}}\big|\mathcal{F}_{i-1}^n\biggr]\\
    &={\bf{E}}\biggl[\bigl|(\Delta_i^n X)^{(j_1)}(\Delta_i^n X)^{(j_2)}\bigr|^4{\bf{1}}_{C^{n}_{i,0,0,0,0}}\big|\mathcal{F}_{i-1}^n\Bigr]+{\bf{E}}\Bigl[\bigl|(\Delta_i^n X)^{(j_1)}(\Delta_i^n X)^{(j_2)}\bigr|^4{\bf{1}}_{C^{n}_{i,1,1,1,1}}\big|\mathcal{F}_{i-1}^n\Bigr]\\
    &\quad+\sum_{(k_1,k_2,k_3,k_4)\in K_1} {\bf{E}}\Bigl[\bigl|(\Delta_i^n X)^{(j_1)}(\Delta_i^n X)^{(j_2)}\bigr|^4{\bf{1}}_{C^{n}_{i,k_1,k_2,k_3,k_4}}\big|\mathcal{F}_{i-1}^n\Bigr]\\
    &\quad+\sum_{(k_1,k_2,k_3,k_4)\in K_2} {\bf{E}}\Bigl[\bigl|(\Delta_i^n X)^{(j_1)}(\Delta_i^n X)^{(j_2)}\bigr|^4{\bf{1}}_{C^{n}_{i,k_1,k_2,k_3,k_4}}\big|\mathcal{F}_{i-1}^n\Bigr]\\
    &\quad+\sum_{(k_1,k_2,k_3,k_4)\in K_3} {\bf{E}}\Bigl[\bigl|(\Delta_i^n X)^{(j_1)}(\Delta_i^n X)^{(j_2)}\bigr|^4{\bf{1}}_{C^{n}_{i,k_1,k_2,k_3,k_4}}\big|\mathcal{F}_{i-1}^n\Bigr]\\
    &\quad+\sum_{(k_1,k_2,k_3,k_4)\in K_4} {\bf{E}}\Bigl[\bigl|(\Delta_i^n X)^{(j_1)}(\Delta_i^n X)^{(j_2)}\bigr|^4{\bf{1}}_{C^{n}_{i,k_1,k_2,k_3,k_4}}\big|\mathcal{F}_{i-1}^n\Bigr]=R_{i-1}(h_n^4,\xi,\delta,\varepsilon,\zeta)
\end{align*}
for a sufficiently large $n$, which yields (\ref{X4prob}).
\end{proof}
\begin{proof}[\textbf{Proof of Theorem \ref{Qtheorem}}]
To show (\ref{Qcons}), it is sufficient to prove 
\begin{align}
    \frac{1}{\tilde{N}_nh_n}\sum_{i=1}^n (\Delta_i^n X)^{(j_1)}(\Delta_i^n X)^{(j_2)}{\bf{1}}_{\{|\Delta_i^n X|\leq Dh_n^{\rho}\}}\stackrel{p}{\longrightarrow}({\bf{\Sigma}}_0)_{j_1j_2} \label{Qconselement}
\end{align}
for $j_1,j_2=1,\ldots,p$. Since it holds from Proposition \ref{Econvprop1} that
\begin{align*}
    \frac{1}{nh_n}\sum_{i=1}^n{\bf{E}}\Bigl[(\Delta_i^n X)^{(j_1)}(\Delta_i^n X)^{(j_2)}{\bf{1}}_{\{|\Delta_i^n X|\leq Dh_n^{\rho}\}}\big|\mathcal{F}_{i-1}^n\Bigr]\stackrel{p}{\longrightarrow} ({\bf{\Sigma}}_0)_{j_1j_2}
\end{align*}
and
\begin{align*}
    \frac{1}{n^2h_n^2}\sum_{i=1}^n{\bf{E}}\Bigl[(\Delta_i^n X)^{(j_1)}(\Delta_i^n X)^{(j_2)}(\Delta_i^n X)^{(j_1)}(\Delta_i^n X)^{(j_2)}{\bf{1}}_{\{|\Delta_i^n X|\leq Dh_n^{\rho}\}}\big|\mathcal{F}_{i-1}^n\Bigr]\stackrel{p}{\longrightarrow}0,
\end{align*}
it follows from Lemma 9 in Genon-Catalot and Jacod \cite{Genon(1993)} that 
\begin{align*}
    \frac{1}{nh_n}\sum_{i=1}^n (\Delta_i^n X)^{(j_1)}(\Delta_i^n X)^{(j_2)}{\bf{1}}_{\{|\Delta_i^n X|\leq Dh_n^{\rho}\}}\stackrel{p}{\longrightarrow}({\bf{\Sigma}}_0)_{j_1j_2},
\end{align*}
so that Proposition \ref{Nprop} implies
\begin{align*}
    \frac{n}{\tilde{N}_n}\times \frac{1}{nh_n}\sum_{i=1}^n (\Delta_i^n X)^{(j_1)}(\Delta_i^n X)^{(j_2)}{\bf{1}}_{\{|\Delta_i^n X|\leq Dh_n^{\rho}\}}\stackrel{p}{\longrightarrow}({\bf{\Sigma}}_0)_{j_1j_2},
\end{align*}
which yields (\ref{Qconselement}). Next, we consider the following convergence:
\begin{align}
    \sqrt{n}(\vec\hat{\bf{\Sigma}}_n-\vec{\bf{\Sigma}}_0)=\sum_{i=1}^n L_{i}^n\stackrel{d}{\longrightarrow}N_{p^2}\bigl(0,{\bf{G}}_0\bigr),\label{vecQprob}
\end{align}
where ${\bf{G}}_0$ is a $p^2\times p^2$ matrix whose elements are 
\begin{align*}
    ({\bf{G}}_0)_{p(j_2-1)+j_1,p(j_4-1)+j_3}=({\bf{\Sigma}}_0)_{j_1j_3}({\bf{\Sigma}}_0)_{j_2j_4}+({\bf{\Sigma}}_0)_{j_1j_4}({\bf{\Sigma}}_0)_{j_2j_3}
\end{align*}
for $j_1,j_2,j_3,j_4=1,\ldots,p$, and
\begin{align*}
    L_i^n=\vec \biggl\{\frac{\sqrt{n}}{\tilde{N}_nh_n}(\Delta_i^n X)(\Delta_i^n X)^{\top}{\bf{1}}_{\{|\Delta_i^n X|\leq Dh_n^{\rho}\}}-\frac{1}{\sqrt{n}}{\bf{\Sigma}}_0\biggr\}.
\end{align*}
Since
\begin{align*}
    \mathbb{D}_p^{+}{\bf{G}}_0\mathbb{D}_p^{+\top}=2\mathbb{D}_p^{+}({\bf{\Sigma}}_0\otimes {\bf{\Sigma}}_0)\mathbb{D}_p^{+\top},
\end{align*}
if (\ref{vecQprob}) holds, (\ref{Qasym}) is shown in a similar way to Theorem $2$ in Kusano and Uchida \cite{Kusano(JJSD)}. This is why we show (\ref{vecQprob}).
In order to prove (\ref{vecQprob}),  it is sufficient to show
\begin{align}
    &\qquad\qquad\qquad\qquad\sum_{i=1}^n {\bf{E}}\bigl[L_i^n|\mathcal{F}_{i-1}^n\bigr]\stackrel{p}{\longrightarrow}0, \label{Lprob}\\
    &\sum_{i=1}^n \Bigr\{{\bf{E}}\bigl[L_i^nL_i^{n\top}|\mathcal{F}_{i-1}^n\bigr]
    -{\bf{E}}\bigl[L_i^n|\mathcal{F}_{i-1}^n\bigr]{\bf{E}}\bigl[L_i^n|\mathcal{F}_{i-1}^n\bigr]^{\top}\Bigr\}\stackrel{p}{\longrightarrow}{\bf{G}}_0  \label{L^2prob}
\end{align}
and
\begin{align}
    \sum_{i=1}^n {\bf{E}}\bigl[|L_i^n|^4|\mathcal{F}_{i-1}^n\bigr]&\stackrel{p}{\longrightarrow}0 \label{L4prob}
\end{align}
by using Theorems 3.2 and 3.4 in Hall and Heyde \cite{Hall(1981)}. Note that 
\begin{align*}
    (\vec A)^{(p(j_2-1)+j_1)}=A_{j_1j_2},\quad (vv^{\top})_{j_1j_2}=v^{(j_1)}v^{(j_2)}
\end{align*}
for any matrix $A\in\mathbb{R}^{p^2\times p^2}$ and vector $v\in\mathbb{R}^{p}$. Proposition \ref{Nprop} and the delta method imply 
\begin{align*}
    \sqrt{n}\Bigl(\frac{n}{\tilde{N}_n}-1\Bigr)\stackrel{p}{\longrightarrow}0,
\end{align*}
so that it follows from Proposition \ref{Econvprop2} that
\begin{align*}
    &\quad\ \biggl(\sum_{i=1}^n {\bf{E}}\bigl[L_i^n|\mathcal{F}_{i-1}^n\bigr]\biggr)^{(p(j_2-1)+j_1)}\\
    &=\frac{n}{\tilde{N}_n}\times \frac{1}{\sqrt{n}h_n}\sum_{i=1}^n \Bigl\{{\bf{E}}\Bigl[(\Delta_i^n X)^{(j_1)}(\Delta_i^n X)^{(j_2)}{\bf{1}}_{\{|\Delta_i^n X|\leq Dh_n^{\rho}\}}\big|\mathcal{F}_{i-1}^n\Bigr]-h_n({\bf{\Sigma}}_0)_{j_1j_2}\Bigr\}\\
    &\qquad+({\bf{\Sigma}}_0)_{j_1j_2}\sqrt{n}\Bigl(\frac{n}{\tilde{N}_n}-1\Bigr)\stackrel{p}{\longrightarrow}0,
\end{align*}
which yields (\ref{Lprob}). Since 
\begin{align*}
    &\quad\ \Bigl({\bf{E}}\bigl[L_i^nL_i^{n\top}|\mathcal{F}_{i-1}^n\bigr]\Bigr)_{p(j_2-1)+j_1,p(j_4-1)+j_3}\\
    &=\frac{n}{\tilde{N}_n^2h_n^2}{\bf{E}}\Bigl[(\Delta_i^n X)^{(j_1)}(\Delta_i^n X)^{(j_2)}(\Delta_i^n X)^{(j_3)}(\Delta_i^n X)^{(j_4)}{\bf{1}}_{\{|\Delta_i^n X|\leq Dh_n^{\rho}\}}\big|\mathcal{F}_{i-1}^n\Bigr]\\
    &\quad-\frac{1}{\tilde{N}_nh_n}{\bf{E}}\Bigl[(\Delta_i^n X)^{(j_1)}(\Delta_i^n X)^{(j_2)}{\bf{1}}_{\{|\Delta_i^n X|\leq Dh_n^{\rho}\}}\big|\mathcal{F}_{i-1}^n\Bigr]({\bf{\Sigma}}_0)_{j_3j_4}\\
    &\quad-\frac{1}{\tilde{N}_nh_n}\bigl({\bf{\Sigma}}_0\bigr)_{j_1j_2}{\bf{E}}\Bigl[(\Delta_i^n X)^{(j_3)}(\Delta_i^n X)^{(j_4)}{\bf{1}}_{\{|\Delta_i^n X|\leq Dh_n^{\rho}\}}\big|\mathcal{F}_{i-1}^n\Bigr]+\frac{1}{n}({\bf{\Sigma}}_0)_{j_1j_2}({\bf{\Sigma}}_0)_{j_3j_4}
\end{align*}
and
\begin{align*}
    &\quad\ \Bigl({\bf{E}}\bigl[L_i^n|\mathcal{F}_{i-1}^n\bigr]{\bf{E}}\bigl[L_i^n|\mathcal{F}_{i-1}^n\bigr]^{\top}\Bigr)_{p(j_2-1)+j_1,p(j_4-1)+j_3}\\
    &=\frac{n}{\tilde{N}_n^2h_n^2}{\bf{E}}\Bigl[(\Delta_i^n X)^{(j_1)}(\Delta_i^n X)^{(j_2)}{\bf{1}}_{\{|\Delta_i^n X|\leq Dh_n^{\rho}\}}\big|\mathcal{F}_{i-1}^n\Bigr]{\bf{E}}\Bigl[(\Delta_i^n X)^{(j_3)}(\Delta_i^n X)^{(j_4)}{\bf{1}}_{\{|\Delta_i^n X|\leq Dh_n^{\rho}\}}\big|\mathcal{F}_{i-1}^n\Bigr]\\
    &\quad-\frac{1}{\tilde{N}_nh_n}{\bf{E}}\Bigl[(\Delta_i^n X)^{(j_1)}(\Delta_i^n X)^{(j_2)}{\bf{1}}_{\{|\Delta_i^n X|\leq Dh_n^{\rho}\}}\big|\mathcal{F}_{i-1}^n\Bigr]({\bf{\Sigma}}_0)_{j_3j_4}\\
    &\quad-\frac{1}{\tilde{N}_nh_n}({\bf{\Sigma}}_0)_{j_1j_2}{\bf{E}}\Bigl[(\Delta_i^n X)^{(j_3)}(\Delta_i^n X)^{(j_4)}{\bf{1}}_{\{|\Delta_i^n X|\leq Dh_n^{\rho}\}}\big|\mathcal{F}_{i-1}^n\Bigr]+\frac{1}{n}({\bf{\Sigma}}_0)_{j_1j_2}({\bf{\Sigma}}_0)_{j_3j_4},
\end{align*}
we see from Propositions \ref{Nprop} and \ref{Econvprop2} that
\begin{align*}
    &\quad\ \sum_{i=1}^n\Bigl({\bf{E}}\bigl[L_i^nL_i^{n\top}|\mathcal{F}_{i-1}^n\bigr]-{\bf{E}}\bigl[L_i^n|\mathcal{F}_{i-1}^n\bigr]{\bf{E}}\bigl[L_i^n|\mathcal{F}_{i-1}^n\bigr]^{\top}\Bigr)_{p(j_2-1)+j_1,p(j_4-1)+j_3}\\
    &=\Bigl(\frac{n}{\tilde{N}_n}\Bigr)^2\times \frac{1}{nh_n^2}\sum_{i=1}^n{\bf{E}}\Bigl[(\Delta_i^n X)^{(j_1)}(\Delta_i^n X)^{(j_2)}(\Delta_i^n X)^{(j_3)}(\Delta_i^n X)^{(j_4)}{\bf{1}}_{\{|\Delta_i^n X|\leq Dh_n^{\rho}\}}\big|\mathcal{F}_{i-1}^n\Bigr]\\
    &\qquad-\biggl(\frac{n}{\tilde{N}_n}\biggr)^2\times\frac{1}{nh_n^2}\sum_{i=1}^n{\bf{E}}\Bigl[(\Delta_i^n X)^{(j_1)}(\Delta_i^n X)^{(j_2)}{\bf{1}}_{\{|\Delta_i^n X|\leq Dh_n^{\rho}\}}\big|\mathcal{F}_{i-1}^n\Bigr]\\
    &\qquad\qquad\qquad\qquad\qquad\qquad\qquad\qquad\times{\bf{E}}\Bigl[(\Delta_i^n X)^{(j_3)}(\Delta_i^n X)^{(j_4)}{\bf{1}}_{\{|\Delta_i^n X|\leq Dh_n^{\rho}\}}\big|\mathcal{F}_{i-1}^n\Bigr]\\
    &\stackrel{p}{\longrightarrow}({\bf{\Sigma}}_0)_{j_1j_3}({\bf{\Sigma}}_0)_{j_2j_4}+({\bf{\Sigma}}_0)_{j_1j_4}({\bf{\Sigma}}_0)_{j_2j_3},
\end{align*}
which implies (\ref{L^2prob}). Furthermore, Propositions \ref{Nprop} and \ref{Econvprop2} show
\begin{align*}
    &\quad\ \sum_{i=1}^n {\bf{E}}\Bigl[|L_i^n|^4\big|\mathcal{F}_{i-1}^n\Bigr]\\
    &\leq C\sum_{j_1=1}^{p}\sum_{j_2=1}^p\sum_{i=1}^n{\bf{E}}\Bigl[\bigl|(L_i^n)^{(p(j_2-1)+j_1)}\bigr|^{4}\big|\mathcal{F}_{i-1}^n\Bigr]\\
    &\leq \Bigl(\frac{n}{\tilde{N}_n}\Bigr)^4\times\frac{C}{n^2h_n^4}\sum_{j_1=1}^{p}\sum_{j_2=1}^p\sum_{i=1}^n{\bf{E}}\biggl[\bigl|(\Delta_i^n X)^{(j_1)}(\Delta_i^n X)^{(j_2)}\bigr|^4{\bf{1}}_{\{|\Delta_i^n X|\leq Dh_n^{\rho}\}}\big|\mathcal{F}_{i-1}^n\biggr]\\
    &\qquad\qquad
    +\frac{C}{n}\sum_{j_1=1}^{p}\sum_{j_2=1}^p\bigl|({\bf{\Sigma}}_0)_{j_1j_2}\bigr|^{4}\stackrel{p}{\longrightarrow}0,
\end{align*}
so that we obtain (\ref{L4prob}), which completes the proof.
\end{proof}
In Proposition \ref{supHprop} and the proof of Theorem \ref{thetatheorem}, we simply write ${\bf{P}}_{\theta_0}$ as ${\bf{P}}$. Define
\begin{align*}
    {\bf{H}}(\theta)
    &=-\frac{1}{2}\log\det{\bf{\Sigma}}(\theta)-\frac{1}{2}\tr\bigl\{{\bf{\Sigma}}(\theta)^{-1}{\bf{\Sigma}}(\theta_0)\bigr\}.
\end{align*}
Set $\partial_{\theta}=\partial/\partial\theta$ and $\partial^2_{\theta}=\partial_{\theta}\partial_{\theta}^{\top}$.
\begin{proposition}\label{supHprop}
Under {\rm{\textbf{[A1]}}}-{\rm{\textbf{[A4]}}}, as $n\longrightarrow\infty$,
\begin{align}
    \sup_{\theta\in\Theta}\Bigl|\frac{1}{n}{\bf{H}}_n(\theta)-{\bf{H}}(\theta)\Bigr|&\stackrel{p}{\longrightarrow}0, \label{Fsup}\\
    \sup_{\theta\in\Theta}\Bigl|\frac{1}{n}\partial^2_{\theta}{\bf{H}}_n(\theta)-\partial^2_{\theta}{\bf{H}}(\theta)\Bigr|&\stackrel{p}{\longrightarrow}0\label{F2sup}
\end{align}
under ${\bf{P}}$.
\end{proposition}
\begin{proof}
Since
\begin{align*}
    &\quad\ \frac{1}{n}{\bf{H}}_n(\theta)-{\bf{H}}(\theta)\\
    &=-\frac{1}{2}\log\det {\bf{\Sigma}}(\theta)\Bigl(\frac{N_n}{n}-1\Bigr)-\frac{1}{2}
    \tr\Bigl\{{\bf{\Sigma}}(\theta)^{-1}\Bigl(\frac{\tilde{N}_n}{n}\hat{\bf{\Sigma}}_n-{\bf{\Sigma}}(\theta_0)\Bigr)\Bigr\},
\end{align*}
Theorem \ref{Qtheorem} and Proposition \ref{Nprop} imply
\begin{align*}
    &\quad\ \sup_{\theta\in\Theta}\Bigl|\frac{1}{n}{\bf{H}}_n(\theta)-{\bf{H}}(\theta)\Bigr|\\
    &\leq \sup_{\theta\in\Theta}\bigl|\log\det {\bf{\Sigma}}(\theta)\bigr|\Bigl|\frac{N_n}{n}-1\Bigr|+\sup_{\theta\in\Theta}\bigl|{\bf{\Sigma}}(\theta)^{-1}\bigr|
    \Bigl|\frac{\tilde{N}_n}{n}\hat{\bf{\Sigma}}_n-{\bf{\Sigma}}(\theta_0)\Bigr|
    \stackrel{p}{\longrightarrow} 0,
\end{align*}
which yields (\ref{Fsup}). Next, we show (\ref{F2sup}). Note that
\begin{align*}
    \frac{1}{n}\partial_{\theta^{(i)}}\partial_{\theta^{(j)}}{\bf{H}}_n(\theta)
    &=\frac{N_n}{2n}\tr\bigl\{({\bf{\Sigma}}(\theta)^{-1})(\partial_{\theta^{(i)}}{\bf{\Sigma}}(\theta))({\bf{\Sigma}}(\theta)^{-1})(\partial_{\theta^{(j)}}{\bf{\Sigma}}(\theta))\bigr\}\\
    &\quad-\frac{N_n}{2n}\tr\bigl\{({\bf{\Sigma}}(\theta)^{-1})(\partial_{\theta^{(i)}}\partial_{\theta^{(j)}}{\bf{\Sigma}}(\theta))\bigr\}\\
    &\quad -\frac{\tilde{N}_n}{2n}\tr\bigl\{({\bf{\Sigma}}(\theta)^{-1})(\partial_{\theta^{(i)}}{\bf{\Sigma}}(\theta))({\bf{\Sigma}}(\theta)^{-1})(\partial_{\theta^{(j)}}{\bf{\Sigma}}(\theta))({\bf{\Sigma}}(\theta)^{-1})\hat{{\bf{\Sigma}}}_{n}\bigr\}\\
    &\quad+\frac{\tilde{N}_n}{2n}\tr\bigl\{({\bf{\Sigma}}(\theta)^{-1})(\partial_{\theta^{(i)}}\partial_{\theta^{(j)}}{\bf{\Sigma}}(\theta))({\bf{\Sigma}}(\theta)^{-1})\hat{{\bf{\Sigma}}}_{n}\bigr\}\\
    &\quad-\frac{\tilde{N}_n}{2n}\tr\bigl\{({\bf{\Sigma}}(\theta)^{-1})(\partial_{\theta^{(j)}}{\bf{\Sigma}}(\theta))({\bf{\Sigma}}(\theta)^{-1})(\partial_{\theta^{(i)}}{\bf{\Sigma}}(\theta))({\bf{\Sigma}}(\theta)^{-1})\hat{{\bf{\Sigma}}}_{n}\bigr\}
\end{align*}
and
\begin{align*}
    \partial_{\theta^{(i)}}\partial_{\theta^{(j)}}{\bf{H}}(\theta)
    &=\frac{1}{2}\tr\bigl\{({\bf{\Sigma}}(\theta)^{-1})(\partial_{\theta^{(i)}}{\bf{\Sigma}}(\theta))({\bf{\Sigma}}(\theta)^{-1})(\partial_{\theta^{(j)}}{\bf{\Sigma}}(\theta))\bigr\}\\
    &\quad-\frac{1}{2}\tr\bigl\{({\bf{\Sigma}}(\theta)^{-1})(\partial_{\theta^{(i)}}\partial_{\theta^{(j)}}{\bf{\Sigma}}(\theta))\bigr\}\\
    &\quad-\frac{1}{2}\tr\bigl\{({\bf{\Sigma}}(\theta)^{-1})(\partial_{\theta^{(i)}}{\bf{\Sigma}}(\theta))({\bf{\Sigma}}(\theta)^{-1})(\partial_{\theta^{(j)}}{\bf{\Sigma}}(\theta))({\bf{\Sigma}}(\theta)^{-1})
    {\bf{\Sigma}}(\theta_0)\bigr\}\\
    &\quad+\frac{1}{2}\tr\bigl\{({\bf{\Sigma}}(\theta)^{-1})(\partial_{\theta^{(i)}}\partial_{\theta^{(j)}}{\bf{\Sigma}}(\theta))({\bf{\Sigma}}(\theta)^{-1}){\bf{\Sigma}}(\theta_0)\bigr\}\\
    &\quad-\frac{1}{2}\tr\bigl\{({\bf{\Sigma}}(\theta)^{-1})(\partial_{\theta^{(j)}}{\bf{\Sigma}}(\theta))({\bf{\Sigma}}(\theta)^{-1})(\partial_{\theta^{(i)}}{\bf{\Sigma}}(\theta))({\bf{\Sigma}}(\theta)^{-1})({\bf{\Sigma}}(\theta_0)\bigr\}
\end{align*}
for $i,j=1,\ldots,q$. Theorem \ref{Qtheorem} and Proposition \ref{Nprop} yield 
\begin{align*}
    &\quad\ \sup_{\theta\in\Theta}\Bigl|\frac{1}{n}\partial^2_{\theta}{\bf{H}}_n(\theta)-\partial^2_{\theta}{\bf{H}}(\theta)\Bigr|\\
    &\leq \sum_{i=1}^q\sum_{j=1}^q \sup_{\theta\in\Theta}\Bigl|\frac{1}{n}\partial_{\theta^{(i)}}\partial_{\theta^{(j)}}{\bf{H}}_n(\theta)-\partial_{\theta^{(i)}}\partial_{\theta^{(j)}}{\bf{H}}(\theta)\Bigr|\\
    &\leq \Bigl|\frac{N_n}{n}-1\Bigr|\sum_{i=1}^q\sum_{j=1}^q 
    \sup_{\theta\in\Theta}\bigl|({\bf{\Sigma}}(\theta)^{-1})(\partial_{\theta^{(i)}}{\bf{\Sigma}}(\theta))({\bf{\Sigma}}(\theta)^{-1})(\partial_{\theta^{(j)}}{\bf{\Sigma}}(\theta))\bigr|\\
    &\quad+\Bigl|\frac{N_n}{n}-1\Bigr|\sum_{i=1}^q\sum_{j=1}^q \sup_{\theta\in\Theta}\bigl|({\bf{\Sigma}}(\theta)^{-1})(\partial_{\theta^{(i)}}\partial_{\theta^{(j)}}{\bf{\Sigma}}(\theta))\bigr|\\
    &\quad+\sum_{i=1}^q\sum_{j=1}^q \sup_{\theta\in\Theta}\bigl|({\bf{\Sigma}}(\theta)^{-1})(\partial_{\theta^{(i)}}{\bf{\Sigma}}(\theta))({\bf{\Sigma}}(\theta)^{-1})(\partial_{\theta^{(j)}}{\bf{\Sigma}}(\theta))({\bf{\Sigma}}(\theta)^{-1})\bigr|
    \Bigl|\frac{\tilde{N}_n}{n}\hat{{\bf{\Sigma}}}_{n}-{\bf{\Sigma}}(\theta_0)\Bigr|\\
    &\quad+\sum_{i=1}^q\sum_{j=1}^q \sup_{\theta\in\Theta}\bigl|({\bf{\Sigma}}(\theta)^{-1})(\partial_{\theta^{(i)}}\partial_{\theta^{(j)}}{\bf{\Sigma}}(\theta))({\bf{\Sigma}}(\theta)^{-1})\bigr|\Bigl|\frac{\tilde{N}_n}{n}\hat{{\bf{\Sigma}}}_{n}-
    {\bf{\Sigma}}(\theta_0)\Bigr|\\
    &\quad+\sum_{i=1}^q\sum_{j=1}^q \sup_{\theta\in\Theta}\bigl|({\bf{\Sigma}}(\theta)^{-1})(\partial_{\theta^{(j)}}{\bf{\Sigma}}(\theta))({\bf{\Sigma}}(\theta)^{-1})(\partial_{\theta^{(i)}}{\bf{\Sigma}}(\theta))({\bf{\Sigma}}(\theta)^{-1})\bigr|
    \Bigl|\frac{\tilde{N}_n}{n}\hat{{\bf{\Sigma}}}_{n}-{\bf{\Sigma}}(\theta_0)\Bigr|\stackrel{p}{\longrightarrow} 0,
\end{align*}
so that we obtain (\ref{F2sup}).
\end{proof}
\begin{proof}[\textbf{Proof of Theorem \ref{thetatheorem}}]
First, we prove (\ref{thetacons}). {\bf{[B1]}} (a) implies that ${\bf{H}}(\theta)$ has a unique maximum value at $\theta=\theta_0$, so that for any $\varepsilon>0$, there exists $\delta>0$ such that
\begin{align*}
    |\theta-\theta_0|>\varepsilon \Longrightarrow {\bf{H}}(\theta_0)-{\bf{H}}(\theta)>\delta.
\end{align*}
Consequently, since ${\bf{H}}_n(\hat{\theta}_n)\geq {\bf{H}}_n(\theta_0)$, Proposition \ref{supHprop} shows
\begin{align*}
    {\bf{P}}\Bigl(|\hat{\theta}_n-\theta_0|>\varepsilon\Bigr)&\leq {\bf{P}}\Bigl({\bf{H}}(\theta_0)-{\bf{H}}(\hat{\theta}_n)>\delta\Bigr)\\
    &\leq {\bf{P}}\biggl({\bf{H}}(\theta_0)-\frac{1}{n}{\bf{H}}_n(\theta_0)
    >\frac{\delta}{3}\biggr)\\
    &\quad+{\bf{P}}\biggl(\frac{1}{n}{\bf{H}}_n(\theta_0)-\frac{1}{n}{\bf{H}}_n(\hat{\theta}_n)
    >\frac{\delta}{3}\biggr)+{\bf{P}}\biggl(\frac{1}{n}{\bf{H}}_n(\hat{\theta}_n)-{\bf{H}}(\hat{\theta}_n)>\frac{\delta}{3}\biggr)\\
    &\leq 2{\bf{P}}\biggl(\sup_{\theta\in\Theta}\Bigr|\frac{1}{n}
    {\bf{H}}_n(\theta)-{\bf{H}}(\theta)\Bigl|>\frac{\delta}{3}\biggr)\longrightarrow 0,
\end{align*}
which yields (\ref{thetacons}). Next, we show (\ref{thetaasym}). Let $\check{\theta}_{n,\lambda}=\theta_0+\lambda(\hat{\theta}_n-\theta_0)$  
and $A_{1,n}=\bigl\{\hat{\theta}_n\in{\rm{Int}}(\Theta)\bigr\}$.
Using Taylor's theorem, on $A_{1,n}$, we have
\begin{align*}
    0=\frac{1}{\sqrt{n}}\partial_{\theta}{\bf{H}}_n(\hat{\theta}_n)
    &=\frac{1}{\sqrt{n}}\partial_{\theta}{\bf{H}}_n(\theta_0)+
    \biggl(\frac{1}{n}\int_{0}^{1}\partial^2_{\theta}{\bf{H}}_n(\check{\theta}_{n,\lambda})d\lambda\biggr)\sqrt{n}\bigl(\hat{\theta}_n-\theta_0\bigr),
\end{align*}
so that
\begin{align}
   -\frac{1}{\sqrt{n}}\partial_{\theta}{\bf{H}}_n(\theta_0)
   =\biggl(\frac{1}{n}\int_{0}^{1}\partial^2_{\theta}{\bf{H}}_n(\check{\theta}_{n,\lambda})d\lambda\biggr)\sqrt{n}\bigl(
   \hat{\theta}_n-\theta_0\bigr).
   \label{Heq}
\end{align}
Since 
\begin{align*}
    \partial_{\theta^{(i)}}{\bf{H}}_n(\theta)=-\frac{N_n}{2}\tr{\bigl\{
    ({\bf{\Sigma}}(\theta)^{-1})
    (\partial_{\theta^{(i)}}{\bf{\Sigma}}(\theta))\bigr\}}+\frac{\tilde{N}_n}{2}\tr\bigl\{({\bf{\Sigma}}(\theta)^{-1})(\partial_{\theta^{(i)}}{\bf{\Sigma}}(\theta))({\bf{\Sigma}}(\theta)^{-1})\hat{\bf{\Sigma}}_n\bigr\},
\end{align*}
one gets
\begin{align*}
    \frac{1}{\sqrt{n}}\partial_{\theta^{(i)}}{\bf{H}}_n(\theta_0)
    &=\frac{\tilde{N}_n}{2n}\tr\bigl\{({\bf{\Sigma}}(\theta_0)^{-1})(\partial_{\theta^{(i)}}{\bf{\Sigma}}(\theta_0))({\bf{\Sigma}}(\theta_0)^{-1})\sqrt{n}(\hat{{\bf{\Sigma}}}_{n}-{\bf{\Sigma}}(\theta_0))\bigr\}\\
    &\qquad+\frac{1}{2}\sqrt{n}\Bigl(\frac{\tilde{N}_n}{n}-\frac{N_n}{n}\Bigr)\tr{\bigl\{({\bf{\Sigma}}(\theta_0)^{-1})
    (\partial_{\theta^{(i)}}{\bf{\Sigma}}(\theta_0))\bigr\}}
\end{align*}
for $i=1,\ldots,q$. Moreover, we see
\begin{align*}
    &\quad\ \tr\bigl\{({\bf{\Sigma}}(\theta_0)^{-1})(\partial_{\theta^{(i)}}{\bf{\Sigma}}(\theta_0))({\bf{\Sigma}}(\theta_0)^{-1})\sqrt{n}(\hat{{\bf{\Sigma}}}_{n}-{\bf{\Sigma}}(\theta_0))\bigr\}\\
    &=\bigl\{\vec{\partial_{\theta^{(i)}}{\bf{\Sigma}}(\theta_0)}\bigr\}^{\top}\bigl({\bf{\Sigma}}(\theta_0)^{-1}\otimes{\bf{\Sigma}}(\theta_0)^{-1}\bigr)\sqrt{n}(\vec \hat{{\bf{\Sigma}}}_{n}-\vec {\bf{\Sigma}}(\theta_0))\\
    &=\bigl\{\vech{\partial_{\theta^{(i)}}{\bf{\Sigma}}(\theta_0)}\bigr\}^{\top}\mathbb{D}_p^{\top}\bigl({\bf{\Sigma}}(\theta_0)\otimes{\bf{\Sigma}}(\theta_0)\bigr)^{-1}\mathbb{D}_p\sqrt{n}(\vech \hat{{\bf{\Sigma}}}_{n}-\vech {\bf{\Sigma}}(\theta_0))\\
    &=2\bigl\{\partial_{\theta^{(i)}}\vech{{\bf{\Sigma}}(\theta_0)}\bigr\}^{\top}{\bf{W}}_0^{-1}\sqrt{n} (\vech {\hat{\bf\Sigma}}_{n}-\vech {\bf{\Sigma}}(\theta_0))\\
    &=\bigl(2\Delta_0^{\top}{\bf{W}}_0^{-1}\sqrt{n}(\vech \hat{{\bf{\Sigma}}}_{n}-\vech {\bf{\Sigma}}(\theta_0))\bigr)^{(i)}
\end{align*}
and Proposition \ref{Nprop} shows
\begin{align}
     \sqrt{n}\Bigl(\frac{\tilde{N}_n}{n}-\frac{N_n}{n}\Bigr)&=\sqrt{n}\Bigl(\frac{\tilde{N}_n}{n}-1\Bigr)-\sqrt{n}\Bigl(\frac{N_n}{n}-1\Bigr)=o_p(1), \label{Rn}
\end{align}
so that Proposition \ref{Nprop} and Theorem \ref{Qtheorem} imply
\begin{align}
\begin{split}
    -\frac{1}{\sqrt{n}}\partial_{\theta}{\bf{H}}_n(\theta_0)
    &=-\frac{\tilde{N}_n}{n}\Delta_0^{\top}{\bf{W}}_0^{-1}\sqrt{n}(\vech \hat{{\bf{\Sigma}}}_{n}-\vech {\bf{\Sigma}}(\theta_0))+o_p(1)\\
    &\qquad\qquad\qquad\qquad\qquad\qquad\qquad\stackrel{d}{\longrightarrow}N_{q}\bigl(0,\Delta_0^{\top}{\bf{W}}_0^{-1}\Delta_0\bigr). \label{sqrtHprob}
\end{split}
\end{align}
Similarly, it is shown that
\begin{align*}
    \partial_{\theta^{(i)}}\partial_{\theta^{(j)}}{\bf{H}}(\theta_0)&=
    -\frac{1}{2}
    \tr\bigl\{({\bf{\Sigma}}(\theta_0)^{-1})(\partial_{\theta^{(j)}}{\bf{\Sigma}}(\theta_0))({\bf{\Sigma}}(\theta_0)^{-1})(\partial_{\theta^{(i)}}{\bf{\Sigma}}(\theta_0))\bigr\}\\
    &=-\bigl\{\partial_{\theta^{(i)}}\vech{{\bf{\Sigma}}(\theta_0)}\bigr\}^{\top}{\bf{W}}_0^{-1}\bigl\{\partial_{\theta^{(j)}}\vech{{\bf{\Sigma}}(\theta_0)}\bigr\}
    =\bigl(-\Delta_0^{\top}{\bf{W}}_0^{-1}\Delta_0\bigr)_{ij}
\end{align*}
for $i,j=1,\ldots,q$, which yields
\begin{align*}
    \partial_{\theta}^2 {\bf{H}}(\theta_0)=-\Delta_0^{\top}{\bf{W}}_0^{-1}\Delta_0.
\end{align*}
As $\partial_{\theta}^2 {\bf{H}}(\theta)$ is continuous in $\theta$, one has
\begin{align*}
    \sup_{\theta\in B_n}\bigl|
     \partial^2_{\theta}{\bf{H}}(\theta)-\partial^2_{\theta}{\bf{H}}(\theta_0)\bigr|\longrightarrow 0
\end{align*}
as $n\longrightarrow\infty$, where $\{\rho_n\}_{n\in\mathbb{N}}$ is a sequence such that $\rho_n\longrightarrow 0$ as $n\longrightarrow\infty$, and
\begin{align*}
    B_n=\bigl\{\theta\in\Theta;\ |\theta-\theta_0|\leq \rho_n\bigr\}.
\end{align*}
Hence, Proposition \ref{supHprop} and (\ref{thetacons}) show that for any $\varepsilon>0$,
\begin{align*}
    &\quad\ {\bf{P}}\Biggl(\biggl|\frac{1}{n}\int_{0}^{1}\partial^2_{\theta}{\bf{H}}_n(\check{\theta}_{n,\lambda})d\lambda+\Delta_0^{\top}{\bf{W}}_0^{-1}\Delta_0\biggr|>\varepsilon\Biggr)\\
    &={\bf{P}}\Biggl(\biggl\{\biggl|\int_{0}^{1}\Bigl\{\frac{1}{n}
    \partial^2_{\theta}{\bf{H}}_n(\check{\theta}_{n,\lambda})d\lambda-\partial^2_{\theta}{\bf{H}}(\theta_0)
    \Bigr\}d\lambda\biggr|>\varepsilon\biggr\}\cap \Bigl\{|\hat{\theta}_n-\theta_0|\leq \rho_n\Bigr\}\Biggr)\\
    &\qquad+{\bf{P}}\Biggl(\biggl\{\biggl|\int_{0}^{1}\Bigl\{\frac{1}{n}
    \partial^2_{\theta}{\bf{H}}_n(\check{\theta}_{n,\lambda})d\lambda-\partial^2_{\theta}{\bf{H}}(\theta_0)
    \Bigr\}d\lambda\biggr|>\varepsilon\biggr\}\cap \Bigl\{|\hat{\theta}_n-\theta_0|>\rho_n\Bigr\}\Biggr)\\
    &\leq {\bf{P}}\biggl(\sup_{\theta\in B_n}\Bigl|
    \frac{1}{n}
    \partial^2_{\theta}{\bf{H}}_n(\theta)-\partial^2_{\theta}{\bf{H}}(\theta_0)
     \Bigr|>\varepsilon\biggr)+{\bf{P}}\Bigl(|\hat{\theta}_n-\theta_0|> \rho_n\Bigr)\\
    &\leq  {\bf{P}}\biggl(\sup_{\theta\in\Theta}\Bigl|
    \frac{1}{n}
    \partial^2_{\theta}{\bf{H}}_n(\theta)-\partial^2_{\theta}{\bf{H}}(\theta)
     \Bigr|>\frac{\varepsilon}{2}\biggr)\\
     &\qquad+{\bf{P}}\biggl(\sup_{\theta\in B_n}
     \bigl|
    \partial^2_{\theta}{\bf{H}}(\theta)-\partial^2_{\theta}{\bf{H}}(\theta_0)
     \bigr|>\frac{\varepsilon}{2}\biggr)+{\bf{P}}\Bigl(|\hat{\theta}_n-\theta_0|> \rho_n\Bigr)
    \longrightarrow 0,
\end{align*}
so that
\begin{align}
    \frac{1}{n}\int_{0}^{1}\partial^2_{\theta}{\bf{H}}_n(\check{\theta}_{n,\lambda})d\lambda\stackrel{p}{\longrightarrow}-\Delta_0^{\top}{\bf{W}}_0^{-1}\Delta_0.
    \label{H2prob}
\end{align}
Set 
\begin{align*}
    B_{1,n}=\frac{1}{\sqrt{n}}\partial_{\theta}{\bf{H}}_n(\theta_0),\quad 
    B_{2,n}=-\frac{1}{n}\int_{0}^{1}\partial^2_{\theta}{\bf{H}}_n(\check{\theta}_{n, \lambda})d\lambda,\quad 
    \tilde{B}_{2,n}=\left\{
    \begin{array}{ll}
    B_{2,n} & \bigl({\rm{on}}\ A_{2,n}\bigr),\\
    \mathbb{I}_{q} & \bigl({\rm{on}}\ A_{2,n}^c\bigr),
    \end{array}
    \right.
\end{align*}
where $A_{2,n}=\bigl\{\det B_{2,n}>0\bigr\}$. Note that (\ref{thetacons}) and (\ref{H2prob}) imply
\begin{align*}
    {\bf{P}}\bigl(A_{1,n}\bigr)\longrightarrow 1,\quad {\bf{P}}\bigl(A_{2,n}\bigr)\longrightarrow 1
\end{align*}
since $\theta_0\in {\rm{Int}}(\Theta)$ and $\Delta_0^{\top}{\bf{W}}_0^{-1}\Delta_0$ is positive. In a similar way to Proposition \ref{Nprop}, (\ref{H2prob}) yields 
\begin{align*}
    \tilde{B}_{2,n}^{-1} \stackrel{p}{\longrightarrow}(\Delta_0^{\top}{\bf{W}}_0^{-1}\Delta_0)^{-1}.
\end{align*}
By Slutsky’s theorem, it holds from (\ref{sqrtHprob}) that
\begin{align*}
    \tilde{B}_{2,n}^{-1}B_{1,n}\stackrel{d}{\longrightarrow}N_q\Bigl(0,(\Delta_0^{\top}{\bf{W}}_0^{-1}\Delta_0)^{-1}\Bigr).
\end{align*}
Thus, for any closed set $C\in\mathbb{R}^q$, it follows from (\ref{Heq}) that
\begin{align*}
    &\quad\ \limsup_{n\longrightarrow\infty}{\bf{P}}\Bigl(\sqrt{n}(\hat{\theta}_n-\theta_0)\in C\Bigr)\\
    &= \limsup_{n\longrightarrow\infty}{\bf{P}}\Bigl(\bigl\{\sqrt{n}(\hat{\theta}_n-\theta_0)\in C\bigr\}\cap (A_{1,n}\cap A_{2,n})\Bigr)\\
    &\qquad+\limsup_{n\longrightarrow\infty}{\bf{P}}\Bigl(\bigl\{\sqrt{n}(\hat{\theta}_n-\theta_0)\in C\bigr\}\cap (A_{1,n}\cap A_{2,n})^c\Bigr)\\
    &\leq \limsup_{n\longrightarrow\infty}{\bf{P}}\Bigl(\bigl\{B_{2,n}^{-1}B_{1,n}\in C\bigr\}\cap (A_{1,n}\cap A_{2,n})\Bigr)+\limsup_{n\longrightarrow\infty}{\bf{P}}\bigl(A_{1,n}^c\cup A_{2,n}^c\bigr)\\
    &\leq \limsup_{n\longrightarrow\infty} {\bf{P}}\Bigl(\tilde{B}_{2,n}^{-1}B_{1,n}\in C\bigr)+\limsup_{n\longrightarrow\infty}{\bf{P}}\bigl(A_{1,n}^c\bigr)+\limsup_{n\longrightarrow\infty}{\bf{P}}\bigl(A_{2,n}^c\bigr)\\
    &\leq {\bf{P}}\Bigl((\Delta_0^{\top}{\bf{W}}_0^{-1}\Delta_0)^{-\frac{1}{2}}Z_q\in C\bigr),
\end{align*}
which yields (\ref{thetaasym}).
\end{proof}
\begin{lemma}\label{thetalemma}
Suppose that {\bf{[A1]}}-{\bf{[A4]}} and {\bf{[B1]}} hold. Then, under $H_0$, 
\begin{align*}
    \sqrt{n}(\hat{\theta}_n-\theta_0)
    &=(\Delta_0^{\top}{\bf{W}}_0^{-1}\Delta_0)^{-1}\Delta_0^{\top}{\bf{W}}_0^{-1}\sqrt{n}(\vech \hat{{\bf{\Sigma}}}_n-\vech {\bf{\Sigma}}(\theta_0))+o_p(1) 
\end{align*}
as $n\longrightarrow\infty$.
\end{lemma}
\begin{proof}
Recall that $A_{1,n}=\bigl\{\hat{\theta}_n\in{{\rm{Int}}(\Theta)}\bigr\}$, $A_{2,n}=\bigl\{\det B_{2,n}>0\bigr\}$ and
\begin{align*}
    B_{1,n}=\frac{1}{\sqrt{n}}\partial_{\theta}{\bf{H}}_n(\theta_0),\quad B_{2,n}=-\frac{1}{n}\int_{0}^{1}\partial^2_{\theta}{\bf{H}}_n(\check{\theta}_{n, \lambda})d\lambda,\quad 
    \tilde{B}_{2,n}=\left\{
    \begin{array}{ll}
    B_{2,n} & \bigl({\rm{on}}\ A_{2,n}\bigr),\\
    \mathbb{I}_{q} & \bigl({\rm{on}}\ A_{2,n}^c\bigr),
    \end{array}
    \right.
\end{align*}
where $\check{\theta}_{n,\lambda}=\theta_0+\lambda(\hat{\theta}_n-\theta_0)$.
In a similar manner to the proof of Theorem \ref{thetatheorem}, under $H_0$, we have
\begin{align}
    \tilde{B}_{2,n}^{-1}\stackrel{p}{\longrightarrow}(\Delta_0^{\top}{\bf{W}}_0^{-1}\Delta_0)^{-1} \label{Bprob}
\end{align}
and 
\begin{align*}
\begin{split}
    B_{1,n}=\frac{1}{\sqrt{n}}\partial_{\theta}{\bf{H}}_n(\theta_0)
    &=\frac{\tilde{N}_n}{n}\Delta_0^{\top}{\bf{W}}_0^{-1}\sqrt{n}(\vech \hat{{\bf{\Sigma}}}_{n}-\vech {\bf{\Sigma}}(\theta_0))+R_n,
\end{split}
\end{align*}
where 
\begin{align*}
    R_n^{(i)}=\frac{1}{2}\sqrt{n}\Bigl(\frac{\tilde{N}_n}{n}-\frac{N_n}{n}\Bigr)\tr{\bigl\{({\bf{\Sigma}}(\theta_0)^{-1})
    (\partial_{\theta^{(i)}}{\bf{\Sigma}}(\theta_0))\bigr\}}
\end{align*} 
for $i=1,\ldots,q$. 
It follows from (\ref{Rn}) that
\begin{align}
    R_n=o_p(1) \label{Rprob}
\end{align}
under $H_0$. Furthermore, it is shown that
\begin{align*}
    \sqrt{n}(\hat{\theta}_n-\theta_0)&=B_{2,n}^{-1}B_{1,n}\\
    &=\frac{\tilde{N}_n}{n}B_{2,n}^{-1}\Delta_0^{\top}{\bf{W}}_0^{-1}\sqrt{n}(\vech \hat{{\bf{\Sigma}}}_{n}-\vech {\bf{\Sigma}}(\theta_0))+B_{2,n}^{-1}R_n
\end{align*}
on $A_{1,n}\cap A_{2,n}$, so that 
\begin{align*}
    &\quad\ \sqrt{n}(\hat{\theta}_n-\theta_0)-(\Delta_0^{\top}{\bf{W}}_0^{-1}\Delta_0)^{-1}\Delta_0^{\top}{\bf{W}}_0^{-1}\sqrt{n}(\vech \hat{{\bf{\Sigma}}}_n-\vech {\bf{\Sigma}}(\theta_0))\\
    &=\Bigl(\frac{\tilde{N}_n}{n}B_{2,n}^{-1}-(\Delta_0^{\top}{\bf{W}}_0^{-1}\Delta_0)^{-1}\Bigr)\Delta_0^{\top}{\bf{W}}_0^{-1}\sqrt{n}(\vech \hat{{\bf{\Sigma}}}_n-\vech {\bf{\Sigma}}(\theta_0))+B_{2,n}^{-1}R_n\\
    &=\Bigl(\frac{\tilde{N}_n}{n}\tilde{B}_{2,n}^{-1}-(\Delta_0^{\top}{\bf{W}}_0^{-1}\Delta_0)^{-1}\Bigr)\Delta_0^{\top}{\bf{W}}_0^{-1}\sqrt{n}(\vech \hat{{\bf{\Sigma}}}_n-\vech {\bf{\Sigma}}(\theta_0))+\tilde{B}_{2,n}^{-1}R_n
\end{align*}
on $A_{1,n}\cap A_{2,n}$. Under $H_0$, it holds from Proposition \ref{Nprop}, Theorem \ref{Qtheorem}, (\ref{Bprob}) and (\ref{Rprob}) that
\begin{align*}
    \Bigl(\frac{\tilde{N}_n}{n}\tilde{B}_{2,n}^{-1}-(\Delta_0^{\top}{\bf{W}}_0^{-1}\Delta_0)^{-1}\Bigr)\Delta_0^{\top}{\bf{W}}_0^{-1}\sqrt{n}(\vech \hat{{\bf{\Sigma}}}_n-\vech {\bf{\Sigma}}(\theta_0))+\tilde{B}_{2,n}^{-1}R_n=o_p(1).
\end{align*}
Therefore, since
\begin{align*}
    {\bf{P}}\bigl(A_{1,n}\cap A_{2,n}\bigr)\longrightarrow 1
\end{align*}
under $H_0$, for any $\varepsilon>0$, we have
\begin{align*}
    &\quad\ {\bf{P}}\biggl(\bigl|\sqrt{n}(\hat{\theta}_n-\theta_0)-(\Delta_0^{\top}{\bf{W}}_0^{-1}\Delta_0)^{-1}\Delta_0^{\top}{\bf{W}}_0^{-1}\sqrt{n}(\vech \hat{{\bf{\Sigma}}}_n-\vech {\bf{\Sigma}}(\theta_0))\bigr|>\varepsilon\biggr)\\
    &\leq {\bf{P}}\biggl(\Bigl\{\Bigl|\Bigl(\frac{\tilde{N}_n}{n}\tilde{B}_{2,n}^{-1}-(\Delta_0^{\top}{\bf{W}}_0^{-1}\Delta_0)^{-1}\Bigr)\Delta_0^{\top}{\bf{W}}_0^{-1}\sqrt{n}(\vech \hat{{\bf{\Sigma}}}_n-\vech {\bf{\Sigma}}(\theta_0))+\tilde{B}_{2,n}^{-1}R_n\Bigr|>\varepsilon\biggr\}\\
    &\qquad\qquad\qquad\qquad\qquad\qquad\qquad\qquad\qquad
    \qquad\qquad\qquad\quad\cap (A_{1,n}\cap A_{2,n}) \biggr)+{\bf{P}}\Bigl((A_{1,n}\cap A_{2,n})^c\Bigr)\\
    &\leq  {\bf{P}}\biggl(\Bigl|\Bigl(\frac{\tilde{N}_n}{n}\tilde{B}_{2,n}^{-1}-(\Delta_0^{\top}{\bf{W}}_0^{-1}\Delta_0)^{-1}\Bigr)\Delta_0^{\top}{\bf{W}}_0^{-1}\sqrt{n}(\vech \hat{{\bf{\Sigma}}}_n-\vech {\bf{\Sigma}}(\theta_0))+\tilde{B}_{2,n}^{-1}R_n\Bigr|>\varepsilon\biggr)\\
    &\qquad\qquad\qquad\qquad\qquad\qquad\qquad\qquad\qquad
    \qquad\qquad\qquad\qquad\qquad\qquad+{\bf{P}}\Bigl((A_{1,n}\cap A_{2,n})^c\Bigr)\longrightarrow 0
\end{align*}
under $H_0$, which completes the proof.
\end{proof}
\begin{lemma}\label{testlemma}
Suppose that {\bf{[A1]}}-{\bf{[A4]}} and {\bf{[B1]}} hold. Then, under $H_0$,
\begin{align*}
    &\quad\ \sqrt{n}(\vech {\bf{\Sigma}}(\hat{\theta}_n)-\vech{\bf{\Sigma}}(\theta_0))\\
    &=\Delta_0(\Delta_0^{\top}{\bf{W}}_0^{-1}\Delta_0)^{-1}\Delta_0^{\top}{\bf{W}}_0^{-1}\sqrt{n}(\vech \hat{{\bf{\Sigma}}}_n-\vech {\bf{\Sigma}}(\theta_0))+o_p(1)
\end{align*}
as $n\longrightarrow\infty$.
\end{lemma}
\begin{proof}
Let $\sigma(\theta)=\vech{\bf{\Sigma}}(\theta)$. Using Taylor's Theorem, we get
\begin{align*}
    \sigma(\hat{\theta}_n)=\sigma(\theta_0)+\biggl(\int_0^{1}\frac{\partial}{\partial\theta^{\top}}\sigma(\theta_0+\lambda(\hat{\theta}_n-\theta_0))d\lambda\biggr)(\hat{\theta}_n-\theta_0),
\end{align*}
which yields 
\begin{align*}
    \sqrt{n}(\vech {\bf{\Sigma}}(\hat{\theta}_n)-\vech {\bf{\Sigma}}(\theta_0))
    &=\biggl(\int_0^{1}\frac{\partial}{\partial\theta^{\top}}\sigma(\theta_0+\lambda(\hat{\theta}_n-\theta_0))d\lambda\biggr)\sqrt{n}(\hat{\theta}_n-\theta_0).
\end{align*}
In a similar manner to the proof of Theorem \ref{thetatheorem}, it is shown that 
\begin{align*}
    \int_0^{1}\frac{\partial}{\partial\theta^{\top}}\sigma(\theta_0+\lambda(\hat{\theta}_n-\theta_0))d\lambda\stackrel{p}{\longrightarrow} \frac{\partial}{\partial\theta^{\top}}\sigma(\theta_0)=\Delta_0
\end{align*}
under $H_0$. Moreover, Lemma \ref{thetalemma} yields
\begin{align*}
    \sqrt{n}(\hat{\theta}_n-\theta_0)
    &=(\Delta_0^{\top}{\bf{W}}_0^{-1}\Delta_0)^{-1}\Delta_0^{\top}{\bf{W}}_0^{-1}\sqrt{n}(\vech \hat{{\bf{\Sigma}}}_n-\vech {\bf{\Sigma}}(\theta_0))+o_p(1) 
\end{align*}
under $H_0$. Consequently, since 
\begin{align*}
    \sqrt{n}(\vech \hat{{\bf{\Sigma}}}_n-\vech {\bf{\Sigma}}(\theta_0))=O_p(1)
\end{align*}
under $H_0$, we obtain
\begin{align*}
    &\quad\ \sqrt{n}(\vech {\bf{\Sigma}}(\hat{\theta}_n)-\vech {\bf{\Sigma}}(\theta_0))\\
    &=\Delta_0(\Delta_0^{\top}{\bf{W}}_0^{-1}\Delta_0)^{-1}\Delta_0^{\top}{\bf{W}}_0^{-1}\sqrt{n}(\vech \hat{{\bf{\Sigma}}}_n-\vech {\bf{\Sigma}}(\theta_0))+o_p(1)
\end{align*}
under $H_0$, which completes the proof.
\end{proof}
\begin{proof}[\bf{Proof of Theorem \ref{testtheorem1}}]
First, we see
\begin{align*}
     {\bf{T}}_n=-2{\bf{H}}_{n}({\bf{\Sigma}}(\hat{\theta}_n))+2{\bf{H}}_{n}(\hat{\bf{\Sigma}}_n)
\end{align*}
on $J_n$. Note that 
\begin{align*}
    \partial_{\sigma^{(i)}}{\bf{H}}_n({\bf{\Sigma}})&=-\frac{N_n}{2}\tr \bigl\{
    ({\bf{\Sigma}}^{-1})(\partial_{\sigma^{(i)}}{\bf{\Sigma}})
    \bigr\}+\frac{N_n}{2}\tr\bigl\{({\bf{\Sigma}}^{-1})(\partial_{\sigma^{(i)}}{\bf{\Sigma}})({\bf{\Sigma}}^{-1})\hat{\bf{\Sigma}}_n\bigr\}
\end{align*}
and
\begin{align*}
    \partial_{\sigma^{(i)}}\partial_{\sigma^{(j)}}{\bf{H}}_n({\bf{\Sigma}})&=\frac{N_n}{2}\tr \bigl\{
    ({\bf{\Sigma}}^{-1})(\partial_{\sigma^{(i)}}{\bf{\Sigma}})({\bf{\Sigma}}^{-1})(\partial_{\sigma^{(j)}}{\bf{\Sigma}})
    \bigr\}-\frac{N_n}{2}\tr \bigl\{
    ({\bf{\Sigma}}^{-1})(\partial_{\sigma^{(i)}}\partial_{\sigma^{(j)}}{\bf{\Sigma}})
    \bigr\}\\
    &\quad-\frac{N_n}{2}\tr\bigl\{({\bf{\Sigma}}^{-1})(\partial_{\sigma^{(i)}}{\bf{\Sigma}})({\bf{\Sigma}}^{-1})(\partial_{\sigma^{(j)}}{\bf{\Sigma}})({\bf{\Sigma}}^{-1})\hat{\bf{\Sigma}}_n\bigr\}\\
    &\quad+\frac{N_n}{2}\tr\bigl\{({\bf{\Sigma}}^{-1})(\partial_{\sigma^{(i)}}\partial_{\sigma^{(j)}}{\bf{\Sigma}})({\bf{\Sigma}}^{-1})\hat{\bf{\Sigma}}_n\bigr\}\\
    &\quad-\frac{N_n}{2}\tr\bigl\{({\bf{\Sigma}}^{-1})(\partial_{\sigma^{(j)}}{\bf{\Sigma}})({\bf{\Sigma}}^{-1})(\partial_{\sigma^{(i)}}{\bf{\Sigma}})({\bf{\Sigma}}^{-1})\hat{\bf{\Sigma}}_n\bigr\}
\end{align*}
for $i,j=1,\ldots,\bar{p}$, where $\sigma=\vech{\bf{\Sigma}}$. Using Taylor's Theorem, one gets
\begin{align*}
   {\bf{H}}_{n}({\bf{\Sigma}}(\hat{\theta}_n))&={{\bf{H}}}_{n}(\hat{\bf{\Sigma}}_n)+\partial_{\sigma} {\bf{H}}_{n}(\hat{\bf{\Sigma}}_n)^{\top}(\vech {\bf{\Sigma}}(\hat{\theta}_n)-\vech \hat{\bf{\Sigma}}_n)\\
    &\quad-\frac{1}{2}\sqrt{n}(\vech {\bf{\Sigma}}(\hat{\theta}_n)-\vech \hat{\bf{\Sigma}}_n)^{\top}{\bf{R}}_n\sqrt{n}(\vech {\bf{\Sigma}}(\hat{\theta}_n)-\vech \hat{\bf{\Sigma}}_n)
\end{align*}
on $J_n$, where $\check{\bf{\Sigma}}_{n,\lambda}=\hat{\bf{\Sigma}}_n+\lambda({\bf{\Sigma}}(\hat{\theta}_n)-\hat{\bf{\Sigma}}_n)$ and
\begin{align*}
    {\bf{R}}_n=-\frac{2}{n}\int_0^1 (1-\lambda)\partial^2_{\sigma} {\bf{H}}_{n}(\check{\bf{\Sigma}}_{n,\lambda})d\lambda.
\end{align*}
Since ${\bf{H}}_n({\bf{\Sigma}})$ has a maximum value at ${\bf{\Sigma}}=\hat{\bf{\Sigma}}_n$ on $J_n$, we have
\begin{align*}
    \partial_{\sigma} {\bf{H}}_{n}(\hat{\bf{\Sigma}}_n)=0
\end{align*}
on $J_n$, which yields
\begin{align}
    {\bf{T}}_n&=\sqrt{n}(\vech {\bf{\Sigma}}(\hat{\theta}_n)-\vech \hat{\bf{\Sigma}}_n)^{\top}{\bf{R}}_n\sqrt{n}(\vech {\bf{\Sigma}}(\hat{\theta}_n)-\vech \hat{\bf{\Sigma}}_n) \label{T}
\end{align}
on $J_n$. Note that
\begin{align*}
    \sqrt{n}(\vech {\bf{\Sigma}}(\hat{\theta}_n)-\vech \hat{\bf{\Sigma}}_n)^{\top}\tilde{\bf{R}}_n\sqrt{n}(\vech {\bf{\Sigma}}(\hat{\theta}_n)-\vech \hat{\bf{\Sigma}}_n)
\end{align*}
is also well-defined on $J_n^c$, where 
\begin{align*}
    \tilde{\bf{R}}_n=\left\{
    \begin{array}{ll}
    {\bf{R}}_n & (on\ J_n),\\
    \mathbb{I}_{\bar{p}} & (on\ J_n^c).
    \end{array}\right.  
\end{align*}
Fix $\varepsilon>0$ arbitrarily. Since $\partial^2_{\sigma} {\bf{H}}({\bf{\Sigma}})$ is continuous in ${\bf{\Sigma}}$, we can take a sufficient small $\delta>0$ such that
\begin{align}
    \sup_{{\bf{\Sigma}}\in B_0(\delta)}\bigl|\partial^2_{\sigma} {\bf{H}}({\bf{\Sigma}})-\partial^2_{\sigma} {\bf{H}}({\bf{\Sigma}}(\theta_0))\bigr|<\frac{\varepsilon}{2},
    \label{supSH}
\end{align}
where 
\begin{align*}
    B_{0}(\delta)=\bigl\{{\bf{\Sigma}}\in\mathcal{M}_p^{+}: |{\bf{\Sigma}}-{\bf{\Sigma}}(\theta_0)|\leq \delta  \bigr\}
\end{align*}
and
\begin{align*}
    \partial_{\sigma^{(i)}}\partial_{\sigma^{(j)}}{\bf{H}}({\bf{\Sigma}})&=\frac{1}{2}\tr \bigl\{
    ({\bf{\Sigma}}^{-1})(\partial_{\sigma^{(i)}}{\bf{\Sigma}})({\bf{\Sigma}}^{-1})(\partial_{\sigma^{(j)}}{\bf{\Sigma}})
    \bigr\}-\frac{1}{2}\tr \bigl\{
    ({\bf{\Sigma}}^{-1})(\partial_{\sigma^{(i)}}\partial_{\sigma^{(j)}}{\bf{\Sigma}})
    \bigr\}\\
    &\quad-\frac{1}{2}\tr\bigl\{({\bf{\Sigma}}^{-1})(\partial_{\sigma^{(i)}}{\bf{\Sigma}})({\bf{\Sigma}}^{-1})(\partial_{\sigma^{(j)}}{\bf{\Sigma}})({\bf{\Sigma}}^{-1}){\bf{\Sigma}}(\theta_0)\bigr\}\\
    &\quad+\frac{1}{2}\tr\bigl\{({\bf{\Sigma}}^{-1})(\partial_{\sigma^{(i)}}\partial_{\sigma^{(j)}}{\bf{\Sigma}})({\bf{\Sigma}}^{-1}){\bf{\Sigma}}(\theta_0)\bigr\}\\
    &\quad-\frac{1}{2}\tr\bigl\{({\bf{\Sigma}}^{-1})(\partial_{\sigma^{(j)}}{\bf{\Sigma}})({\bf{\Sigma}}^{-1})(\partial_{\sigma^{(i)}}{\bf{\Sigma}})({\bf{\Sigma}}^{-1}){\bf{\Sigma}}(\theta_0)\bigr\}.
\end{align*}
In an analogous manner to Proposition \ref{supHprop}, it is shown that
\begin{align}
    \sup_{{\bf{\Sigma}}\in B_{0}(\delta)}\Bigl|\frac{1}{n}\partial^2_{\sigma}{\bf{H}}_n({\bf{\Sigma}})-\partial^2_{\sigma}{\bf{H}}({\bf{\Sigma}})\Bigr|\stackrel{p}{\longrightarrow}0 \label{supSHn}
\end{align}
under $H_0$. Since
\begin{align*}
    \tr\bigl\{({\bf{\Sigma}}^{-1})(\partial_{\sigma^{(i)}}{\bf{\Sigma}})({\bf{\Sigma}}^{-1})(\partial_{\sigma^{(j)}}{\bf{\Sigma}})\bigr\}&=\vec(\partial_{\sigma^{(i)}}{\bf{\Sigma}})^{\top}\bigl({\bf{\Sigma}}^{-1}\otimes{\bf{\Sigma}}^{-1}\bigr)
    \vec(\partial_{\sigma^{(j)}}{\bf{\Sigma}})\\
    &=\bigl\{\partial_{\sigma^{(i)}}\vech{\bf{\Sigma}}\bigr\}
    ^{\top}\mathbb{D}_p^{\top}\bigl({\bf{\Sigma}}^{-1}\otimes{\bf{\Sigma}}^{-1}\bigr)\mathbb{D}_p\bigl\{\partial_{\sigma^{(j)}}\vech{\bf{\Sigma}}\bigr\}\\
    &=\bigl\{\partial_{\sigma^{(i)}}\sigma\bigr\}
    ^{\top}\bigl(\mathbb{D}_p^{+}(
    {\bf{\Sigma}}\otimes{\bf{\Sigma}})\mathbb{D}_p^{+\top}\bigr)^{-1}\bigl\{\partial_{\sigma^{(j)}}\sigma\bigr\}\\
    &=\bigl(\bigl(\mathbb{D}_p^{+}(
    {\bf{\Sigma}}\otimes{\bf{\Sigma}})\mathbb{D}_p^{+\top}\bigr)^{-1}\bigr)_{ij},
\end{align*}
one has
\begin{align*}
    \partial_{\sigma^{(i)}}\partial_{\sigma^{(j)}}{\bf{H}}({\bf{\Sigma}}(\theta_0))&=-\frac{1}{2}\tr\bigl\{({\bf{\Sigma}}(\theta_0)^{-1})(\partial_{\sigma^{(i)}}{\bf{\Sigma}}(\theta_0))({\bf{\Sigma}}(\theta_0)^{-1})(\partial_{\sigma^{(j)}}{\bf{\Sigma}}(\theta_0))\bigr\}\\
    &=-\bigl(\bigl(2\mathbb{D}_p^{+}(
    {\bf{\Sigma}}(\theta_0)\otimes{\bf{\Sigma}}(\theta_0))\mathbb{D}_p^{+\top}\bigr)^{-1}\bigr)_{ij}=-({\bf{W}}_0^{-1})_{ij},
\end{align*}
which yields
\begin{align}
    \partial^2_{\sigma} {\bf{H}}({\bf{\Sigma}}(\theta_0))=-{\bf{W}}_0^{-1}. \label{HW}
\end{align}
Furthermore, it holds that
\begin{align*}
    |\check{\bf{\Sigma}}_{n,\lambda}-{\bf{\Sigma}}(\theta_0)|&\leq |\hat{\bf{\Sigma}}_n-{\bf{\Sigma}}(\theta_0)|+\lambda|{\bf{\Sigma}}(\hat{\theta}_n)-\hat{\bf{\Sigma}}_n|\\
    &\leq |\hat{\bf{\Sigma}}_n-{\bf{\Sigma}}(\theta_0)|+|{\bf{\Sigma}}(\hat{\theta}_n)-\hat{\bf{\Sigma}}_n|\\
    &\leq  |\hat{\bf{\Sigma}}_n-{\bf{\Sigma}}(\theta_0)|+
    |{\bf{\Sigma}}(\hat{\theta}_n)-{\bf{\Sigma}}(\theta_0)|+
    |{\bf{\Sigma}}(\theta_0)-\hat{\bf{\Sigma}}_n|\leq\delta
\end{align*}
for $\lambda\in[0,1]$ on $C_n$, where
\begin{align*}
    C_n=\Bigl\{|\hat{\bf{\Sigma}}_n-{\bf{\Sigma}}(\theta_0)|\leq\frac{\delta}{3}\Bigr\}\cap \Bigl\{|{\bf{\Sigma}}(\hat{\theta}_n)-{\bf{\Sigma}}(\theta_0)|\leq\frac{\delta}{3}\Bigr\}.
\end{align*}
Thus, we see from (\ref{HW}) that
\begin{align*}
    \bigl|\tilde{\bf{R}}_n-{\bf{W}}_0^{-1}\bigr|&=\biggl|2\int_0^1 (1-\lambda)\Bigl\{\frac{1}{n}\partial^2_{\sigma} {\bf{H}}_{n}(\check{\bf{\Sigma}}_{n,\lambda})-\partial^2_{\sigma} {\bf{H}}({\bf{\Sigma}}(\theta_0))\Bigr\}d\lambda\biggr|\\
    &\leq 2\int_0^1 (1-\lambda)\Bigl|\frac{1}{n}\partial^2_{\sigma} {\bf{H}}_{n}(\check{\bf{\Sigma}}_{n,\lambda})-\partial^2_{\sigma} {\bf{H}}({\bf{\Sigma}}(\theta_0))\Bigr|d\lambda\\
    &\leq \sup_{{\bf{\Sigma}}\in B_0(\delta)}\Bigl|\frac{1}{n}\partial^2_{\sigma} {\bf{H}}_{n}({\bf{\Sigma}})-\partial^2_{\sigma} {\bf{H}}({\bf{\Sigma}}(\theta_0))\Bigr|\\
    &\leq \sup_{{\bf{\Sigma}}\in B_0(\delta)}\Bigl|\frac{1}{n}\partial^2_{\sigma} {\bf{H}}_{n}({\bf{\Sigma}})-\partial^2_{\sigma} {\bf{H}}({\bf{\Sigma}})\Bigr|
    +\sup_{{\bf{\Sigma}}\in B_0(\delta)}\bigl|\partial^2_{\sigma} {\bf{H}}({\bf{\Sigma}})-\partial^2_{\sigma} {\bf{H}}({\bf{\Sigma}}(\theta_0))\bigr|
\end{align*}
on $C_n\cap J_n$. Since ${\bf{\Sigma}}(\theta_0)$ is positive, Theorems \ref{Qtheorem}-\ref{thetatheorem} imply 
\begin{align*}
    {\bf{P}}\bigl(J_n\bigr)\longrightarrow 1,\quad {\bf{P}}\bigl(C_n\bigr)\longrightarrow 1
\end{align*}
under $H_0$. Consequently, it follows from (\ref{supSH})-(\ref{supSHn}) that
\begin{align*}
   &\quad\ {\bf{P}}\Bigl(\bigl|\tilde{\bf{R}}_n-{\bf{W}}_0^{-1}\bigr|>\varepsilon\Bigr)\\
   &={\bf{P}}\Bigl(\Bigl\{\bigl|\tilde{\bf{R}}_n-{\bf{W}}_0^{-1}\bigr|>\varepsilon\Bigr\}\cap\bigl(C_n\cap J_n\bigr)\Bigr)
   +{\bf{P}}\Bigl(\Bigl\{\bigl|\tilde{\bf{R}}_n-{\bf{W}}_0^{-1}\bigr|>\varepsilon\Bigr\}\cap\bigl(C_n\cap J_n\bigr)^c\Bigr)\\
   &\leq {\bf{P}}\Biggl(\biggl\{\sup_{{\bf{\Sigma}}\in B_0(\delta)}\Bigl|\frac{1}{n}\partial^2_{\sigma} {\bf{H}}_{n}({\bf{\Sigma}})-\partial^2_{\sigma} {\bf{H}}({\bf{\Sigma}})\Bigr|\\
    &\qquad\qquad\qquad+\sup_{{\bf{\Sigma}}\in B_0(\delta)}\bigl|\partial^2_{\sigma} {\bf{H}}({\bf{\Sigma}})-\partial^2_{\sigma} {\bf{H}}({\bf{\Sigma}}(\theta_0))\bigr|>\varepsilon\biggr\}\cap \bigl(C_n\cap J_n\bigr)\Biggr)+ {\bf{P}}\Bigl(C_n^c\cup J_n^c\Bigr)\\
    &\leq {\bf{P}}\biggl(\sup_{{\bf{\Sigma}}\in B_0(\delta)}\Bigl|\frac{1}{n}\partial^2_{\sigma} {\bf{H}}_{n}({\bf{\Sigma}})-\partial^2_{\sigma} {\bf{H}}({\bf{\Sigma}})\Bigr|
    +\sup_{{\bf{\Sigma}}\in B_0(\delta)}\bigl|\partial^2_{\sigma} {\bf{H}}({\bf{\Sigma}})-\partial^2_{\sigma} {\bf{H}}({\bf{\Sigma}}(\theta_0))\bigr|>\varepsilon\biggr)+{\bf{P}}\Bigl(C_n^c\cup J_n^c\Bigr)\\
    &\leq {\bf{P}}\biggl(\sup_{{\bf{\Sigma}}\in B_0(\delta)}\Bigl|\frac{1}{n}\partial^2_{\sigma} {\bf{H}}_{n}({\bf{\Sigma}})-\partial^2_{\sigma} {\bf{H}}({\bf{\Sigma}})\Bigr|>\frac{\varepsilon}{2}\biggr)\\
    &\qquad+{\bf{P}}\biggl(\sup_{{\bf{\Sigma}}\in B_0(\delta)}\bigl|\partial^2_{\sigma} {\bf{H}}({\bf{\Sigma}})-\partial^2_{\sigma} {\bf{H}}({\bf{\Sigma}}(\theta_0))\bigr|>\frac{\varepsilon}{2}\biggr)+{\bf{P}}\bigl(C_n^c\bigr)+{\bf{P}}\bigl(J_n^c\bigr)\longrightarrow 0
\end{align*}
under $H_0$, so that one gets
\begin{align}
    \tilde{\bf{R}}_n\stackrel{p}{\longrightarrow} {\bf{W}}_0^{-1} \label{tildeR}
\end{align}
under $H_0$. As it holds from Lemma \ref{testlemma} that
\begin{align*}
    &\quad\ \sqrt{n}(\vech {\bf{\Sigma}}(\hat{\theta}_n)-\vech \hat{\bf{\Sigma}}_n)\\
    &=\sqrt{n}(\vech {\bf{\Sigma}}(\hat{\theta}_n)-\vech {\bf{\Sigma}}(\theta_0))-\sqrt{n}(\vech \hat{\bf{\Sigma}}_n-\vech {\bf{\Sigma}}(\theta_0))\\
    &=\bigl\{\Delta_0(\Delta_0^{\top}{\bf{W}}_0^{-1}\Delta_0)^{-1}\Delta_0^{\top}{\bf{W}}_0^{-1}-\mathbb{I}_{\bar{p}}\bigr\}\sqrt{n}(\vech \hat{\bf{\Sigma}}_n-\vech {\bf{\Sigma}}(\theta_0))+o_p(1)
\end{align*}
under $H_0$, Theorem \ref{Qtheorem} and (\ref{tildeR}) imply
\begin{align*}
    &\quad\ \sqrt{n}(\vech {\bf{\Sigma}}(\hat{\theta}_n)-\vech \hat{\bf{\Sigma}}_n)^{\top}\tilde{\bf{R}}_n\sqrt{n}(\vech {\bf{\Sigma}}(\hat{\theta}_n)-\vech \hat{\bf{\Sigma}}_n)\\
    &=\sqrt{n}(\vech \hat{\bf{\Sigma}}_n-\vech {\bf{\Sigma}}(\theta_0))^{\top}({\bf{W}}_0^{-1}-{\bf{H}}_0)\sqrt{n}(\vech \hat{\bf{\Sigma}}_n-\vech {\bf{\Sigma}}(\theta_0))+o_p(1)\\
    &\stackrel{d}{\longrightarrow}
    Z_{\bar{p}}^{\top}{\bf{P}}_0Z_{\bar{p}},
\end{align*}
where
\begin{align*}
    {\bf{H}}_0={\bf{W}}_0^{-1}\Delta_0({\bf{\Delta}}_0^{\top}{\bf{W}}_0^{-1}\Delta_0)^{-1}\Delta_0^{\top}{\bf{W}}_0^{-1},\quad 
    {\bf{P}}_0={\bf{W}}_0^{\frac{1}{2}}({\bf{W}}_0^{-1}-{\bf{H}}_0){\bf{W}}_0^{\frac{1}{2}}.
\end{align*}
On the other hand, ${\bf{P}}_0$ is an orthogonal projection matrix, and {\bf{[B1]}} (b) yields 
\begin{align*}
    \rank {\bf{P}}_0=\bar{p}-q,
\end{align*}
so that it is shown that
\begin{align*}
    Z_{\bar{p}}^{\top}{\bf{P}}_0Z_{\bar{p}}\sim \chi^2_{\bar{p}-q}.
\end{align*}
Therefore, for any closed set $C\in\mathbb{R}$, we see from (\ref{T}) that
\begin{align*}
    &\quad\ \limsup_{n\longrightarrow\infty}{\bf{P}}\bigl({\bf{T}}_n\in C\bigr)\\
    &=\limsup_{n\longrightarrow\infty}{\bf{P}}\Bigl(\bigl\{
    \sqrt{n}(\vech {\bf{\Sigma}}(\hat{\theta}_n)-\vech \hat{\bf{\Sigma}}_n)^{\top}{\bf{R}}_n\sqrt{n}(\vech {\bf{\Sigma}}(\hat{\theta}_n)-\vech \hat{\bf{\Sigma}}_n)
    \in C\bigr\}\cap J_n\Bigr)\\
    &\qquad+\limsup_{n\longrightarrow\infty}{\bf{P}}\Bigl(\bigl\{{\bf{T}}_n\in C\bigr\}\cap J_n^c\Bigr)\\
    &\leq\limsup_{n\longrightarrow\infty}{\bf{P}}\Bigl(
    \sqrt{n}(\vech {\bf{\Sigma}}(\hat{\theta}_n)-\vech \hat{\bf{\Sigma}}_n)^{\top}\tilde{\bf{R}}_n\sqrt{n}(\vech {\bf{\Sigma}}(\hat{\theta}_n)-\vech \hat{\bf{\Sigma}}_n)
    \in C\Bigr)\\
    &\qquad+\limsup_{n\longrightarrow\infty}{\bf{P}}\bigl(J_n^c\bigr)
    \leq {\bf{P}}\bigl(\chi^2_{\bar{p}-q}\in C\bigr)
\end{align*}
under $H_0$, which completes the proof.
\end{proof}
\begin{proposition}\label{supH1prop}
Suppose that {\bf{[A1]}}-{\bf{[A4]}} hold. Then, as $n\longrightarrow\infty$,
\begin{align*}
    \sup_{\theta\in\Theta}\Bigl|\frac{1}{n}{\bf{T}}_n(\theta)-{\bf{U}}(\theta)\Bigr|\stackrel{p}{\longrightarrow} 0
\end{align*}
under $H_1$.
\end{proposition}
\begin{proof}
    In an analogous manner to Proposition \ref{supHprop}, we can show the result.
\end{proof}
\begin{proposition}\label{H1prop}
Suppose that {\bf{[A1]}}-{\bf{[A4]}} and {\bf{[B2]}} hold. Then, as $n\longrightarrow\infty$, 
\begin{align*}
    \hat{\theta}_n\stackrel{p}{\longrightarrow}\theta^*
\end{align*}
under $H_1$.
\end{proposition}
\begin{proof}
     In a similar way to Theorem \ref{thetatheorem}, the result can be proven.
\end{proof}
\begin{proof}[\bf{Proof of Theorem \ref{testtheorem2}}]
By the continuous mapping theorem, it holds from Proposition \ref{H1prop} that
\begin{align*}
    {\bf{U}}(\hat{\theta}_n)\stackrel{p}{\longrightarrow}{\bf{U}}(\theta^*)
\end{align*}
under $H_1$, so that Proposition \ref{supH1prop} shows
\begin{align*}
    \Bigl|\frac{1}{n}{\bf{T}}_n(\hat{\theta}_n)-{\bf{U}}(\theta^*)\Bigr|&\leq \Bigl|\frac{1}{n}{\bf{T}}_n(\hat{\theta}_n)-{\bf{U}}(\hat{\theta}_n)\Bigr|+\bigl|{\bf{U}}(\hat{\theta}_n)-{\bf{U}}(\theta^*)\bigr|\\
    &\leq\sup_{\theta\in\Theta}\Bigl|\frac{1}{n}{\bf{T}}_n(\theta)-{\bf{U}}(\theta)\Bigr|+\bigl|{\bf{U}}(\hat{\theta}_n)-{\bf{U}}(\theta^*)\bigr|\stackrel{p}{\longrightarrow}0
\end{align*}
under $H_1$, which yields
\begin{align*}
    \frac{1}{n}{\bf{T}}_n(\hat{\theta}_n)\stackrel{p}{\longrightarrow} {\bf{U}}(\theta^*)
\end{align*}
under $H_1$. Since ${\bf{U}}(\theta^*)>0$ under $H_1$, we obtain
\begin{align*}
    {\bf{P}}\Bigl({\bf{T}}_n(\hat{\theta}_n)>\chi^2_{\bar{p}-q}\Bigr)=
    1-{\bf{P}}\biggl(\frac{1}{n}{\bf{T}}_n(\hat{\theta}_n)\leq\frac{\chi^2_{\bar{p}-q}}{n}\biggr) \longrightarrow 1
\end{align*}
under $H_1$, which completes the proof.
\end{proof}

\end{document}